\documentclass[10pt]{article}
\usepackage{graphicx}
\usepackage{amsfonts,amssymb,latexsym,amsmath,amsthm}

\newtheorem{theorem}{Theorem}[section]
\newtheorem{definition}[theorem]{Definition}

\newtheorem{remark}[theorem]{Remark}

\newtheorem{example}[theorem]{Example}
\numberwithin{equation}{section}
\usepackage[symbol]{footmisc}

\title{Application of Zernike polynomials in solving certain first and second order partial differential equations}
\author{K. B. Datta$^{a}$ and S. Datta$^{b,*}$\\
$^a$ Department of Electrical Engineering, IIT, Kharagpur-721302, India\\
$^b$ Department of Mathematics \& Statistical Science, University of Idaho, ID 83844, USA
}
\date{}
\begin{document}
%\date{} %this command will not print the date in the title
\maketitle

\footnote{Corresponding author
E-mail address: sdatta@uidaho.edu
}

\vspace{-5mm}

\begin{abstract}
{\noindent
Integration operational matrix methods based on Zernike polynomials are used to determine approximate solutions of a class of non-homogeneous partial differential equations (PDEs) of first and second order. Due to the nature of the Zernike polynomials being described in the unit disk, this method is particularly effective in solving PDEs over a circular region. Further, the
proposed method can solve PDEs with discontinuous Dirichlet and Neumann boundary conditions,
and as these discontinuous functions cannot be defined at some of the Chebyshev or Gauss-Lobatto
points, the much acclaimed pseudo-spectral methods are not directly applicable to such problems.
Solving such PDEs is also a new application of Zernike polynomials as so far the main application of these polynomials seem to have been in the study of optical aberrations of circularly symmetric optical systems. In the present method, the given PDE is converted to a system of linear equations of the form $A\mathbf{x} = \mathbf{b}$ which may be solved by both $l_1$ and $l_2$ minimization methods among which the $l_1$ method is found to be more accurate. Finally, in the expansion of a function in terms of Zernike polynomials, the rate of decay of the coefficients is given for certain classes of functions.}
\end{abstract}

\\
%numerical computation of PDEs; computational mathematics;  integration operational matrix; Laplace equation; Zernike polynomials; approximation in numerical methods.

\noindent \textbf{2020 Math Subject Classification:} 35AXX (Primary), 41A10(Secondary)

%%%%%%%%%%%%%%%%%%%%%%%%%%%%%%%%%%%%%%%%%%%%%%%%%%%%%
\section{Introduction}
\label{sec:Intro}
%%%%%%%%%%%%%%%%%%%%%%%%%%%%%%%%%%%%%%%%%%%%%%%%%%%%%
\subsection{Background}
\label{Background}
%%%%%%%%%%%%%%%%%%%%%%%%%%%%%%%%%%%%%%%%%%%%%%%%%%%%%
If the analytical solution of a partial differential equation (PDE) with a forcing function and given boundary conditions is difficult, we go for numerical methods as discussed in \cite{Boyd:2011}, \cite{Matsushima:1995}, and \cite{Morton:2005}. In many practical cases such as fluid flow in a rotating cylinder, electromagnetic equations in cylindrical waveguides and optical lens design, and many other cases of a cylindrical or a spherical boundary with axial symmetry, one needs to solve PDEs over a disk instead of a rectangular boundary \cite{Boyd:2011, Boyd:2014, Matsushima:1995}.
\par
One of the important numerical methods to solve PDEs is using integration operational matrices (IOMs), first introduced in \cite{{Corrington:1973}} who showed that integral and differential equations could be reduced to a set of linear algebraic equations with an approximation in the sense of least-squares by taking an orthonormal set of Walsh functions.
This approach was subsequently applied to solve PDEs in \textit{rectangular} regions using piecewise constant orthogonal functions (PCOF) and orthogonal polynomials (OP), comprehensive accounts of which are available in \cite{Rao:1983} and \cite{Datta:1995}. A set of 2-D orthogonal functions known as Zernike polynomials and defined on the unit disk is used in the analysis and evaluation of optical systems with circular pupils by expanding optical wavefront functions in series of these polynomials \cite{Zernike:1934, Nijboer:1942,Born:1999}. It appears that these polynomials have never been used in the analysis of PDEs and in this paper they are fruitfully employed to solve PDEs on a disk.

With the above motivation, the main contribution of this paper is a new method of solving PDEs in circular regions with discontinuous Dirichlet and Neumann boundary conditions using Zernike polynomials and IOMs. In practice, to get an IOM, a pre-selected set of orthogonal functions (OFs) is first integrated analytically. The result of integration is then expressed in terms of a fixed finite number of functions in the original set of OFs. This gives an approximation of the integration operator.  On the other hand, for a set of OPs, a three-term derivative recurrence relation is available, which on integration allows one to express any orthogonal function in the set in terms of the original set of OPs. For the radial parts of Zernike  polynomials, neither a three-term recurrence relation is available, nor on integrating any radial polynomial in the set can it be expressed automatically in terms of other polynomials in the set. However, this difficulty can be obviated by using a derivative relation of the radial parts of Zernike  polynomials given in \cite{Noll:1976}, different from the three-term derivative recurrence relation of OPs, which needs a trivial matrix inversion to get the IOM of the radial parts of Zernike  polynomials.

\par
If a known function $u(r,\theta)$  is approximated by the Zernike polynomials $R^m_n(r) e^{i m \theta} $ of order $(m,n)$, $m$ being the azimuthal frequency and $n$ the degree of the radial polynomial $R^m_n(r)$, then the radial polynomials of degree higher than $n$ are neglected. In the method proposed here, to obtain the approximate solution of a PDE, these higher order polynomials are approximated by lower order polynomials with some reliable interpolation formula such as Lagrange's, see Remark~\ref{Rem:Approximation}, equations (\ref{SOPDE12}) and  (\ref{ApproxLagInterpolation}). If this projection of the neglected polynomials on the space generated by the lower order polynomials is not done then our IOM method of solving PDEs using  Zernike polynomials may fail. In \cite{Rao:1983}, the author used two dimensional block pulse functions to solve second order PDEs. However, the approximate solutions did not converge, and this may be attributed to the fact that the above mentioned idea of projecting higher order terms was not considered.

\par
Using IOMs, a given PDE is reduced to a system of linear equations of the form $A\mathbf{x} = \mathbf{b},$ where $A$ is a sparse matrix. This must then be solved to obtain an approximate solution of the PDE. This is usually solved with least squares approximation by using standard matrix pseudo-inverse or Moore-Penrose inverse and is called $l_2$ method. An alternative method to solve such a sparse system is an $l_1$-minimization algorithm developed in \cite{Romberg:2006} which is used in this paper and found to be more accurate than the least squares solution.

As a comparison with some of the existing methods, it may be noted that pseudo-spectral collocation methods seem unsuited for solving PDEs with discontinuous boundary conditions. In any of the pseudo-spectral collocation methods, the Chebyshev points:
$$
\tfrac 12 (\cos \tfrac{i\pi}{M}+1),\; i=0,1,\dots,M
$$
or the Gauss-Lobatto points in the interval [0,1] are chosen so as to minimize Runge phenomenon, and the boundary conditions (BCs) must be defined at these discrete points. In the problem that we have considered here in Example~\ref{EX:SOPDE} of Section~\ref{Sec:SecondOrderPDE}, this would imply that the mixed BCs are discontinuous for $i=M$. Therefore, the given boundary conditions will not be defined at one of the collocation points.
The method proposed here can naturally take care of such discontinuous BCs although Gibbs-Wilbraham phenomenon will still appear. This is expected when using Fourier series, and cosine and sine functions are part of the structure
of Zernike polynomials.

One of the primary concerns in numerically solving PDEs is the convergence of the solutions.
The ingenuous basis functions developed by Livermore et al. in \cite{Livermore:2007} to apply Galerkin's method on the disk and sphere behave asymptotically as Jacobi polynomials for large degree,
and so their convergence rates are similar to the latter polynomials. In the polynomial representation of a function $f(x)$ of one variable, where the domain is an interval, it is generally true that Chebyshev polynomials give the fastest rate of convergence from the larger family of Jacobi polynomials except when $f(x)$ is singular at one or both end points. In this setting, all Gegenbauer polynomials (including Legendre and Chebyshev) converge equally fast at the endpoints, but Gegenbauer polynomials converge more rapidly on the interior with increasing order of the degree $m$. However, for functions on the unit disk, Zernike polynomials are superior in terms of rate of convergence when compared to Chebyshev-Fourier series \cite{Boyd:2014}. So, in terms of the rate of convergence, the performance of the method proposed in this paper using the IOM of Zernike polynomials has the same superiority as discussed in \cite{Boyd:2014}. Moreover, it is also shown in this paper (see Theorem~\ref{Thm:Decay}) that for functions that are H\"{o}lder  continuous of order $\lambda,$ the coefficients $C_{nm}$ in the expansion of a function $u(r,\theta)$ in terms of Zernike polynomials $R_n^m(r) e^{im\theta}$ decay at least like $m^{-\lambda+1},\, \lambda \ge 1$. Using methods similar to the decay of Fourier coefficients for functions of a single variable, a similar result (Theorem~\ref{THM:RateDecayCk}) is given for functions that are $k$ times continuously differentiable.
\\
Some preliminary ideas underlying the Zernike polynomials
%Zernike circle polynomials
are given next.

\subsection{Preliminaries and Notation }
\label{sec:Notation}
%%%%%%%%%%%%%%%%%%%%%%%%%%%%%%%%%%%%%%%%%%
\emph{Zernike polynomials} are used to conveniently expand optical wavefront functions that arise in optical systems with circular pupils \cite{Zernike:1934, Born:1999, Lakshminarayanan:2011}. Proposed by F. von Zernike
%von. F. Zernike
in \cite{Zernike:1934}, these polynomials are orthogonal on the unit disk $B(0,1) = \{(x, y) \in \mathbb{R}^2 : x^2 + y^2 \leq 1\}$ and can be found in the following way \cite{Zernike:1934, Nijboer:1942, Born:1999}.
To start with, one considers a partial differential equation that is invariant under rotations of the %co-ordinate
coordinate axes about the origin. Such an equation has the form:
 	\begin{equation} \label{eq:PDEforcirclepolynomials}
 	\bigtriangleup U + \alpha \left(x \frac{\partial}{\partial x} + y \frac{\partial}{\partial y}\right)^2 U + \beta \left(x \frac{\partial}{\partial x} + y \frac{\partial}{\partial y}\right) U + \gamma U = 0.
 	\end{equation}	
In polar coordinates $r$ and $\phi$, by using the transformation: $x=r \cos \phi$ and $y=r\sin \phi$,  equation (\ref{eq:PDEforcirclepolynomials}) becomes
	\begin{equation} \label{eq:PDEforcirclepolynomialsPolar}
	(1 + \alpha r^2) \frac{\partial^2 U}{\partial r^2} + \left(\frac 1r + (\alpha + \beta)r \right) \frac{\partial U}{\partial r} + \frac{1}{r^2} \frac{\partial^2 U}{\partial \phi^2} + \gamma U = 0.
	\end{equation}
Choosing $\alpha = -1,$ $\beta = -1,$ and $\gamma = n(n+2)$ gives the hypergeometric equation
	\begin{equation} \label{eq:HypergeometricRadialFunction}
	x(1 - x) \frac{d^2 y}{dx^2} + (1 - 2x) \frac{dy}{dx} + \frac 14 \left[n(n+2) - \frac{m^2}{x}\right] y = 0.
 	\end{equation}
The solution of (\ref{eq:HypergeometricRadialFunction}) denoted by $R_n^m(r)$ is known as the \textit{radial part of a Zernike  polynomial} and is given by (\ref{Sol:Hypergeometric}) below when $n$, $m$ are non-negative integers, and $n - m$ is even and non-negative, see %\cite{Nijboer:1947}
\cite{Nijboer:1942}, \cite{Prata:1989} and \cite{Born:1999}:
	\begin{equation} \label{Sol:Hypergeometric}
	R_n^m(r) = \sum_{\ell = 0}^{\frac{n - m}{2}} (-1)^{\ell} \frac{(n - \ell)!}{\ell! (\frac{n - m}{2} - \ell)! (\frac{n + m}{2} - \ell)!} r^{n - 2\ell}.
	\end{equation}
The radial part of Zernike polynomials can be expressed in terms of classical Jacobi polynomials defined on the interval $[0, 1],$ as outlined in \cite{Weisstein} and \cite{Dunkl:2014}. The classical Jacobi polynomials satisfy a three term recurrence relation, a second order differential equation as (1.3), and interesting properties such as (2.9), see \cite{Abramowitz:1972} and \cite{Chihara:1978}. A method of computing the radial parts of Zernike polynomials
%Zernike circle polynomials
of arbitrary degree using the discrete Fourier transform has been discussed in \cite{Janssen:2007}. In \cite{Shakibeai:2013}, a recurrence relation that depends neither on the degree nor on the azimuthal order of the radial polynomials is developed.
The %Zernike circle polynomials
Zernike polynomials are solutions to (\ref{eq:PDEforcirclepolynomials}) or (\ref{eq:PDEforcirclepolynomialsPolar}), and are given by
	\begin{equation} \label{ZernikePolynomials}
    U^m_n(r, \phi) =  R_n^m(r)(C_1 \cos m \phi + C_2 \sin m \phi), \quad n \in \mathbb{N} \cup \{0\},  \ n - m \geq 0, \ n - m \ \textrm{even},
	\end{equation}
where $C_1$ and $C_2$  are arbitrary constants.

It is worth mentioning that \textit{Zernike polynomials} is the general name for a class of bivariate orthogonal polynomials on
the unit disk, and they are a particular case of orthogonal polynomials on the unit disk, see \cite{Dunkl:2014}. They are defined by  a radial part that is a univariate orthogonal Jacobi polynomial defined on the interval [0, 1], and a non-radial part that is a bivariate spherical harmonic. The %Zernike circle polynomials
Zernike polynomials form a complete orthogonal set for the interior of the unit disk $B(0,1)$, see \cite{Bhatia:1954, Nijboer:1942, Born:1999}. They can therefore be normalized to form an orthonormal basis for the space $L^2(B(0,1)),$ see Section~\ref{sec:Convergence}. Let $f(r, \phi)$ be an arbitrary function defined on $B(0,1)$. In terms of the Zernike polynomials
%Zernike circle polynomials
given in (\ref{ZernikePolynomials}), $f$ can be represented as \cite{Zernike:1934, Born:1999}
	\begin{equation} \label{RepresentationZernikePoly}
	f(r, \phi) = \sum_{n=0}^{\infty} \sum_{\substack{0 \leq m \leq n \\ n - m \ \textrm{even}}} (A_{nm} \cos m \phi + B_{nm} \sin m \phi) R_n^m(r)
	\end{equation}	
where
	\begin{equation} \label{EQ:ZernikeCoeff}
	A_{nm} = \frac{\epsilon_m(n+1)}{\pi} \int_0^1 \int_0^{2\pi} f(r, \phi) \cos m \phi R_n^m(r) r d \phi dr,
	B_{nm} = \frac{\epsilon_m(n+1)}{\pi} \int_0^1 \int_0^{2\pi} f(r, \phi) \sin m \phi R_n^m(r) r d \phi dr,
	\end{equation}
and $\epsilon_m$ is the Neumann factor:
	$$\epsilon_m = \left\{
	\begin{array}{cc}
	1 & \textrm{if} \ m = 0, \\
	2 & \textrm{otherwise.}
	\end{array}
	\right.
	$$
An efficient algorithm for calculating the coefficients $A_{nm}$ and $B_{nm}$ is discussed in \cite{Prata:1989}, see also  \cite{Hwang:2006}. An approximation of $f$ of order $(m,n)$ can then be calculated as
\begin{eqnarray*}
	f(r, \phi) = \sum_{i=0}^{n} \; \sum_{\substack{0 \leq j \leq i \\ i - j \ \textrm{even}}}^m (A_{ij} \cos j \phi + B_{ij} \sin j \phi) R_i^j(r).
\end{eqnarray*}
For details on the properties of Zernike polynomials
%Zernike circle polynomials
the reader is referred to  \cite{Nijboer:1942} and \cite{Born:1999}. In the context of opto-mechanical analysis
% of the eye, it has been studied that whereas
by finite element methods,
% cannot be used with success for such analysis,
a Zernike polynomial representation of the surface distortions is found to be better than by other orthogonal functions \cite{Genberg:2002}.

At times it will be convenient to write a product of matrices in a vector and tensor product form. In such cases, a matrix $P$ of size %$m$ by $n$
$m \times n$ is represented as a vector of size $mn,$ and written as $\textrm{vec}(P).$ The tensor product of matrices is denoted by  $\otimes.$ To clarify, let $A$, $B$, $X$, and $Y$ be $2 \times 2$ matrices and consider
	\begin{align}
	Y &= A X B \notag \\
	\textrm{or,} \quad	
	\begin{bmatrix}
	y_1 & y_2 \\
	y_3 & y_4
	\end{bmatrix}
	& =
	\begin{bmatrix}
	a_{11} & a_{12} \\
	a_{21} & a_{22}
	\end{bmatrix}
	\begin{bmatrix}
	x_1 & x_2 \\
	x_3 & x_4
	\end{bmatrix}
	\begin{bmatrix}
	b_{11} & b_{12} \\
	b_{21} & b_{22}
	\end{bmatrix}
	\nonumber \\
	\textrm{or,} \quad
	\begin{bmatrix}
	y_1 \\
	y_2 \\
	y_3 \\
	y_4
	\end{bmatrix}
	& =
	\begin{bmatrix}
	a_{11} b_{11} & a_{12} b_{11} & a_{11} b_{21} & a_{12} b_{21} \\
	a_{21} b_{11} & a_{22} b_{11} & a_{21} b_{21} & a_{22} b_{21} \\
	a_{11} b_{12} & a_{12} b_{12} & a_{11} b_{22} & a_{12} b_{22} \\
	a_{21} b_{12} & a_{22} b_{12} & a_{21} b_{22} & a_{22} b_{22}
	\end{bmatrix}
	\begin{bmatrix}
	x_1 \\
	x_2 \\
	x_3 \\
	x_4
	\end{bmatrix}
	\nonumber \\
	\implies \textrm{vec}(Y) &= (B^{\textrm{\tiny{T}}} \otimes A) \textrm{vec}(X) \label{EQ:VectorMatrixTensorProduct}
	\end{align}
%
%	The above is a well known transformation of linear systems of matrix unknowns by means of the Kronecker product studied in \cite{Horn:1991}.
The above method of transforming a linear system of matrix unknowns to a linear system involving a vector of unknowns by means of tensor product (also known as Kronecker product) is well known (Chapter 4; \cite{Horn:1991}).
%%%%%%%%%%%%%%%%%%%%%%%%%%%%%%%%%%%%%%%%%%%%%%%%%%%%%%%%%%%%%%
\subsection{Outline}
\label{sec:Outline}
%%%%%%%%%%%%%%%%%%%%%%%%%%%%%%%%%%%%%%%%%%%%%%%%%%%%%%%%%%%%%%
The remaining part of the paper is organized as follows. Section~\ref{Sec:FirstOrderPDE} discusses the solution of a first order PDE by using the IOM method with Zernike polynomials,
%Zernike circle polynomials
and the %efficacy
accuracy of the method is shown by means of a specific example. Section~\ref{Sec:SecondOrderPDE} discusses the solution of a second order PDE by using the IOM method with Zernike polynomials.
%Zernike circle polynomials.
The results of the proposed method are applied to a particular Laplace equation. Surface plots of the solutions and error estimates are provided for various orders of approximation. Section~\ref{sec:Convergence} provides some results related to the decay of the Zernike coefficients of functions that are $k$ times continuously differentiable and ones that are H\"{o}lder continuous. Some derivations in obtaining the IOMs are given in Section~\ref{Appendix}, and Section~\ref{Conclusion} has some concluding remarks including future directions.
%%%%%%%%%%%%%%%%%%%%%%%%%%%%%%%%%%%%%%%%%%%%%%%%%%%%%%%%%%%%%%%
\section{Solving first order partial differential equations using integration operational matrix}
\label{Sec:FirstOrderPDE}
A linear first order partial differential equation (FOPDE) in $u(x,y)$ with forcing function $f(x,y)$ has the general form
    \begin{equation} \label{EQ:Gen_FOPDE}
    \alpha(x,y) \frac{\partial u}{\partial x} + \beta(x,y) \frac{\partial u}{\partial y} + \gamma(x,y) u = f(x,y).
    \end{equation}
In general, getting the analytical solution of (\ref{EQ:Gen_FOPDE}), subject to some boundary conditions, is often not feasible or too cumbersome. Consequently, one seeks numerical methods to solve such problems. This section describes such a technique using the integrational operational matrix (IOM) of %Zernike circle polynomials
Zernike polynomials. The technique shown below can be adapted for some given $\alpha,$ $\beta,$ $\gamma$.
For the sake of demonstration, we consider the following form of a FOPDE:
    \begin{equation}\label{FirstOrderPDE}
    \alpha x \frac{\partial u}{\partial x} + \beta y \frac{\partial u}{\partial y} + \gamma u = f,
    \end{equation}
where $\alpha,$ $\beta$, and $\gamma$ are constants.
Changing (\ref{FirstOrderPDE}) to polar coordinates $(r, \phi)$ gives
    \begin{equation}\label{FirstOrderPDEPolar}
    r(\alpha \cos^2 \phi + \beta \sin^2 \phi)\frac{\partial u}{\partial r} - (\alpha - \beta)\sin \phi \cos \phi \frac{\partial u}{\partial \phi} + \gamma u = f.
    \end{equation}
%where $\alpha,$ $\beta,$ and $\gamma$ are constants.
%
Equation (\ref{FirstOrderPDEPolar}) has to be solved subject to the boundary conditions
	 \begin{align*}
		u(r_0, \phi) &= h(\phi), \\
    		u(r, \phi_0) &= g(r) .
     \end{align*}
Integrating (\ref{FirstOrderPDEPolar}) first with respect to $r$ from $r_0$ to $r$ and then with respect to $\phi$ from $\phi_0$ to $\phi$ using integration by parts and the given boundary conditions gives
    \begin{align}
    & \alpha r \int_{\phi_0}^{\phi} u \cos^2 \phi \ d \phi + \beta r \int_{\phi_0}^{\phi} u \sin^2 \phi \ d\phi - r_0\int_{\phi_0}^{\phi} h(\phi)(\alpha \cos^2 \phi + \beta \sin^2 \phi) \ d \phi \ - \  \nonumber\\
    - & \int_{\phi_0}^{\phi}(\alpha \cos^2 \phi + \beta \sin^2 \phi)\left[\int_{r_0}^r u \ dr\right]\ d\phi - \frac{(\alpha - \beta)}{2}\left[ \int_{r_0}^{r} \sin 2\phi \ u(r, \phi) - \sin 2 \phi_0 \int_{r_0}^r g(r) \ dr \right .\nonumber \\
& \left .- 2 \int_{r_0}^{r} \int_{\phi_0}^{\phi} u \cos 2\phi \ d \phi \ dr \right]
+ \ \gamma \int_{\phi_0}^{\phi} \int_{r_0}^r u \ dr d\phi = \int_{\phi_0}^{\phi} \int_{r_0}^r f \ dr d\phi. \label{MainIntegralEquation}
    \end{align}
To solve for $u,$ matrix representations for the integral operators, the forcing function $f,$ and the unknown $u,$ in terms of trigonometric and %Zernike radial polynomials
radial parts of Zernike polynomials, are needed. The idea is to write every term in (\ref{MainIntegralEquation}) in terms of %a integration
an integration operational matrix and reduce (\ref{MainIntegralEquation}) to an algebraic equation.

Let the trigonometric functions be written as a vector
$$\Phi(\phi) = [1, \ \cos \phi, \ \sin \phi, \ \cos 2\phi, \ \sin 2\phi, \ \cdots ]^{\textrm{\scriptsize{T}}}.$$
For practical purposes, only a finite number of terms of $\Phi$ can be used. If only terms up to azimuthal frequency $m \phi$ are used, then, by an abuse of notation, we shall also denote by $\Phi$ the resulting vector of size $M = 2m + 1$, i.e.,
	$$
	\Phi(\phi) = [1, \ \cos \phi, \ \sin \phi, \ \cos 2\phi, \ \sin 2\phi, \ \cdots, \ \cos m\phi, \ \sin m \phi]^{\textrm{\scriptsize{T}}} ,
	$$
    and
    \begin{align*}
	\int_{\phi_0}^{\phi} \Phi(\phi) \ d\phi &=
    [\phi - \phi_{0},  \sin \phi - \sin \phi_0, - \cos \phi + \cos \phi_0, \frac{\sin 2\phi}{2} - \frac{\sin 2\phi_0}{2}, \\
    & \phantom{+++++} \cdots, \frac{\sin m\phi}{m} - \frac{\sin m \phi_0}{m}, - \frac{\cos m \phi}{m} + \frac{\cos m \phi_0}{m}]^{\textrm{\scriptsize{T}}} .
    \end{align*}
In order to express the above integral in matrix form, the $\phi$ appearing on the right side has to be expressed in terms of $\{1, \ \cos \phi, \ \sin \phi, \ \cos 2\phi, \ \sin 2\phi, \ldots \}$. To achieve this, we take the Fourier series expansion of $\phi$ over $[0, 2\pi]$ which is
    $$\phi = \pi - \sum_{k=1}^{\infty} \frac 2k \sin k\phi.$$
%which implies
This yields
    $$
    \int_{\phi_0}^{\phi} \Phi(\phi) \ d\phi
    =
    \left[
    \arraycolsep=2.4pt\def\arraystretch{.75}
    \begin{array}{cccccccccc}
    \pi - \phi_0 & 0 & -2 & 0 & -1 & 0 & -\frac23 & \cdots & 0 & -\frac2m \\
    - \sin \phi_0 & 0 & 1 & 0 & 0 & 0 & 0 & \cdots & 0 & 0 \\
    \cos \phi_0 & -1 & 0 & 0 & 0 & 0 & 0 & \cdots & 0 & 0 \\
    -\frac{\sin 2\phi_0}{2} & 0 & 0 & 0 & \frac12 & 0 & 0 & \cdots & 0 & 0 \\
    \frac{\cos 2 \phi_0}{2} & 0 & 0 & -\frac12 & 0 & 0 & 0 & \cdots & 0 & 0 \\
    \vdots & \vdots & \vdots & \vdots & \vdots & \vdots & \vdots & \cdots & \vdots & \vdots \\
    -\frac{\sin m \phi_0}{m} & 0 & 0 & 0 & 0 & 0 & 0 & \cdots & 0 & \frac 1m \\
    \frac{\cos m \phi_0 }{m} & 0 & 0 & 0 & 0 & 0 & 0 & \cdots & -\frac1m & 0
    \end{array}
    \right]
    \left[
    \arraycolsep=2.4pt\def\arraystretch{.75}
    \begin{array}{c}
    1 \\
    \cos \phi \\
    \sin \phi \\
    \cos 2\phi \\
    \sin 2\phi \\
    \vdots \\
    \cos m\phi \\
    \sin m \phi
    \end{array}
    \right]
    = E_{\phi \phi_0} \Phi.
    $$
The %Zernike radial polynomials
radial parts of Zernike  polynomials
$R^m_n(r)$ can be written sequentially as an infinite vector as
	$$
	R(r) = [R_0^0(r), R_1^1(r), R_2^0(r), R_2^2(r), R_3^1 (r), R_3^3(r), \ldots,  R_n^m(r), \ldots ]^{\textrm{\tiny{T}}}, \ n \in \mathbb{N} \cup \{0\}, \ 0 \leq n-m, \ n - m \ \textrm{even} .
	$$
For a fixed $n,$ the number of radial polynomials of degree less than or equal to $n$ is denoted by $N$.
Thus, for approximation with radial polynomials with degree up to $n$, only $N$ elements of the above vector $R(r)$ are used. As in the case of $\Phi$, by an abuse of notation, this is also denoted by $R(r)$ and is given by
	$$
	R(r) = [R_0^0(r), R_1^1(r), R_2^0(r), R_2^2(r), R_3^1 (r), R_3^3(r), \ldots,  R_n^m(r)]^{\textrm{\tiny{T}}}, \ n \in \mathbb{N} \cup \{0\}, \ 0 \leq n-m, \ n - m \ \textrm{even} .
	$$
The solution $u$ represented in terms of the %Zernike circle polynomials
Zernike polynomials then results in an approximation
    \begin{equation}\label{EQ:ApproxSol}
    \widetilde{u} = \Phi^{\textrm{\tiny{T}}}U R,
    \end{equation}
where $U$ contains the coefficients of $u$ with respect to the polynomials up to a chosen degree. To help in understanding, first consider the first term on the left side of (\ref{MainIntegralEquation}). This can be written as
	\begin{equation} \label{EXP:TermOne}
	\alpha r \int_{\phi_0}^{\phi} \cos^2 \phi \Phi^{\textrm{\tiny{T}}}U R \ d \phi.
	\end{equation}
Take $\phi_0 = 0.$
%Let
%    $$
%    x_1 = \frac{1}{2m} + \frac{1}{4(m-2)} + \frac{1}{4(m+2)}.
%    $$
Then
 \footnotesize{
 \begin{align*}
     &\int_0^{\phi} \cos^2 \phi \ \Phi(\phi) \ d \phi  = \\
     &
     \left[
      \arraycolsep=3pt\def\arraystretch{1}
     \begin{array}{cccccccccccccccccc}
     \frac{\pi}{2} & 0 & -1 & 0 & -\frac{1}{4} & 0 & -\frac{1}{3} & 0 & -\frac{1}{4} & 0 & -\frac{1}{5} & \cdots & 0 & \frac{-1}{m-2} & 0 & \frac{-1}{m-1} & 0 & \frac{-1}{m} \\
     0 & 0 & \frac{3}{4} & 0 & 0 & 0 & \frac{1}{12} & 0 & 0 & 0 & 0 & \cdots & 0 & 0 & 0 & 0 & 0 & 0 \\
     \frac{1}{3} & -\frac{1}{4} & 0 & 0 & 0 & -\frac{1}{12} & 0 & 0 & 0 & 0 & 0 & \cdots & 0 & 0 & 0 & 0 & 0 & 0 \\
    \pi/4 & 0 & -\frac12 & 0 & 0 & 0 & -\frac{1}{6} & 0 & -\frac{1}{16} & 0 & -\frac{1}{10} & \cdots & 0 & \frac{-1}{2(m-2)} & 0 & -\frac{1}{2(m-1)} & 0 & \frac{-1}{2m}\\
     \frac{5}{16} & 0 & 0 & -\frac{1}{4} & 0 & 0 & 0 & -\frac{1}{16} & 0 & 0 & 0 & \cdots & 0 & 0 & 0 & 0 & 0 & 0\\
    0 & 0 & \frac{1}{4} & 0 & 0 & 0 & \frac{1}{6} & 0 & 0 & 0 & \frac{1}{20} & \cdots & 0 & 0 & 0 & 0 & 0 & 0\\	
    \frac{7}{15} & -\frac{1}{4} & 0 & 0 & 0 & -\frac{1}{6} & 0 & 0 & 0 & -\frac{1}{20} & 0 & \cdots & 0 & 0 & 0 & 0 & 0 & 0 \\
    \vdots & \vdots & \vdots & \vdots & \vdots & \vdots & \vdots & \vdots & \vdots & \vdots & \vdots & \cdots & \vdots & \vdots & \vdots & \vdots & \vdots & \vdots \\
    0 & 0 & 0 & 0 & 0 & 0 & 0 & 0 & 0 & 0 & 0 & \cdots & 0 & \frac{1}{4(m-2)} & 0 & 0 & 0 & \frac{1}{2m}\\
    x_1 & 0 & 0 & 0 & 0 & 0 & 0 & 0 & 0 & 0 & 0 & \cdots & \frac{-1}{4(m-2)} & 0 & 0 & 0 & \frac{-1}{2m} & 0
     \end{array}
     \right]
     \begin{bmatrix}
    1 \\
    \cos \phi \\
    \sin \phi \\
    \cos 2\phi \\
    \sin 2\phi \\
    \cos 3\phi \\
    \sin 3 \phi \\
    \cos 4\phi \\
    \sin 4 \phi \\
    \cos 5\phi \\
    \sin 5 \phi \\
    \vdots \\
    \cos m \phi \\
    \sin m \phi
    \end{bmatrix}
    \\
    &=
    E_{\phi}^{\cos^2\phi} \Phi(\phi),
    \end{align*}
}
\normalsize
where $x_1 = \frac{1}{2m} + \frac{1}{4(m-2)} + \frac{1}{4(m+2)}$ in the last row of the matrix,
%Let
%$$
%x_2 = \frac{1}{2m} - \frac{1}{4(m-2)} - \frac{1}{4(m+2)}.
%$$
%Then
\footnotesize{
    \begin{align*}
     &\int_0^{\phi} \sin^2 \phi \ \Phi(\phi) \ d \phi  = \\
     &
     \left[
     \arraycolsep=3pt\def\arraystretch{1}
     \begin{array}{cccccccccccccccccc}
     \frac{\pi}{2} & 0 & -1 & 0 & -\frac{3}{4} & 0 & -\frac{1}{3} & 0 & -\frac{1}{4} & 0 & -\frac{1}{5} & \cdots & 0 & \frac{-1}{m-2} & 0 & \frac{-1}{m-1} & 0 & \frac{-1}{m} \\
     0 & 0 & \frac{1}{4} & 0 & 0 & 0 & -\frac{1}{12} & 0 & 0 & 0 & 0 & \cdots & 0 & 0 & 0 & 0 & 0 & 0 \\
     \frac{2}{3} & -\frac{3}{4} & 0 & 0 & 0 & \frac{1}{12} & 0 & 0 & 0 & 0 & 0 & \cdots & 0 & 0 & 0 & 0 & 0 & 0 \\
    -\frac{\pi}{4} & 0 & \frac{1}{2} & 0 & \frac{1}{2} & 0 & \frac{1}{6} & 0 & \frac{1}{8} & 0 & \frac{1}{10} & \cdots & 0 &  \frac{1}{2(m-2)}  & 0 & \frac{1}{2(m-1)} & 0 & \frac{1}{2m}\\
     \frac{3}{16} & 0 & 0 & -\frac{1}{4} & 0 & 0 & 0 & \frac{1}{16} & 0 & 0 & 0 & \cdots & 0 & 0 & 0 & 0 & 0 & 0\\
    0 & 0 & -\frac{1}{4} & 0 & 0 & 0 & \frac{1}{6} & 0 & 0 & 0 & -\frac{1}{20} & \cdots & 0 & 0 & 0 & 0 & 0 & 0\\	
    \frac{-2}{15} & \frac{1}{4} & 0 & 0 & 0 & -\frac{1}{6} & 0 & 0 & 0 & \frac{1}{20} & 0 & \cdots & 0 & 0 & 0 & 0 & 0 & 0 \\
    \vdots & \vdots & \vdots & \vdots & \vdots & \vdots & \vdots & \vdots & \vdots & \vdots & \vdots & \cdots & \vdots & \vdots & \vdots & \vdots & \vdots & \vdots \\
    0 & 0 & 0 & 0 & 0 & 0 & 0 & 0 & 0 & 0 & 0 & \cdots & 0 & \frac{-1}{4(m-2)} & 0 & 0 & 0 & \frac{1}{2m}\\
    x_2 & 0 & 0 & 0 & 0 & 0 & 0 & 0 & 0 & 0 & 0 & \cdots & \frac{1}{4(m-2)} & 0 & 0 & 0 & \frac{-1}{2m} & 0
     \end{array}
     \right]
     \begin{bmatrix}
    1 \\
    \cos \phi \\
    \sin \phi \\
    \cos 2\phi \\
    \sin 2\phi \\
    \cos 3\phi \\
    \sin 3 \phi \\
    \cos 4\phi \\
    \sin 4 \phi \\
    \cos 5\phi \\
    \sin 5 \phi \\
    \vdots \\
    \cos m \phi \\
    \sin m \phi
    \end{bmatrix}
    \\
    &=
    E_{\phi}^{\sin^2\phi} \Phi(\phi),
    \end{align*}
}
\normalsize
where $x_2 = \frac{1}{2m} - \frac{1}{4(m-2)} - \frac{1}{4(m+2)}$ in the last row of the matrix, \\ and
% $$
% x_3 = \frac{1}{2(m+2)} + \frac{1}{2(m-2)}.
% $$
%Then
\footnotesize{
 \begin{align*}
     &\int_0^{\phi} \cos 2 \phi \ \Phi(\phi) \ d \phi = \\
     &
     \left[
     \arraycolsep=3pt\def\arraystretch{1}
     \begin{array}{cccccccccccccccccc}
     0 & 0 & 0 & 0 & \frac12 & 0 & 0 & 0 & 0 & 0 & 0 & \cdots & 0 & 0 & 0 & 0 & 0 & 0 \\
     0 & 0 & \frac12 & 0 & 0 & 0 & \frac16 & 0 & 0 & 0 & 0 & \cdots & 0 & 0 & 0 & 0 & 0 & 0 \\
     -\frac13 & \frac12 & 0 & 0 & 0 & -\frac16 & 0 & 0 & 0 & 0 & 0 & \cdots & 0 & 0 & 0 & 0 & 0 & 0 \\
     \frac{\pi}{2} & 0 & -1 & 0 & -\frac12 & 0 & -\frac13 & 0 & -\frac{1}{8} & 0 & 0 & \cdots & 0 & -\frac{1}{m-2} & 0 & -\frac{1}{m-1} & 0 & -\frac 1m\\
     \frac{1}{8} & 0 & 0 & 0 & 0 & 0 & 0 & -\frac18 & 0 & 0 & 0 & \cdots & 0 & 0 & 0 & 0 & 0 & 0\\
    \vdots & \vdots & \vdots & \vdots & \vdots & \vdots & \vdots & \vdots & \vdots & \vdots & \vdots & \cdots & \vdots & \vdots & \vdots & \vdots & \vdots & \vdots \\
    0 & 0 & 0 & 0 & 0 & 0 & 0 & 0 & 0 & 0 & 0 & \cdots & 0  & \frac{1}{2(m-2)} & 0 & 0 & 0 & 0\\
    x_3 & 0 & 0 & 0 & 0 & 0 & 0 & 0 & 0 & 0 & 0 & \cdots & - \frac{1}{2(m-2)} & 0 & 0 & 0 & 0 & 0
     \end{array}
     \right]
     \begin{bmatrix}
    1 \\
    \cos \phi \\
    \sin \phi \\
    \cos 2\phi \\
    \sin 2\phi \\
    \cos 3\phi \\
    \sin 3 \phi \\
    \cos 4\phi \\
    \sin 4 \phi \\
    \cos 5\phi \\
    \sin 5 \phi \\
    \vdots \\
    \cos m \phi \\
    \sin m \phi
    \end{bmatrix}
    =
    E_{\phi}^{\cos 2\phi} \Phi(\phi).
    \end{align*}
 }
\normalsize
where $x_3 = \frac{1}{2(m+2)} + \frac{1}{2(m-2)}$ in the last row of the matrix.
\\
\\
When one considers integrating from some angle $\phi_0$ to $\phi$ instead of from zero to $\phi$ then the following adjustments have to be made.
The first column of $E_{\phi}^{\cos^2\phi}$ changes to
	$$
	\begin{bmatrix}
	\frac{\pi}{2}-\frac{\phi_0}{2} - \frac{\sin 2\phi_0}{4} \\
	-\frac34\sin\phi_0-\frac{\sin 3\phi_0}{12} \\
	\frac{\cos \phi_0}{4} + \frac{\cos 3\phi_0}{12} \\
	\vdots\\
	-\frac{\sin(m-2)\phi_0}{4(m-2)} - \frac{\sin m\phi_0}{2m} - \frac{\sin(m+2)\phi_0}{4(m+2)} \\
	\frac{\cos(m-2)\phi_0}{4(m-2)} + \frac{\cos m\phi_0}{2m} + \frac{\cos(m+2)\phi_0}{4(m+2)}
	\end{bmatrix},
	$$
and the resulting matrix will be denoted by $E_{\phi \phi_0}^{\cos^2 \phi}$. \footnote{In general, an integration operational matrix will be denoted by $E,$ the subscript indicating the limits of integration while the superscript indicating the integrand.}
The first column of $E_{\phi}^{\sin^2\phi}$ changes to
	$$
	\begin{bmatrix}
	\frac{\pi}{2}-\frac{\phi_0}{2} + \frac{\sin 2\phi_0}{4} \\
	-\frac14\sin\phi_0 + \frac{\sin 3\phi_0}{12} \\
	\frac34 \cos \phi_0 - \frac{\cos 3\phi_0}{12} \\
	\vdots \\
	\frac{\sin(m-2)\phi_0}{4(m-2)} - \frac{\sin m\phi_0}{2m} + \frac{\sin(m+2)\phi_0}{4(m+2)} \\
	-\frac{\cos(m-2)\phi_0}{4(m-2)} + \frac{\cos m \phi_0}{2m} - \frac{\cos (m+2)\phi_0}{4(m+2)}
	\end{bmatrix},
	$$
and the resulting  matrix will be denoted by $E_{\phi \phi_0}^{\sin^2 \phi}.$
The first column of $E_{\phi}^{\cos 2\phi}$ becomes
$$
\left[
\arraycolsep=2.4pt\def\arraystretch{.75}
	\begin{array}{c}
	- \frac{\sin 2 \phi_0}{2} \\
	- \frac{\sin \phi_0}{2} - \frac{\sin 3 \phi_0}{6} \\
	-\frac{\cos \phi_0}{2} + \frac{\cos 3 \phi_0}{6} \\
	 - \frac{\phi_0}{2} - \frac{\sin 4 \phi_0}{8} \\
	\frac{\pi}{2} - \frac{\phi_0}{2} - \frac{\sin 4 \phi_0}{8} \\
	\frac{\cos 4 \phi_0}{8}  \\
	\vdots \\
	-\frac{\sin (m-2)\phi_0}{2(m-2)} - \frac{\sin (m+2)\phi_0}{2(m+2)} \\
	\frac{\cos(m-2)\phi_0}{2(m-2)} + \frac{\cos (m+2)\phi_0}{2(m+2)}
	\end{array}
	\right],
	$$
and the resulting matrix will be denoted by $E_{\phi \phi_0}^{\cos 2 \phi}.$
For demonstration purposes, to keep things less cumbersome, the %Zernike radial  polynomials
radial parts of Zernike polynomials up to degree three will be used below, and will be denoted by the vector $R(r).$
    \begin{equation}\label{RadialBasis_3}
    R(r) = [R_0^0(r), \ R_1^1(r), \ R_2^0(r), \ R_2^2(r), \ R_3^1(r), \ R_3^3(r)]^{\textrm{\tiny{T}}}
    = [1, \ r, \ 2r^2 - 1, \  r^2, \ 3r^3 - 2r, \ r^3]^{\textrm{\tiny{T}}}.
    \end{equation}
 In (\ref{EXP:TermOne}), $rR(r)$ can be approximated as
	\begin{equation}\label{IOP:M_Rr}
	r R(r)=
	\left[
	\arraycolsep=2.4pt\def\arraystretch{.75}
	\begin{array}{c}
	r \\
	r^2 \\
	2r^3 - r \\
	r^3 \\
	3r^4 - 2r^2 \\
	r^4
	\end{array}
	\right]
	\approx
	\left[
	\arraycolsep=2.4pt\def\arraystretch{.75}
	\begin{array}{cccccc}
	0 & 1 & 0 & 0 & 0 & 0 \\
	0 & 0 & 0 & 1 & 0 & 0 \\
	0 & -1 & 0 & 0& 0 & 2 \\
	0 & 0 & 0 & 0 & 0 & 1\\
	0 & 0 & 0 & -2 & 0 & 0 \\
	0 & 0& 0 & 0 & 0 & 0
	\end{array}
	\right]
	\left[
	\arraycolsep=2.4pt\def\arraystretch{.75}
	\begin{array}{c}
    	1 \\
    	r \\
    	2r^2 - 1 \\
    	r^2 \\
    	3r^3 - 2r \\
    	r^3
    	\end{array}
	\right]
	= M_R^r R(r).
	\end{equation}
In (\ref{IOP:M_Rr}), the higher order terms involving $r^4$ are ignored %making the last row of $M_R^r$ zero.
making the last row of $M_R^r$ equal to zero.
See  Remark~\ref{Rem:Approximation} for how to get a better approximation by projecting terms involving $r^4$ on the space generated by $R(r)$. The first two terms of (\ref{MainIntegralEquation}) can thus be written as \footnote{Strictly speaking, the resulting matrix representation is an approximation of some integral. However, we will always use the equality sign to express that the operators on the right side of the equality stand for corresponding integrals on the left side.}
    $$
    \alpha r \int_{\phi_0}^{\phi} u \cos^2 \phi \ d \phi + \beta r \int_{\phi_0}^{\phi} u \sin^2 \phi \ d\phi =
    \alpha \Phi^{\textrm{\tiny{T}}} \left( E_{\phi \phi_0}^{\cos^2 \phi}\right)^{\textrm{\tiny{T}}} U M_R^r R
    +
     \beta \Phi^{\textrm{\tiny{T}}} \left( E_{\phi \phi_0}^{\sin^2 \phi}\right)^{\textrm{\tiny{T}}} U M_R^r R .
    $$
The function $h(\phi)$ is expanded in terms of trigonometric functions as
    $$h(\phi) = \mathbf{h}^{\textrm{\tiny{T}}} \Phi(\phi) = \Phi(\phi)^{\textrm{\tiny{T}}} \mathbf{h}, $$
 where $\mathbf{h}$ contains the coefficients of $h$ in terms of the functions in $\Phi.$
Also, $r_0$ can be expressed as
    $$r_0 = M_{r_0}^{\textrm{\tiny{T}}} R(r),$$
where $M_{r_0}$ is a vector whose first entry is $r_0$ and the rest are zero. These expressions for $h(\phi)$ and $r_0$ transform the third term in (\ref{MainIntegralEquation}) to
        $$r_0 \int_{\phi_0}^{\phi} h(\phi) (\alpha \cos^2 \phi + \beta \sin^2 \phi) \ d \phi = \alpha  \Phi^{\textrm{\tiny{T}}}\left( E_{\phi \phi_0}^{\cos^2 \phi}\right)^{\textrm{\tiny{T}}} \mathbf{h} M_{r_0}^{\textrm{\tiny{T}}} R(r) + \beta \Phi^{\textrm{\tiny{T}}}\left( E_{\phi \phi_0}^{\sin^2 \phi}\right)^{\textrm{\tiny{T}}} \mathbf{h} M_{r_0}^{\textrm{\tiny{T}}} R(r).
        $$
Using the recurrence relation given in \cite{Noll:1976} and \cite{Prata:1989}
     \begin{equation}\label{eq:ZernikePolyRecurrence}
     \int_{r_0}^r \left[R_n^m(r) + R_{n}^{m+2}(r) \right] \ dr = \frac{1}{n+1} \left[R_{n+1}^{m+1}(r) - R_{n-1}^{m+1}(r)\right]\Big|_{r_0}^r,
     \end{equation}
     gives
    \begin{equation} \label{EQ:IOPRadialPoly}
    \int_{r_0}^r R(\rho) \ d\rho
    \approx
    \left[
    \arraycolsep=2.4pt\def\arraystretch{.75}
    \begin{array}{lccccc}
    -R_1^1(r_0) & 1 & 0 & 0 & 0 & 0 \\
    -\frac12 R_2^2(r_0) & 0 & 0 & \frac 12 & 0 & 0 \\
    -\frac13\left[R_3^1(r_0) - R_1^1(r_0) - R_3^3(r_0) \right] & -\frac13 & 0 & 0 & \frac13 & -\frac13 \\
    -\frac13 R_3^3(r_0) & 0 & 0 & 0 & 0 & \frac13 \\
    -\frac14\left[R_4^2(r_0) - R_2^2(r_0) - R_4^4(r_0)\right] & 0 & 0& -1 & 0 & 0 \\
    -\frac14 R_4^4(r_0) & 0 & 0 & 0 & 0 & 0
    \end{array}
    \right] R(r) = E_{r r_0} R(r),
    \end{equation}
where recall that $R_0^0(r) = 1,$ $R_1^1(r) = r,$ $R_2^0 (r) = 2r^2 -1,$ $R_2^2(r) = r^2,$ $R_3^1(r) = 3r^3 - 2r,$ and $R_3^3(r) = r^3.$ Once again, as in (\ref{IOP:M_Rr}), the higher order terms involving $r^4$ are ignored resulting in zeros in the last row of $E_{r r_0}$ in (\ref{EQ:IOPRadialPoly}), see Remark~\ref{Rem:Approximation}.
Even though the derivation here uses only %Zernike radial polynomials
radial parts of Zernike  polynomials up to degree three, the general expression of the integration operational matrix $E_{r r_0},$ using all radial polynomials up to some given degree $n,$ is given in Section \ref{Appendix}.
Using (\ref{EQ:IOPRadialPoly}), the fourth term of (\ref{MainIntegralEquation}) is
     \begin{align*}
     \alpha \int_{\phi_0}^{\phi} \cos^2 \phi \left[\int_{r_0}^r u \ dr \right] \ d \phi
     &= \alpha \int_{\phi_0}^{\phi} \cos^2 \phi \int_{r_0}^r \Phi^{\textrm{\tiny{T}}} U R \ dr \ d \phi
     = \alpha \int_{\phi_0}^{\phi} \cos^2 \phi \ \Phi^{\textrm{\tiny{T}}} \ d \phi \ U \int_{r_0}^r R \ dr \\
     &= \alpha \left(E_{\phi \phi_0}^{\cos^2 \phi} \Phi \right)^{\textrm{\tiny{T}}} U E_{r r_0} R ,
     \end{align*}
and similarly
    $$
\beta \int_{\phi_0}^{\phi} \sin^2 \phi \left[\int_{r_0}^r u \ dr \right] \ d\phi = \beta \left(E_{\phi \phi_0}^{\sin^2 \phi} \Phi \right)^{\textrm{\tiny{T}}} U E_{r r_0} R.
	$$
If $\sin 2\phi \Phi$ is written as
    $$\sin 2\phi \Phi = M_{\Phi}^{\sin 2\phi} \Phi$$
then
    $$\frac{\alpha - \beta}{2} \int_{r_0}^r \sin 2\phi  \ u(r, \phi) \ dr = \int_{r_0}^r \sin 2\phi \Phi^{\textrm{\tiny{T}}} U R \ dr = \Phi^{\textrm{\tiny{T}}} \left(M_{\Phi}^{\sin 2\phi}\right)^{\textrm{\tiny{T}}} U E_r R.$$
Expressing $g(r)$ as
    $$g(r) = \mathbf{g}^{\textrm{\tiny{T}}} R(r)$$
gives
    $$
    \int_{r_0}^r g(r) \ dr = \int_{r_0}^r \mathbf{g}^{\textrm{\tiny{T}}} R(r) \ dr = \mathbf{g}^{\textrm{\tiny{T}}} E_{r r_0} R(r).
    $$	
Also, $\sin 2\phi_0$ can be expressed as
    $$
    \sin 2\phi_0 = M_{\sin 2\phi_0}^{\textrm{\tiny{T}}} \Phi(\phi) = \Phi(\phi)^{\textrm{\tiny{T}}} M_{\sin 2\phi_0},
    $$
where $M_{\sin 2\phi_0}$ is a vector whose first entry is $\sin 2\phi_0$ and the rest are zero. Therefore,
    $$
    \frac{\alpha - \beta}{2} \sin 2 \phi_0\int_{r_0}^r g(r) \ dr = \frac{\alpha - \beta}{2} \Phi^{\textrm{\tiny{T}}} M_{\sin 2\phi_0} \mathbf{g}^{\textrm{\tiny{T}}} E_{r r_0} R(r).
    $$
In matrix form,
    $$
    \int_{\phi_0}^{\phi} \cos 2 \phi \Phi = E_{\phi \phi_0}^{\cos 2\phi} \Phi ,
    $$
and thus
    $$
    \frac{\alpha - \beta}{2} \cdot 2 \int_{r_0}^r \int_{\phi_0}^{\phi} u \cos 2\phi \ d\phi \ dr
    =
    (\alpha - \beta)  \ \Phi^{\textrm{\tiny{T}}} \left(E_{\phi \phi_0}^{\cos 2\phi}\right)^{\textrm{\tiny{T}}} U E_{r r_0} R.
    $$
The last term on the left side of (\ref{MainIntegralEquation}) becomes, in matrix form,
	\begin{align*}
	\gamma \int_{\phi_0}^{\phi} \int_{r_0}^r u \ dr \ d \phi &= \gamma \int_{\phi_0}^{\phi} \int_{r_0}^r  \Phi^{\textrm{\tiny{T}}} U R \ dr \ d \phi
	=  \gamma \left(\int_{\phi_0}^{\phi} \Phi^{\textrm{\tiny{T}}} \ d \phi \right) U \left( \int_{r_0}^r R \ dr \right) \\
	& = \gamma (E_{\phi \phi_0} \Phi)^{\textrm{\tiny{T}}} U E_{r r_0} R.
	\end{align*}
%\cos \phi \\
%		
Finally, expressing the forcing function $f$ in terms of the %Zernike circle polynomials
Zernike polynomials as
    $$f = \Phi^{\textrm{\tiny{T}}}F R,$$
where $F$ contains the coefficients of $f,$ the integral on the right side of (\ref{MainIntegralEquation}) can be written as
    $$
    \int_{\phi_0}^{\phi} \int_{r_0}^r f \ dr \ d \phi  = (E_{\phi \phi_0} \Phi)^{\textrm{\tiny{T}}} F E_{r r_0} R.
    $$
%For ease, the
The operator matrices and the corresponding notation are summarized below:
\begin{enumerate}
\item $\int_{\phi_0}^{\phi} \Phi \ d \phi = E_{\phi \phi_0} \Phi.$
\item $\int_{\phi_0}^{\phi} \cos^2 \phi \ \Phi \ d\phi = E_{\phi \phi_0}^{\cos^2 \phi} \Phi.$
\item $\int_{\phi_0}^{\phi} \sin^2 \phi \ \Phi \ d\phi = E_{\phi \phi_0}^{\sin^2 \phi} \Phi.$
\item   $\int_{r_0}^{r} R(r) \ dr = E_{r r_0} R(r).$
\item $\int_{\phi_0}^{\phi} \cos 2\phi \ \Phi \ d\phi = E_{\phi \phi_0}^{\cos 2\phi} \Phi.$
\end{enumerate}
Putting everything together, (\ref{MainIntegralEquation}) reduces to the algebraic equation
\begin{align}
	\Phi^{\textrm{\tiny{T}}}
	&
	\left[
	\alpha  \left(E_{\phi \phi_0}^{\cos^2 \phi}\right)^{\textrm{\tiny{T}}}  U M_R^r
	+
	\beta  \left(E_{\phi \phi_0}^{\sin^2 \phi}\right)^{\textrm{\tiny{T}}}  U M_R^r
	 -
   	\alpha  \left( E_{\phi \phi_0}^{\cos^2 \phi}\right)^{\textrm{\tiny{T}}} \mathbf{h} M_{r_0}^{\textrm{\tiny{T}}}
   	-
   	\beta  \left( E_{\phi \phi_0}^{\sin^2 \phi}\right)^{\textrm{\tiny{T}}} \mathbf{h} M_{r_0}^{\textrm{\tiny{T}}}
	\right.
	\nonumber \\
	& -
	\alpha \left(E_{\phi \phi_0}^{\cos^2 \phi}\right)^{\textrm{\tiny{T}}}  U E_{r r_0}
	-
	\beta  \left(E_{\phi \phi_0}^{\sin^2 \phi}\right)^{\textrm{\tiny{T}}}  U E_{r r_0}
	-
	\frac{(\alpha - \beta)}{2}  \left(M_{\Phi}^{\sin 2\phi}\right)^{\textrm{\tiny{T}}} U E_{r r_0}
	 +
	 \frac{\alpha - \beta}{2} M_{\sin 2\phi_0} \mathbf{g}^{\textrm{\tiny{T}}} E_{r r_0} \nonumber \\
	 & +
	(\alpha - \beta)  \ \left(E_{\phi \phi_0}^{\cos 2\phi}\right)^{\textrm{\tiny{T}}} U E_{r r_0}
	+
    	\gamma E_{\phi \phi_0}^{\textrm{\tiny{T}}} U E_{r r_0} \Big] R(r)
	=
\Phi^{\textrm{\tiny{T}}} E_{\phi \phi_0}^{\textrm{\tiny{T}}} F E_{r r_0} R(r). \label{MatrixEquationGeneral}
	\end{align}
Equation (\ref{MatrixEquationGeneral}) has to be solved for the matrix $U$ to get an approximation of the solution $u$ of the original PDE (\ref{FirstOrderPDEPolar}).
To solve for $U,$ it is convenient to rewrite (\ref{MatrixEquationGeneral}) using the vector and tensor product representation introduced in (\ref{EQ:VectorMatrixTensorProduct}) of Section~\ref{sec:Intro}.
With this notation, (\ref{MatrixEquationGeneral}) becomes
    \begin{align}
    & \left(\alpha {M_R^r}^{\textrm{\tiny{T}}} \otimes \left(E_{\phi \phi_0}^{\cos 2 \phi}\right)^{\textrm{\tiny{T}}}
     + \beta {M_R^r}^{\textrm{\tiny{T}}} \otimes \left(E_{\phi \phi_0}^{\sin 2 \phi}\right)^{\textrm{\tiny{T}}}
     - \alpha E_{r r_0}^{\textrm{\tiny{T}}} \otimes \left(E_{\phi \phi_0}^{\cos 2 \phi}\right)^{\textrm{\tiny{T}}}
     - \beta E_{r r_0}^{\textrm{\tiny{T}}} \otimes \left(E_{\phi \phi_0}^{\sin 2 \phi}\right)^{\textrm{\tiny{T}}}
     \right.
     \nonumber \\
    & \phantom{++++}
    - \frac{\alpha - \beta}{2} E_{r r_0}^{\textrm{\tiny{T}}} \otimes \left(M_{\phi}^{\sin 2\phi} \right)^{\textrm{\tiny{T}}}
    + (\alpha - \beta) E_{r r_0}^{\textrm{\tiny{T}}} \otimes \left(E_{\phi \phi_0}^{\cos 2\phi}\right)^{\textrm{\tiny{T}}}
    + \gamma E_{rr_0}^{\textrm{\tiny{T}}} \otimes E_{\phi \phi_0}^{\textrm{\tiny{T}}}\Big) \textrm{vec}(U)
    \nonumber \\
    & \phantom{+++++}
    = \textrm{vec}\left(
    E_{\phi \phi_0}^{\textrm{\tiny{T}}} F E_{ r r_0} +
    \alpha \left(E_{\phi \phi_0}^{\cos^2\phi}\right)^{\textrm{\tiny{T}}} \mathbf{h} M_{r_0}^{\textrm{\tiny{T}}}
    +
    \beta \left(E_{\phi \phi_0}^{\sin^2\phi}\right)^{\textrm{\tiny{T}}} \mathbf{h} M_{r_0}^{\textrm{\tiny{T}}}
    - \frac{\alpha - \beta}{2} M_{\sin 2\phi_0} \mathbf{g}^{\textrm{\tiny{T}}} E_{r r_0}
    \right)  \label{eq:TensorRep}
    \end{align}
    which can be thought of as a linear system
    \begin{equation} \label{eq:Ax=y}
    A \mathbf{x} = \mathbf{b},
    \end{equation}
    where $\mathbf{x} = \textrm{vec}(U)$ is an unknown vector of size $MN$, $\mathbf{b}$ is a known vector also of size $MN,$ and $A$ is a sparse matrix of order $MN$. The solution $\mathbf{x}$ is then reshaped as an $M\times N$ matrix $U$ which gives the approximate solution $\tilde{u}(r,\phi)=\Phi^T(\phi)UR(r)$ in (\ref{EQ:ApproxSol}). One can consider solving (\ref{eq:Ax=y}) in two ways. One way is to get the minimum norm least squares solution $\mathbf{x}=A^\dag \mathbf{b}$ where $A^\dag$ is the standard matrix pseudo-inverse
or Moore-Penrose inverse of $A$. This pseudo-inverse exists and unique for any matrix. The solution provided by
$A^\dag \mathbf{b}$ is a  least squares minimum norm solution and is called here the $l_2$ solution.
The other way is to get the minimum $l_1$-norm solution by linear programming using $l_1$-magic \cite{Romberg:2006}. In the latter case, the problem is  formulated  as
    \begin{eqnarray*}
        \text{Minimize } |x_1| + \cdots + |x_{MN}| \text{ subject to } A\mathbf{x} = \mathbf{b}.
    \end{eqnarray*}
\begin{remark}
\label{Rem:Approximation}
It is important to say a few words on the operational matrices $M_R^r$ and $E_{r r_0}$ in (\ref{IOP:M_Rr}) and (\ref{EQ:IOPRadialPoly}), respectively. In obtaining these matrices, all terms of degree greater than $n = 3$ have been neglected. For the sake of higher accuracy of the solution, these neglected terms can be represented in terms of the radial polynomials in $R(r),$ see(\ref{RadialBasis_3}), by projecting on the space spanned by $R(r)$. Alternatively, a Lagrange interpolation polynomial can be constructed using $R(r)$ to represent each of the neglected higher order terms. The calculated coefficients in the representation of these higher order terms can then be used in the integration operational matrix as explained in connection with (\ref{SOPDE12}) in the next section.

In the case of the FOPDE in Example~\ref{EX:FOPDE} below, when the projection of higher order terms is not considered, the solution surface with the $l_1$ method is found to be acceptable when compared with the actual solution but quite distorted with the $l_2$ method. It is found that projecting these higher order terms on the space of lower order radial polynomials yields solutions with higher accuracy in both the $l_1$ and $l_2$ methods.
This method of projection %of the ignored higher order terms on the space of lower order polynomials
has been used to obtain the results in Example~\ref{EX:FOPDE} that are shown in Figure~\ref{fig:figure3} although we have not displayed the updated matrices considering the projections in our calculations above. We have shown this for the case of a second order PDE in Section~\ref{Sec:SecondOrderPDE}. In the case of a second order PDE, as discussed in Section~\ref{Sec:SecondOrderPDE}, this method of projection is found to be crucial in getting a satisfactory solution.
\end{remark}
\begin{example} \label{EX:FOPDE}
\rm
Let $\alpha = 1,$ $\beta = -1,$ $\gamma = 1,$ and $f = e^{r\cos \phi}(1 + r \cos \phi)$ in (\ref{FirstOrderPDE}) and (\ref{FirstOrderPDEPolar}).  With this choice one can proceed to solve the following initial value problem:
    \begin{equation} \label{FirstOrderPDE_Example}
    r\cos 2\phi \frac{\partial u}{\partial r} - \sin 2\phi \frac{\partial u}{\partial \phi} + u =  e^{r\cos \phi}(1 + r \cos \phi)
    \end{equation}
subject to the initial conditions
    \begin{align*}
    u(0, \phi) &= 1,\\
    u(r, 0) &= e^r.
    \end{align*}
It can be checked by direct substitution that $u(r, \phi) = e^{r\cos \phi}$ is a solution to the above initial value problem (\ref{FirstOrderPDE_Example}). By keeping terms %up to the third degree
of degree at most three in the expansion by Zernike polynomials, an approximation of the actual solution $u(r, \phi) = e^{r\cos \phi}$ is
\begin{eqnarray*}
\widetilde{u}(r, \phi) &=& 1 + r \cos \phi + \frac 12 r^2 \cos^2 \phi + \frac 16 r^3 \cos^3 \phi
=
1 + r \cos \phi + \frac 14 r^2 \cos 2 \phi + \frac{1}{24} r^3 \cos 3 \phi + \frac18 r^3 \cos \phi \nonumber \\
&=&
\Phi^{\textrm{\scriptsize{T}}}(\phi) U R(r), \label{ExpU}
\end{eqnarray*}
where
$$
U =
\left[
\arraycolsep=2.4pt\def\arraystretch{.9}
\begin{array}{cccccc}
1 & 0 & 0 & \tfrac14 & 0 & 0 \\
0 & 1 & 0 & 0 & 0 & \tfrac18 \\
0 & 0 & 0 & 0 & 0 & 0 \\
0 & 0 & 0 & \tfrac14 & 0 & 0 \\
0 & 0 & 0 & 0 & 0 & 0 \\
0 & 0 & 0 & 0 & 0 & \tfrac{1}{24} \\
0 & 0 & 0 & 0 & 0 & 0 \\
\end{array}
\right].
$$
Expanding the forcing function $f$ in terms of the Zernike polynomials, up to an approximation of order three, gives
	\begin{align}
	f(r, \phi)  &= e^{r\cos \phi}(1 + r \cos \phi)
	=
	1 + 2r\cos \phi + \frac34 r^2 + \frac34 r^2 \cos 2\phi + \frac{r^3}{2}\cos \phi + \frac{r^3}{6} \cos 3 \phi \nonumber \\
	&=
	\Phi^{\textrm{\scriptsize{T}}}(\phi) F R(r) \label{ExpForcingF}
	\end{align}
where
	$$F =
	\left[
	\arraycolsep=2.4pt\def\arraystretch{.9}
	\begin{array}{cccccc}
	1 & 0 & 0 & \tfrac34 & 0 & 0 \\
	0 & 2 & 0 & 0 & 0 & \tfrac12 \\
	0 & 0 & 0 & 0 & 0 & 0 \\
	0 & 0 & 0 & \tfrac34 & 0 & 0 \\
	0 & 0 & 0 & 0 & 0 & 0 \\
	0 & 0 & 0 & 0 & 0 & \tfrac13 \\
	0 & 0 & 0 & 0 & 0 & 0
	\end{array}
	\right].
	$$
The other unknowns in (\ref{MatrixEquationGeneral}) are the vectors $\mathbf{h}$ and $\mathbf{g}$ which can be substituted as follows.
Let
$$
\mathbf{h}^{\textrm{\tiny{T}}} = \left[
\arraycolsep=2pt\def\arraystretch{.5}
\begin{array}{cccccccc}
1 & 0 & 0 & 0 & 0 & 0 & 0
\end{array}
 \right].
$$
Then $h(\phi) = \mathbf{h}^{\textrm{\tiny{T}}} \Phi(\phi) = 1$ as needed. Let
$$
\mathbf{g}^{\textrm{\tiny{T}}} = \left[
\arraycolsep=2pt\def\arraystretch{.5}
\begin{array}{ccccccc}
1 & 1 & 0 & \frac12 & 0 & \frac16
\end{array}
\right].
$$	
Then
$$
\mathbf{g}^{\textrm{\tiny{T}}} R(r) = \left[
\arraycolsep=2.4pt\def\arraystretch{.75}
\begin{array}{ccccccc}
1 & 1 & 0 & \frac12 & 0 & \frac16
\end{array}
\right]
\left[
\arraycolsep=2.4pt\def\arraystretch{.75}
\begin{array}{l}
1\\
r\\
2r^2 - 1\\
r^2 \\
3r^3 - 2r \\
r^3
\end{array} \right] = 1 + r + \frac12 r^2 + \frac 16 r^3
$$
which can be thought of as the approximation of $g(r) = e^r$ using terms of degree at most three. These specific vectors are then used in (\ref{eq:TensorRep}) to solve for the unknown $U.$ The minimum least squares solution using MATLAB is $\mathbf{x} = \textrm{psinv}(A) \mathbf{y}$ from which converting vector $\mathbf{x}$ to matrix $U$ we get
    $$
    U = \mathrm{vec2mat}(\mathbf{x})
    = \left[
    \arraycolsep=2.4pt\def\arraystretch{.75}
    \begin{array}{cccccc}
    1 & 0 & 0 & 0.25 & 0 & 0 \\
    0 & 0.2 & 0 & 0 & -0.4 & 0 \\
    0 & 0 & 0 & 0 & 0 & 0 \\
    0 & 0 & 0 & 0.25 & 0 & 0 \\
    0 & 0 & 0 & 0 & 0 & 0 \\
    0 & 0 & 0 & 0 & 0 & 0 \\
    0 & 0 & 0 & 0 & 0 & 0 \\
    \end{array}
    \right],
    $$
which in terms of %Zernike circle polynomials
Zernike polynomials is
    \begin{align*}
    \widetilde{u}(r, \phi) & = 1 + \frac14 r^2 + 0.2 r \cos \phi - 0.4(3r^3 - 2r) \cos \phi + \frac14 r^2 \cos 2\phi \\
    &= 1 + r \cos \phi + \frac14 r^2(1 + \cos 2\phi) - \frac65 r^3 \cos \phi.
    \end{align*}
Alternatively, using an $l_1$ optimization algorithm based on basis pursuit as explained in \cite{Romberg:2006}
% and \cite{Kristen:2006},
the solution matrix for $U$ is
      $$
    U = \mathrm{vec2mat}(\mathbf{x})
    = \left[
    \arraycolsep=2.4pt\def\arraystretch{.75}
    \begin{array}{cccccc}
    1 & 0 & 0 & 0.25 & 0 & 0 \\
    0 & 0 & 0 & 0 & -0.5 & 0 \\
    0 & 0 & 0 & 0 & 0 & 0 \\
    0 & 0 & 0 & 0.25 & 0 & 0 \\
    0 & 0 & 0 & 0 & 0 & 0 \\
    0 & 0 & 0 & 0 & 0 & 0 \\
    0 & 0 & 0 & 0 & 0 & 0 \\
    \end{array}
    \right],
    $$
which gives
    \begin{equation*}
    \widetilde{u}(r, \phi)  = 1 + \frac14 r^2(1 + \cos 2\phi) - \frac12 (3r^2 - 2r) \cos \phi
    = 1 + r \cos \phi + \frac14 r^2(1 + \cos 2\phi) - \frac32 r^3 \cos \phi.
    \end{equation*}
Both of these may be compared with the exact solution mentioned at the start of the example.
\end{example}
%

%\vspace{-5mm}

\begin{figure}
\centering
\includegraphics[width=.4\textwidth]{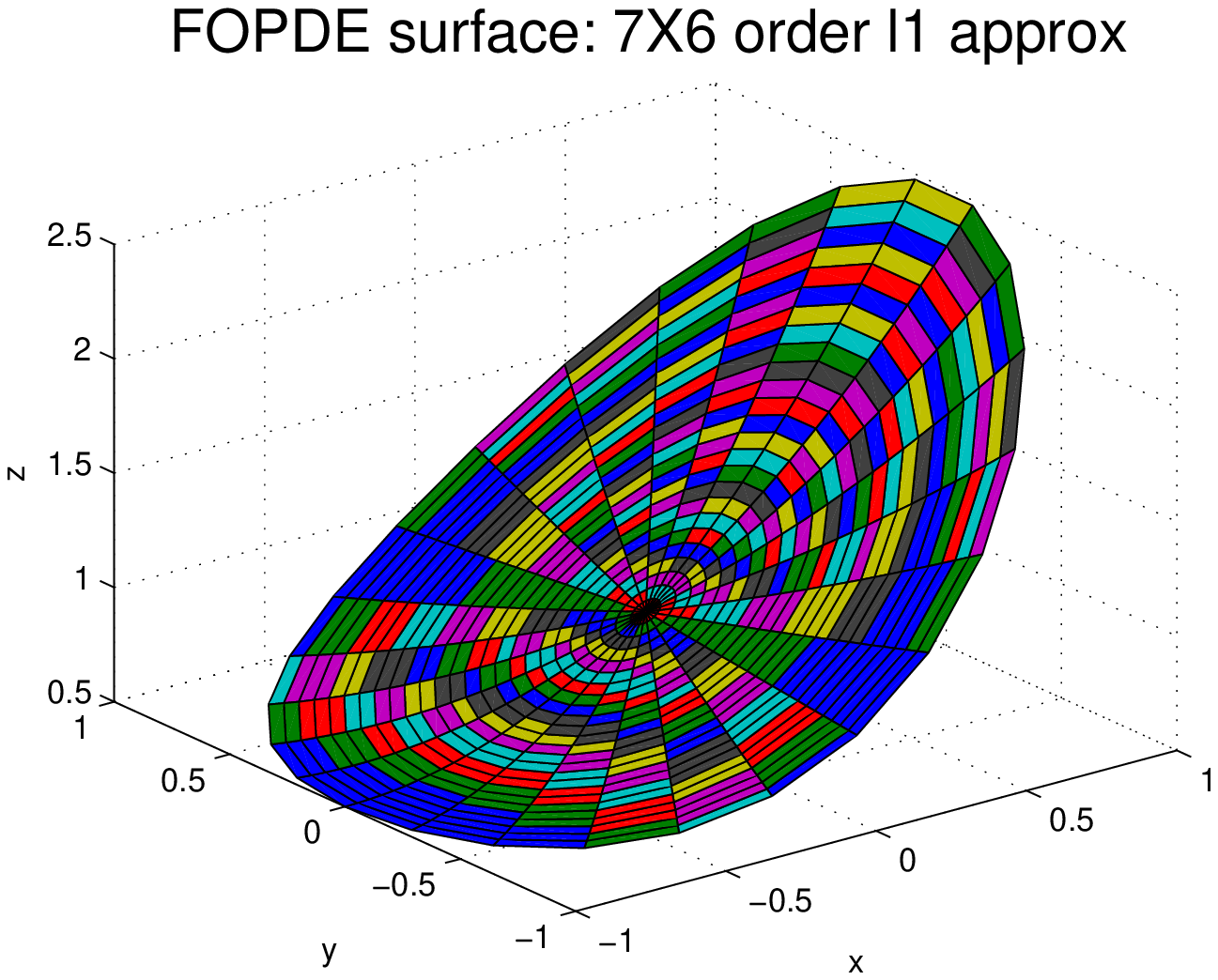}
\includegraphics[width=.4\textwidth]{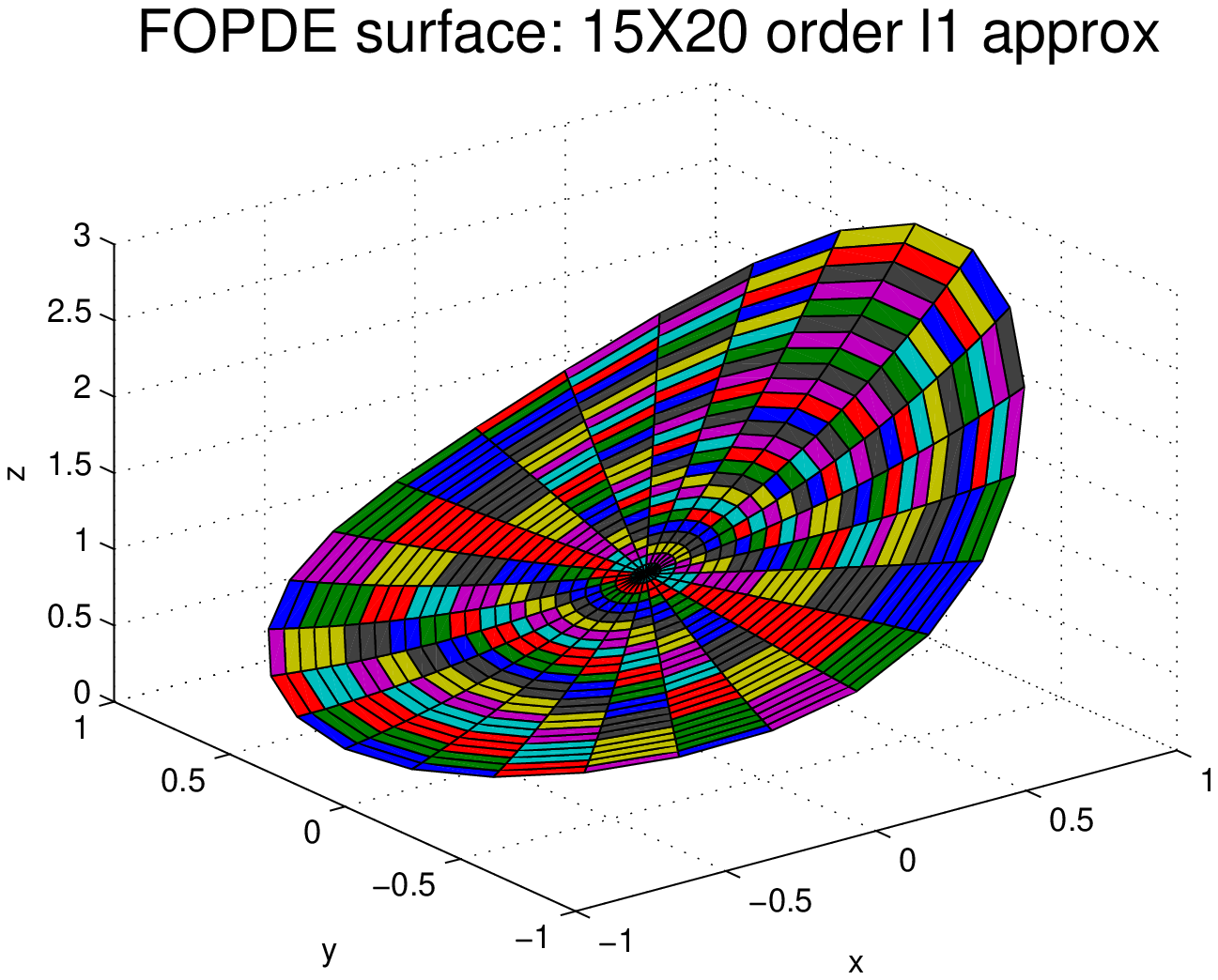}\\
\label{fig:figure2}
\includegraphics[width=.4\textwidth]{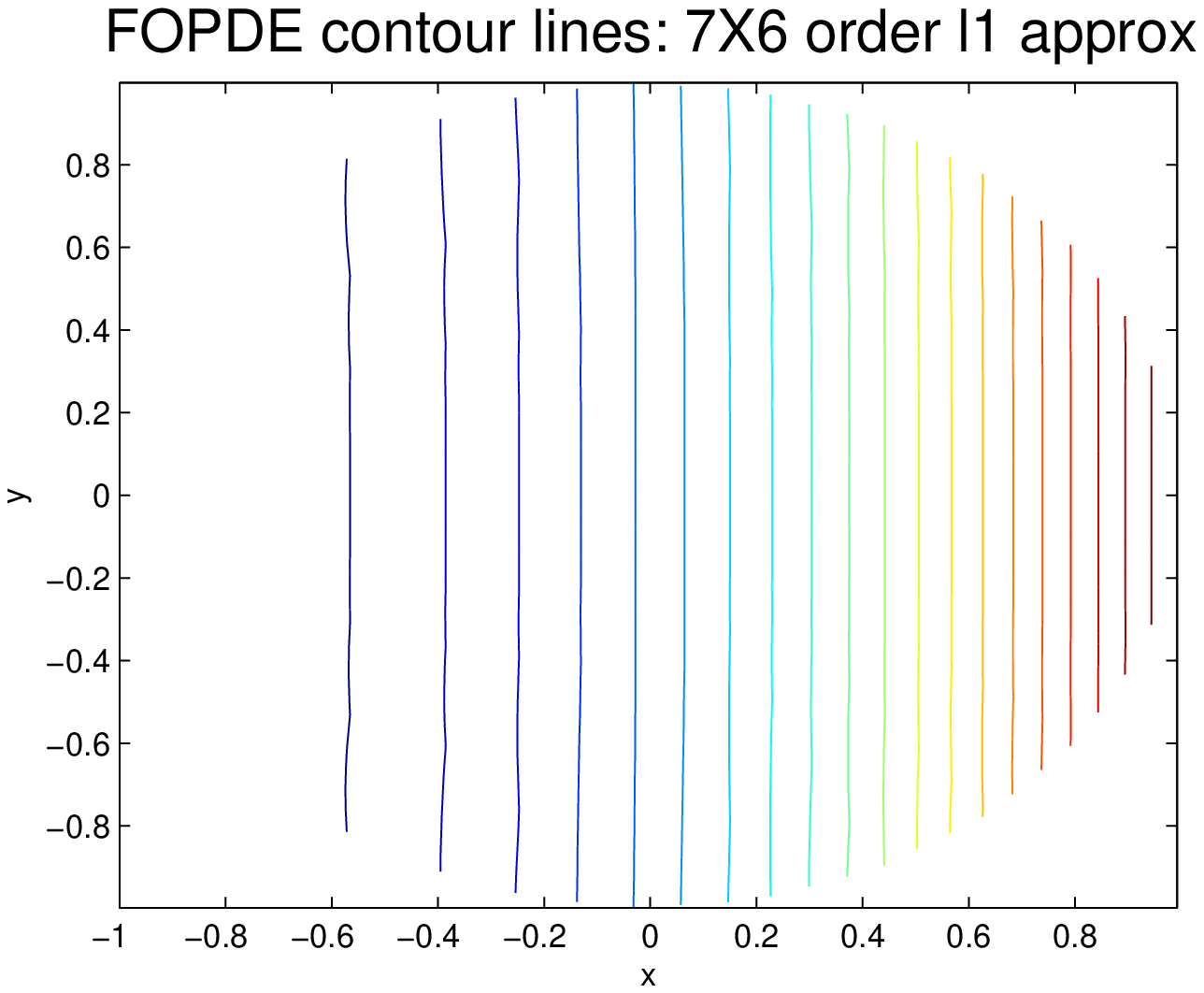}
\includegraphics[width=.4\textwidth]{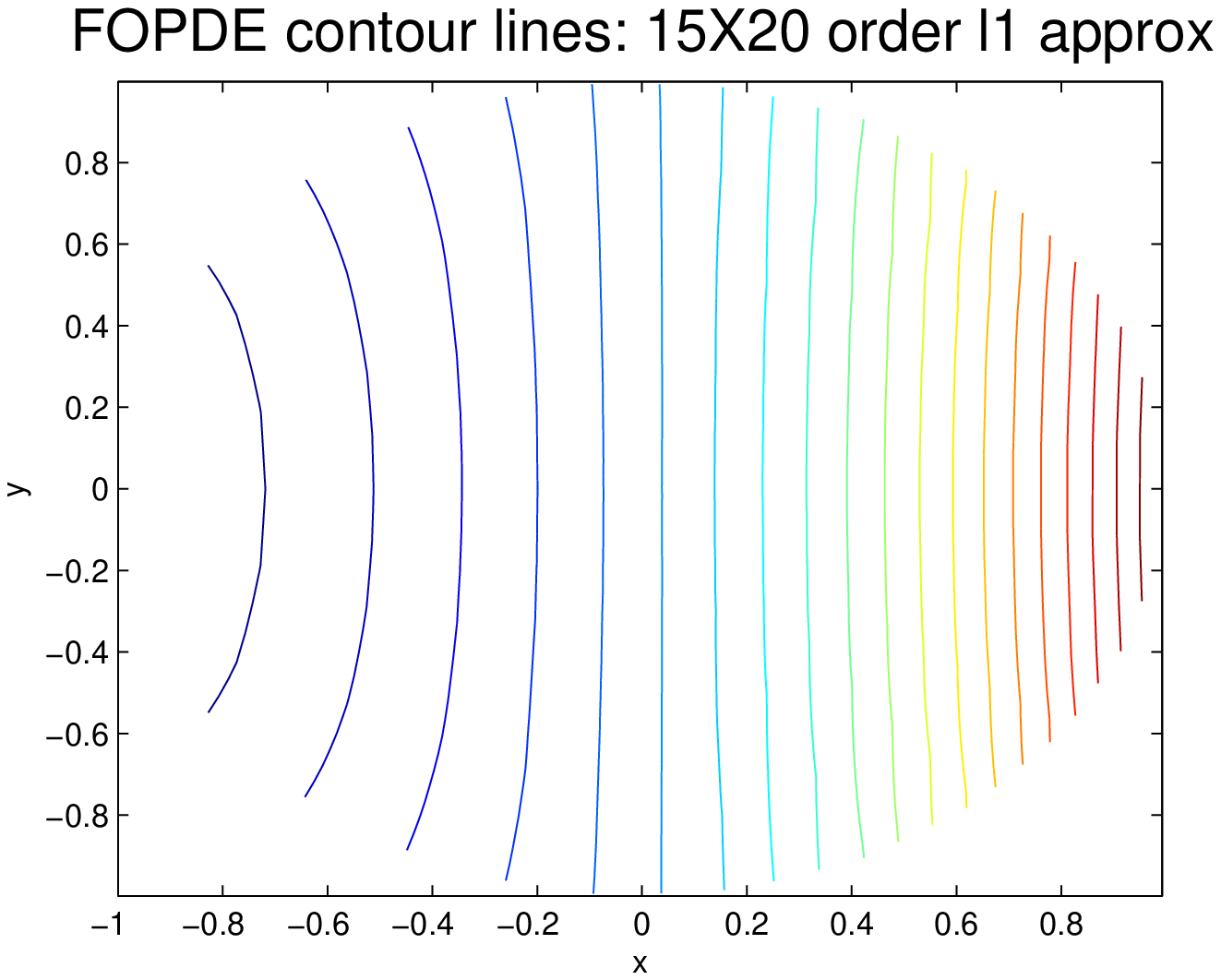}\\
\includegraphics[width=.4\textwidth]{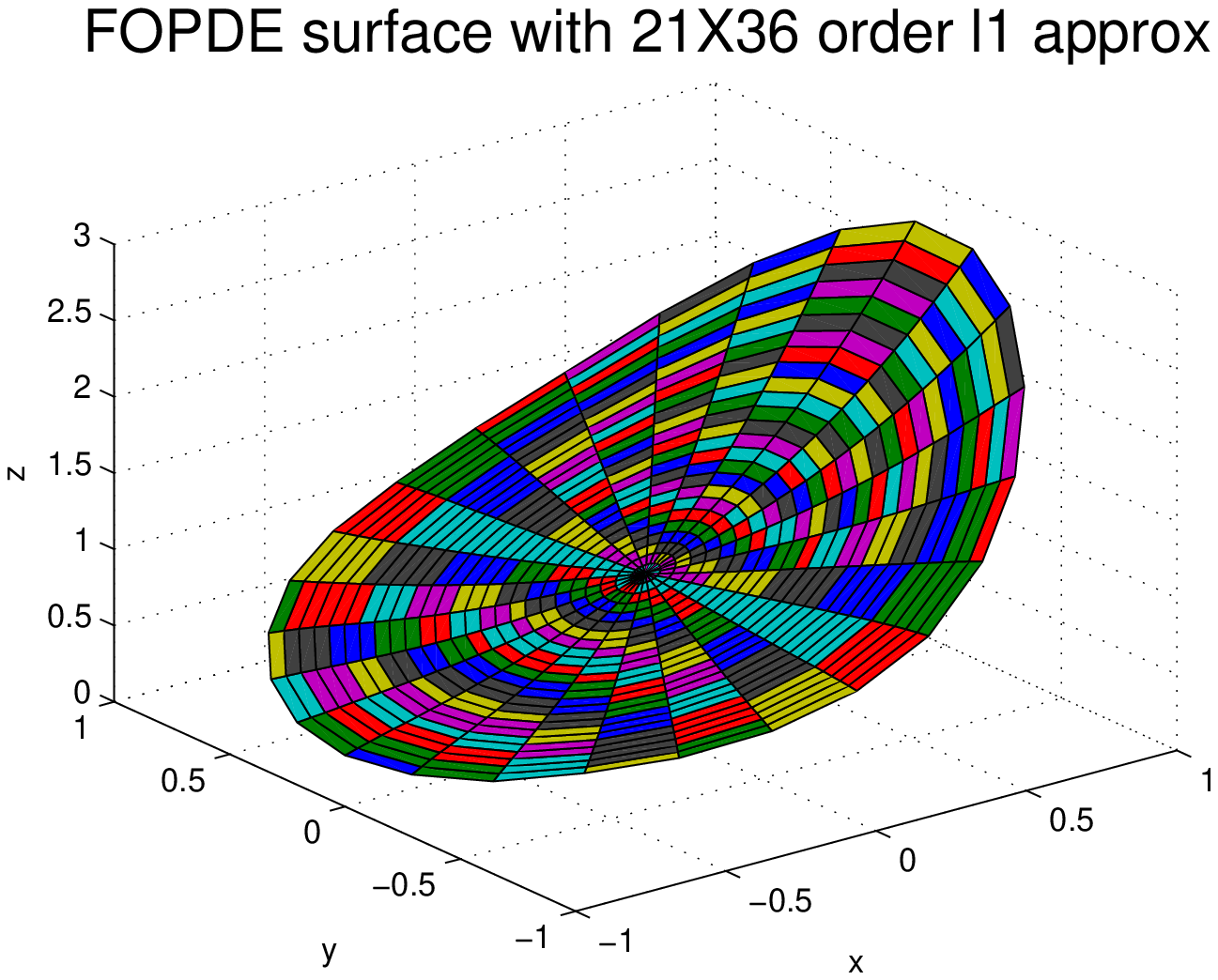}
\includegraphics[width=.4\textwidth]{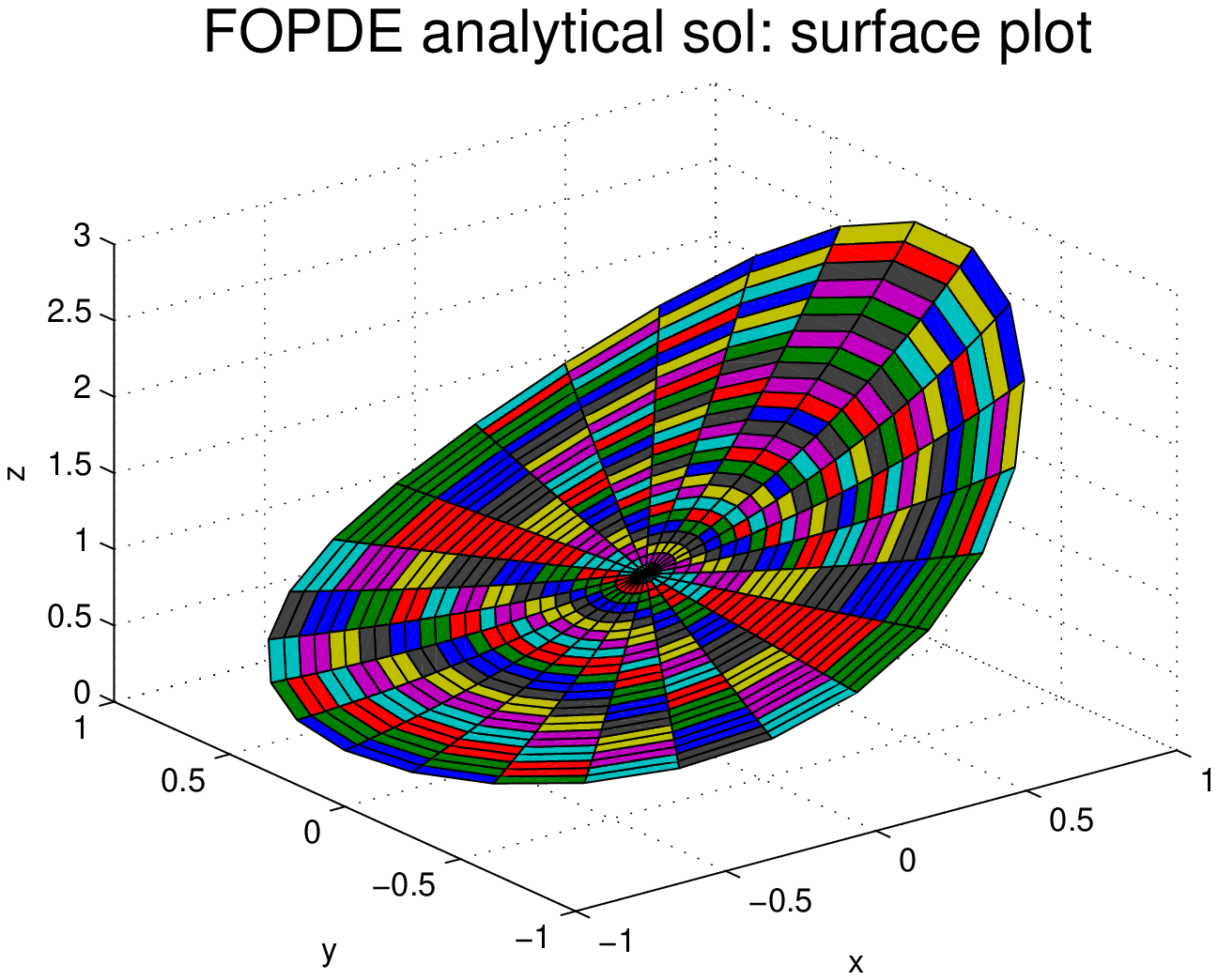}\\
\includegraphics[width=.4\textwidth]{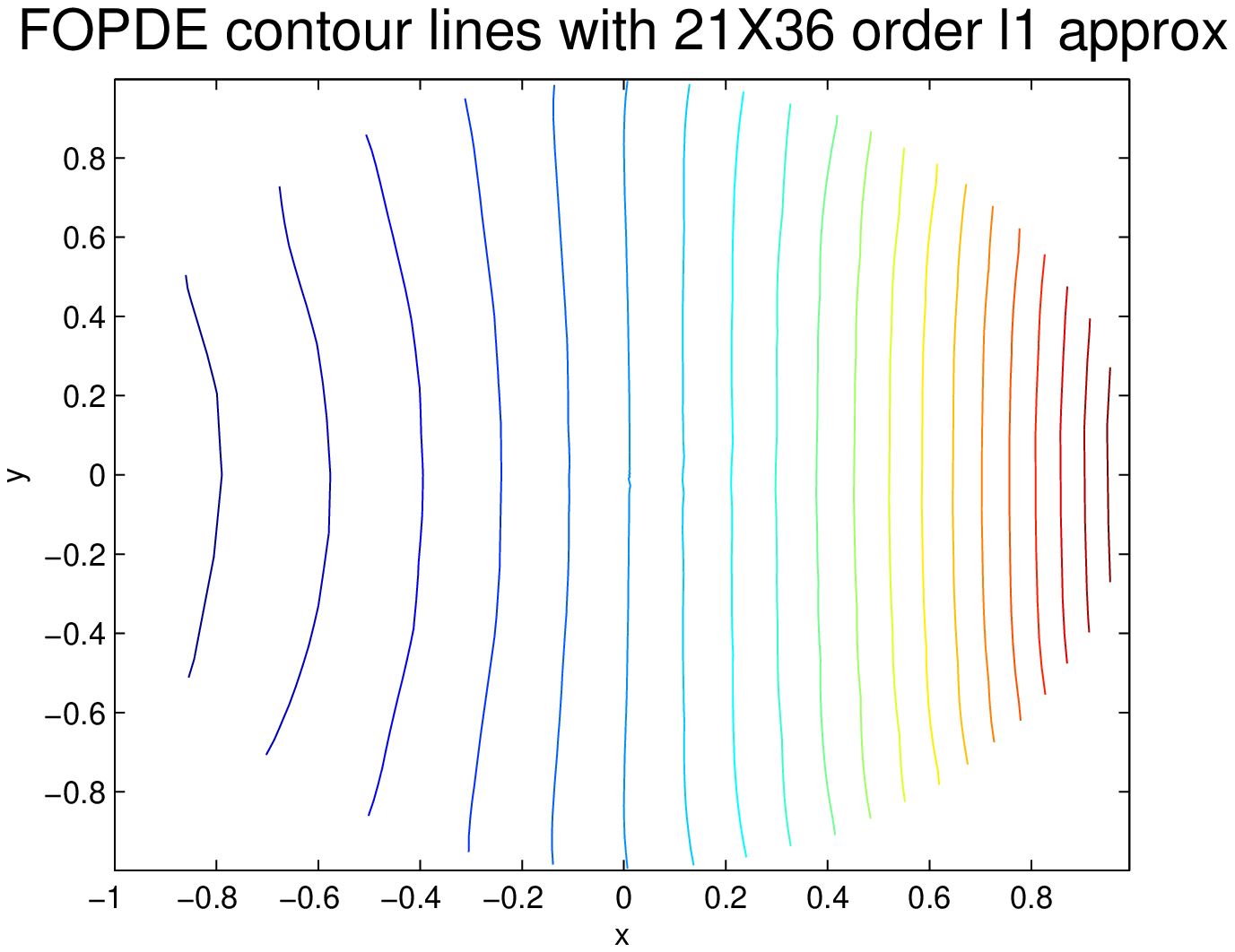}
\includegraphics[width=.4\textwidth]{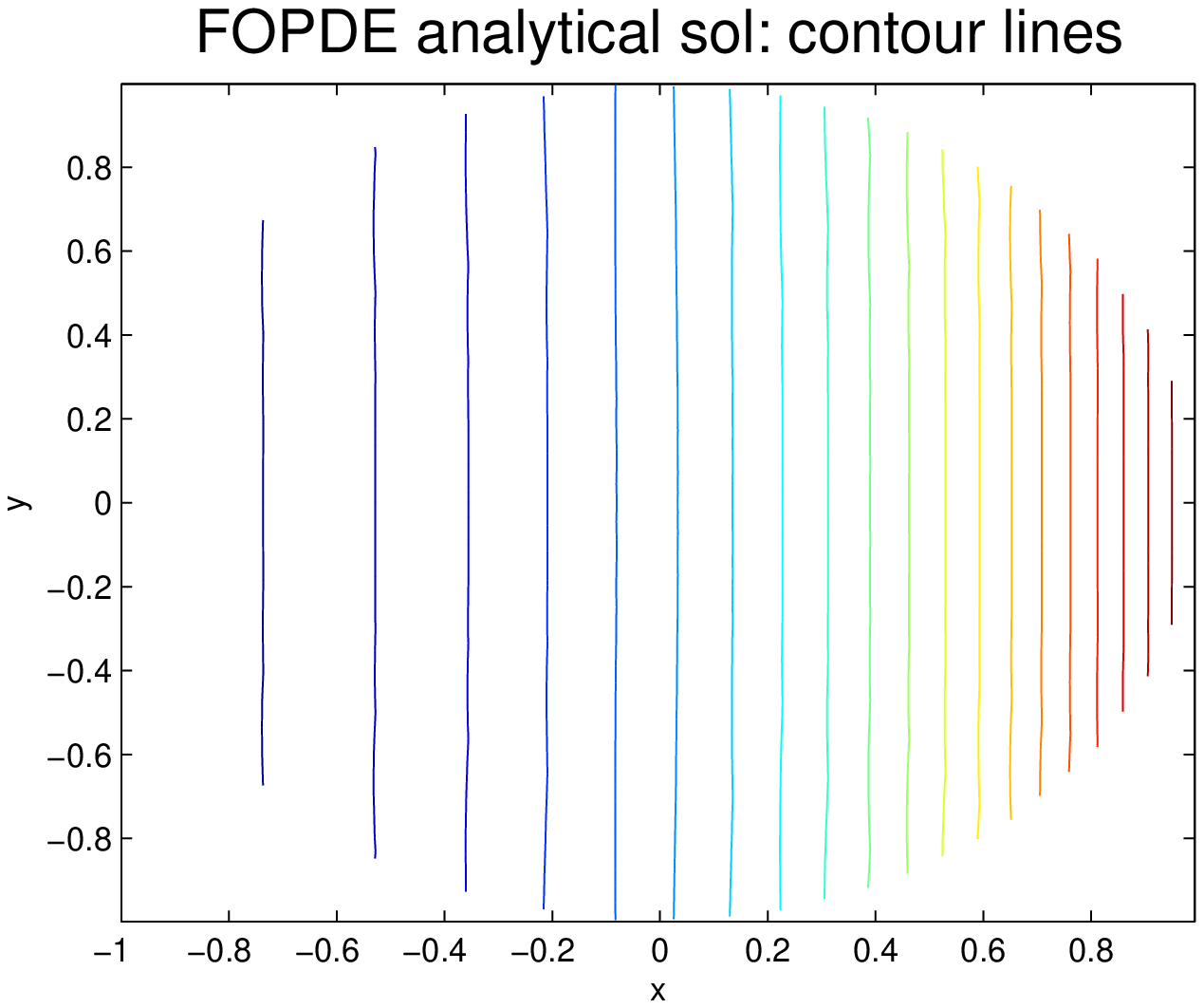}\\
\caption{FOPDE: Solution surfaces and contour lines.}
\label{fig:figure3}
\end{figure}
\begin{figure}
\centering
\includegraphics[width=.4 \textwidth]{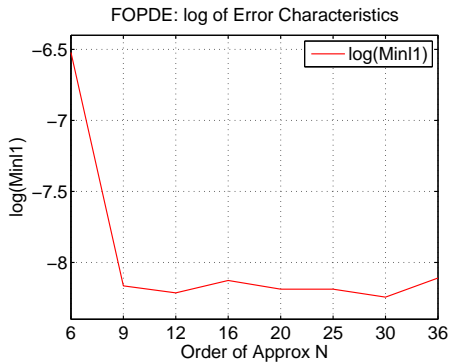} %\hfill
\caption{FOPDE solution: log of error curve with Lagrange approximation of higher order terms.}
\label{fig:figure4}
\end{figure}
%%%%%%%%%%%%%%%%%%%%%%%%%%%%%%%%%%%%%%%%%%%%%%%%%%%
\paragraph{\textbf{Error estimates}} Recall that the sizes of $\Phi(\phi)$ and $R(r)$ are $M$ and $N,$ respectively.
To study the error for different orders of approximation, the solution of FOPDE (\ref{FirstOrderPDE_Example}) has been determined numerically for the pair $(M,N)$ to be (7,6), (9,9), (11,12), (13,16), (15,20), (17,25), (19,30) and (21,36). The solution surfaces have then been compared with that generated by the exact solution.
To do this, one needs to solve (\ref{eq:Ax=y}): $A\mathbf{x}=\mathbf{b}.$ Recall that $A$ is a sparse matrix of order $MN,$ and $M = 21$ and $ N = 36$ for the highest order of approximation considered.
The surfaces and contour lines for values of $(M,N)$ equal to (7,6), (15,20) and (21,36) obtained by the minimum $l_1$-norm  solution are shown in Figure~\ref{fig:figure3} which may be compared with the actual solution surface and contour lines, also displayed in Figure~\ref{fig:figure3}. The surfaces provided by the minimum norm least squares solution are much inferior to the minimum $l_1$-norm  solution, and not shown here. The Mean Square Error (MSE) between the actual and the computed solution is given by the mathematical formula:
\begin{eqnarray*}
\text{MSE} =\tfrac{1}{m_1n_1} \sum^{m_1}_1  \sum^{n_1}_1 [X(i,j) - X_c(i,j)]^2 ,
\end{eqnarray*}
where $X(x,y)$ represents the actual solution surface, $X_c(x,y)$, the computed surface, and $m_1 \times n_1$ are the number of grid points on the surface. By comparing the minimum $l_1$-norm  solution with that of the
minimum norm least squares solution for different orders of approximation it is found that the minimum $l_1$-norm  solution is much superior as can be inferred from Table~\ref{FOPDE_error}. The log of the error in the minimum $l_1$-norm solution is shown in Figure~\ref{fig:figure4}. In Figure~\ref{fig:figure4}, the eight distinct points on the error curves correspond to the values of $(M,N)$ given at the start of this paragraph. The minimum $l_1$-norm  solution has less error when the higher order terms are projected on the space of lower degree polynomials by using Lagrange interpolation, see Remark~\ref{Rem:Approximation}.
\begin{table}[!t]
\caption{Errors in solving FOPDE}\label{FOPDE_error}
\vspace*{3mm}
\begin{center}
{\scriptsize
\begin{tabular}{|r|r|r|r|r|r|r|r|r|} \hline \label{tab:FOPSE_error}
Order of Zernike pol.& 7 $\times$ 6 & 9 $\times$ 9 & 11 $\times$ 12 & 13 $\times$ 16 & 15 $\times$ 20 & 17 $\times$ 25 & 19 $\times$ 30 & 21 $\times$ 36 \\
\hline
$\ell_2$-error (order $10^{-4}$) & 783.0700  & 13.2580  & 33.1340  & 24.0200  & 239.6800  & 3.1834  & 187.5100   & 5.8121 \\
\hline
$\ell_1$-error (order $10^{-4}$)  & 14.7180  & 2.8439  & 2.7084  & 2.9555  & 2.7794  & 2.7793  & 2.6274  & 3.0066
\\
\hline
\end{tabular}}
\end{center}
\end{table}
%
%%%%%%%%%%%%%%%%%%%%%%%%%%%%%%%%%%%%%%%%%%%%%%%%%%%%%%%%%%
\section{Solving second order partial differential equations}
\label{Sec:SecondOrderPDE}
%%%%%%%%%%%%%%%%%%%%%%%%%%%%%%%%%%%%%%%%%%%%%%%%%%%%%%%%%%
%
The general form for a linear second order partial differential equation (SOPDE) is
    \begin{equation} \label{EQ:Gen_SOPDE}
    a \frac{\partial^2 u}{\partial x^2} + b \frac{\partial^2 u}{\partial x \partial y} + c \frac{\partial^2 u}{\partial y^2} + a_1 \frac{\partial u}{\partial x} + a_2 \frac{\partial u}{\partial y} + a_0 u = f(x,y)
    \end{equation}
where $a,$ $b,$ $c,$ $a_1,$ $a_2,$ $a_0,$ and $f$ are continuous functions of $x$ and $y.$
Since we are motivated by problems involving SOPDEs arising in circular regions such as in a refracted wavefront through an optical system, we shall consider a SOPDE that is invariant under rotations of the coordinate axes about the origin. We have considered below a rotational invariant second order linear PDE with discontinuous BCs, and this special type includes Poisson's PDE appearing in physical systems. In addition to the above mentioned motivation, this choice is also for the purpose of demonstration. However, it is important to emphasize that the proposed method can be applied to PDEs that are not rotational invariant such as parabolic and hyperbolic PDEs, and following the procedure described below one can set up an algebraic equation as in (\ref{EQ:LinearSystem}) that will lead to the solution of the desired PDE. To demonstrate our method, we shall consider
\begin{equation} \label{SOPDE1}
\bigtriangleup u + \alpha \left(x \frac{\partial}{\partial x} + y \frac{\partial}{\partial y}\right)^2 u + \beta \left(x \frac{\partial}{\partial x} + y \frac{\partial}{\partial y}\right) u + \gamma u = f.
\end{equation}
In polar coordinates $r$ and $\phi,$ equation (\ref{SOPDE1}) becomes
\begin{equation} \label{SOPDE2}
(1 + \alpha r^2) \frac{\partial^2 u}{\partial r^2} + \left(\frac 1r + (\alpha + \beta)r \right) \frac{\partial u}{\partial r} + \frac{1}{r^2} \frac{\partial^2 u}{\partial \phi^2} + \gamma u = f.
\end{equation}
The aim is to solve this SOPDE by integration operational matrix using %Zernike circle polynomials
Zernike polynomials in the region $0 <r <r_0$ and $0 <\phi <2 \pi$ with the given boundary conditions,
%\begin{eqnarray*}
%$$
%g(\phi) = u(r_0,\phi),\;
%h(\phi) = {\partial u(r,\phi) \over \partial r} \bigg |_{r=r_0},\;
%p(r) = u(r,\phi_0), \;
%q(r) = {\partial u(r,\phi) \over \partial \phi}\bigg |_{\phi=\phi_0}.
%$$
$$
u(r_0,\phi) = g(\phi),\;
{\partial u(r,\phi) \over \partial r} \bigg |_{r=r_0} = h(\phi),\;
u(r,\phi_0) = p(r), \;
{\partial u(r,\phi) \over \partial \phi}\bigg |_{\phi=\phi_0} = q(r),
$$
%\end{eqnarray*}
where any of the functions may have discontinuities.
Multiplying both sides of (\ref{SOPDE2}) by $r^2$ gives
\begin{equation} \label{SOPDE3}
r^2(1 + \alpha r^2) \frac{\partial^2 u}{\partial r^2} +(r + \alpha r^3 + \beta r^3 ) \frac{\partial u}{\partial r} + \frac{\partial^2 u}{\partial \phi^2} +\gamma r^2  u =r^2 f.
\end{equation}
Integrating twice with respect to $r$ from $r_0$ to $r$, the first, third and sixth terms in the left side of (\ref{SOPDE3}), which do not contain the parameters $\alpha, \beta$ and $\gamma$, one gets
\begin{eqnarray*}
&& \int_{r_0}^r\int_{r_0}^r r^2{\partial^2u \over \partial r^2}\,(dr)^2 +\int_{r_0}^r\int_{r_0}^r r{\partial u \over \partial r}\,(dr)^2+\int_{r_0}^r\int_{r_0}^r {\partial^2u \over \partial \phi^2}\,(dr)^2 \notag \\
=&&r^2u(r,\phi)-r_0^2g(\phi)-3\int_{r_0}^r ru(r,\phi) \,dr - \int_{r_0}^r r_0^2h(\phi) \,dr \\
&&+\int_{r_0}^r r_0g(\phi)  \,dr +\int_{r_0}^r \int_{r_0}^r u(r,\phi)\,(dr)^2
+\int_{r_0}^r\int_{r_0}^r {\partial^2u \over \partial \phi^2}\,(dr)^2.
\end{eqnarray*}
Again, integrating the above expression with respect to $\phi$ twice from $\phi_0$ to $\phi$ gives
\begin{eqnarray}
&&\int_{\phi_0}^\phi \int_{\phi_0}^\phi r^2u(r,\phi)(d \phi)^2 -\int_{\phi_0}^\phi \int_{\phi_0}^\phi r_0^2g(\phi)(d \phi)^2 -3\int_{\phi_0}^\phi \int_{\phi_0}^\phi \int_{r_0}^r ru(r,\phi) \,dr (d \phi)^2 \notag \\
&&- \int_{\phi_0}^\phi \int_{\phi_0}^\phi \int_{r_0}^r r_0^2h(\phi) \,dr (d \phi)^2+ \int_{\phi_0}^\phi \int_{\phi_0}^\phi \int_{r_0}^r r_0g(\phi)  \,dr  (d \phi)^2 \notag  \\
&& + \int_{\phi_0}^\phi \int_{\phi_0}^\phi \int_{r_0}^r \int_{r_0}^r u(r,\phi)\,(dr)^2 (d \phi)^2
+\int_{\phi_0}^\phi \int_{\phi_0}^\phi \int_{r_0}^r\int_{r_0}^r {\partial^2u \over \partial \phi^2}\,(dr)^2 (d \phi)^2. \label{SOPDE10}
\end{eqnarray}
Expand the solution $u(r,\phi)$ of (\ref{SOPDE1}) in terms of %Zernike circle polynomials
Zernike polynomials up to some order $(m,n)$  where $n \ge m$ and $n-m$ is even. This gives an approximation of $u$ as $\widetilde{u}(r,\phi)=\Phi^T(\phi)\,U\,R(r),$ where $\Phi(\phi)$ is a matrix of size $M \times 1$, $R(r)$ is of size $N \times 1$, $M=2m+1$ and $N$ is the number of radial polynomials of degree less than or equal to $n$. Then each term of (\ref{SOPDE10}) has the following simplifications.
%
%\subsubsection{The First Term}
The first term in (\ref{SOPDE10}) can be written as \footnote{Strictly speaking, the resulting matrix representation is an approximation of some integral. However, we will always use the equality sign to express that the operators on the right side of the equality stand for corresponding integrals on the left side.}
\begin{eqnarray}
%&&
\int_{\phi_0}^\phi \int_{\phi_0}^\phi r^2u(r,\phi)(d \phi)^2
= \int_{\phi_0}^\phi \int_{\phi_0}^\phi r^2\,(\Phi^T(\phi)UR(r))(d\phi)^2
= \Phi^T(\phi)  E^T_{D\phi \phi_0} UM_R^{r^2}R(r),
\label{SOPDE11}
\end{eqnarray}
where $M_R^{r^2}$ is the matrix representation of $r^2R(r)$ with respect to $R(r)$, $E_{D\phi \phi_0}$ is the IOM of double integration of $\Phi(\phi)$ and is $E_{D\phi \phi_0}=E_{\phi \phi_0}^2$, and $E_{D\phi \phi_0}(1,1)=2\pi^2/3$. Using radial polynomials up to order $(3,3)$,
\begin{equation}
r^2 R(r) = \left[
\arraycolsep=2.4pt\def\arraystretch{.75}
\begin{array}{l}
r^2 \\
r^3 \\
2r^4-r^2 \\
r^4 \\
3r^5-2r^3\\
r^5 \\
  \end{array}
\right]
\approx
\left[
\arraycolsep=2.4pt\def\arraystretch{.75}
     \begin{array}{rrrrrr}
       0 & 0 & 0 & 1 & 0 & 0 \\
       0 & 0 & 0 & 0 & 0 & 1 \\
       0 & 0 & 0 &-1 & 0 & 0 \\
       0 & 0 & 0 & 0 & 0 & 0 \\
       0 & 0 & 0 & 0 & 0 & -2 \\
       0 & 0 & 0 & 0 & 0 & 0 \\
     \end{array}
   \right]
   \left[
   \arraycolsep=2.4pt\def\arraystretch{.75}
  \begin{array}{c}
1 \\
r \\
2r^2-1 \\
r^2 \\
3r^3-2r\\
r^3 \\
\end{array}
\right]=M_R^{r^2} R(r).\label{SOPDE12}
\end{equation}
As $R(r)$ contains radial polynomials of maximum degree 3, the terms $r^4$ and $r^5$ are ignored in the above. To obtain a better solution, we approximate $r^4$ and $r^5$ in terms of the set $\{1,r,r^2,r^3 \}$ by  Lagrange interpolation formula with equally spaced nodes (0,\,1/3,\,2/3,\,1) in the interval $0 \le r \le 1$ as,
\begin{eqnarray} \label{ApproxLagInterpolation}
r^4 & \approx & 2r^3-\tfrac{11}{9} r^2+\tfrac 29 r, \quad %\\
r^5  \approx  \tfrac 19 (25r^3-20r^2+4r),
\end{eqnarray}
and incorporate them in the IOM representation whenever they are encountered. Then,
\begin{eqnarray*}
2r^4-r^2 & \approx & \tfrac 49 r-\tfrac{31}{9} r^2+4r^3,\quad %\\
3r^5-2r^3   \approx  \tfrac 43r- \tfrac {26}{3}r^2+\tfrac {25}{3}r^3.
\end{eqnarray*}
Consequently, in (\ref{SOPDE12}), we upgrade $M_R^{r^2}$ to
%in the representation (\ref{SOPDE12}) for (3,3) order circle polynomials we change $M_R^{r^2}$ to
%is changed to
%\[r^2 R(r)= M_R^{r^2} R(r),\]
%where
    $$
    M_R^{r^2}
    =
    \left[
    \arraycolsep=2.4pt\def\arraystretch{.75}
     \begin{array}{rrrrrr}
       0 & 0        & 0 & 1              & 0 & 0 \\
       0 & 0        & 0 & 0              & 0 & 1 \\
       0 & \frac 49 & 0 &-\frac {31}{9}  & 0 & 4 \\
       0 & \frac 29 & 0 & -\frac {11}{9} & 0 & 2 \\
       0 & \frac 43 & 0 & -\frac {26}{3} & 0 & \frac {25}{3} \\
       0 & \frac 49 & 0 & -\frac {20}{9} & 0 & \frac {25}{9} \\
     \end{array}
    \right].
    $$
If these higher order terms are completely ignored (instead of considering their projection), then this approach of solving SOPDE using %Zernike circle polynomials
Zernike polynomials with IOM will fail.
\\
The second term in (\ref{SOPDE10}) is
\begin{align*}
&&-\int_{\phi_0}^\phi \int_{\phi_0}^\phi r_0^2g(\phi)(d \phi)^2% \\
%&=&
=-\int_{\phi_0}^\phi \int_{\phi_0}^\phi  \Phi^T(\phi) \mathbf{g} r_0^2(d \phi)^2 %\\
%&=
= - \Phi^T(\phi)  E^T_{D\phi \phi_0}  \mathbf{g}M_{r_0^2}^T\,R(r),
%\end{eqnarray*}
\end{align*}
where $g(\phi)=\Phi^T(\phi)\mathbf{g},$ %the representation
is the representation
of $u(r_0,\phi)=g(\phi)$ in terms of trigonometric functions, and $\mathbf{g}$ is an $M\times 1$ vector. $M_{r_0^2}$ is an $N\times 1$ vector with first element $r_0^2$ and others zero. Using %Zernike circle polynomials
Zernike polynomials up to degree three:
\begin{eqnarray*}
r^2_0 &=&\left[
\arraycolsep=2pt\def\arraystretch{.75}
  \begin{array}{llllll}
r_0^2 & 0 &
0 &
0 &
0 &
0
\end{array}
\right]\left[
\arraycolsep=2pt\def\arraystretch{.75}
  \begin{array}{c}
1 \\
r \\
2r^2-1 \\
r^2 \\
3r^3-2r\\
r^3 \\
\end{array}
\right]=M_{r_0^2}^TR(r).
\end{eqnarray*}
The third term in (\ref{SOPDE10}) is
\begin{eqnarray}
&-&3\int_{\phi_0}^\phi \int_{\phi_0}^\phi \int_{r_0}^r ru(r,\phi) \,dr (d \phi)^2 %\notag \\
%=&-& 3 \int_{\phi_0}^\phi \int_{\phi_0}^\phi \int_{r_0}^r \Phi^T(\phi)U\;rR(r)\,dr (d \phi)^2 \notag \\
%=&-&3
%
=-\Phi^T(\phi) E^T_{D\phi \phi_0} U\,E_{rr_0}^r R(r), \label{SOPDE13}
\end{eqnarray}
where $E_{rr_0}^r$ is the IOM of $rR(r)$ in which powers of $r$ higher than $n$ are included after approximating in terms of lower powers of $r,$ by the Lagrange interpolation formula with equally spaced nodes (0,\,1/3,\,2/3,\,1) in the interval $0 \le r \le 1$, as mentioned earlier.
\\
The fourth, fifth, and sixth terms in (\ref{SOPDE10}) are, respectively,
\begin{alignat}{2}
&- \int_{\phi_0}^\phi \int_{\phi_0}^\phi \int_{r_0}^r h(\phi)\,r_0^2 \,dr (d \phi)^2 &
%&=- \int_{\phi_0}^\phi \int_{\phi_0}^\phi \int_{r_0}^r \Phi^T(\phi) \mathbf{h}\,M^T_{r_0^2}\, R(r) \,dr (d \phi)^2 &
&=- \Phi^T(\phi)  E^T_{D\phi \phi_0}  \mathbf{h} M^T_{r_0^2} E_{rr_0} R(r), \label{SOPDE14}\\
&\int_{\phi_0}^\phi \int_{\phi_0}^\phi \int_{r_0}^r g(\phi)  \,r_0 \, dr  (d \phi)^2 &
%&=\int_{\phi_0}^\phi \int_{\phi_0}^\phi \int_{r_0}^r \Phi^T(\phi) \mathbf{g} M^T_{r_0}R(r) \, dr  (d \phi)^2 &
&=\Phi^T(\phi) E^T_{D\phi \phi_0} \mathbf{g} \,M^T_{r_0}\,E_{rr_0}R(r), \label{SOPDE15}\\
&\int_{\phi_0}^\phi \int_{\phi_0}^\phi \int_{r_0}^r \int_{r_0}^r u(r,\phi)\,(dr)^2 (d \phi)^2 &
%&=\int_{\phi_0}^\phi \int_{\phi_0}^\phi \int_{r_0}^r \int_{r_0}^r \Phi^T(\phi) U R(r)\,(dr)^2 (d \phi)^2 &
&= \Phi^T(\phi) E^T_{D\phi \phi_0} U E_{Drr_0} R(r) .
\label{SOPDE16}
\end{alignat}
The seventh term in (\ref{SOPDE10}) is
\begin{eqnarray}
 \int_{\phi_0}^\phi\int_{\phi_0}^\phi \int_{r_0}^r \int_{r_0}^r {\partial^2u \over \partial \phi^2}\,(dr)^2 \,(d\phi)^2
&=&\int_{r_0}^r \int_{r_0}^r [u(r,\phi)-u(r,\phi_0)]\,(dr)^2-\int_{\phi_0}^\phi \int_{r_0}^r \int_{r_0}^r q(r) (dr)^2 \,d\phi \notag\\
&=&\Phi^T(\phi)U  E_{Drr_0} R(r) -  \Phi^T(\phi)\mathbf{e}_1 \mathbf{p}^T  E_{Drr_0} R(r)  \label{SOPDE17} \\
&& \hspace{50mm} -  \Phi^T(\phi) E^T_{\phi \phi_0}     \mathbf{e}_1 \mathbf{q}^T  E_{Drr_0} R(r), \notag
\end{eqnarray}
where $\mathbf{e}_1=[1\;0\; \cdots \; 0]^T$ is of size $M \times 1$. Inserting the simplifications (\ref{SOPDE11})-(\ref{SOPDE17}) in (\ref{SOPDE10}), gives
\begin{eqnarray}
&&\Phi^T(\phi) [
 E^T_{D\phi \phi_0} UM_R^{r^2}
- E^T_{D\phi \phi_0}  \mathbf{g}M_{r_0^2}^T
-3 E^T_{D\phi \phi_0} U\,E_{rr_0}^r
 - E^T_{D\phi \phi_0}  \mathbf{h} (M_{r_0^2})^T E_{rr_0}
 + E^T_{D\phi \phi_0} \mathbf{g}(M_{r_0} )^TE_{rr_0}
  \notag \\
&& \phantom{+++++++}
+ E^T_{D\phi \phi_0} U E_{Drr_0}
+ I_{m}U  E_{Drr_0} -\mathbf{e}_1 \mathbf{p}^T  E_{Drr_0}  -E^T_{\phi \phi_0}\mathbf{e}_1 \mathbf{q}^T  E_{Drr_0} ]R(r).
\label{SOPDE20}
\end{eqnarray}
Integrating with respect to $r$ from $r_0$ to $r$, the second, fourth, and fifth terms in (\ref{SOPDE3}), which contain the parameters $\alpha$ and $\beta$, one gets
\begin{eqnarray} \label{SOPDE23}
&&\int_{r_0}^r \left (\alpha r^4 \frac{\partial^2 u(r,\phi)}{\partial r^2} +(\alpha+ \beta) r^3 \frac{\partial u(r,\phi)}{\partial r} \right )\,dr \notag \\
&=& \alpha r^4\, \frac{\partial u(r,\phi)}{\partial r}-\alpha r_0^4\, \frac{\partial u(r,\phi)}{\partial r}\bigg |_{r=r_0}+(-3\alpha+\beta) \int_{r_0}^r r^3\frac{\partial u(r,\phi)}{\partial r} \,dr\notag \\
&=&\alpha r^4\, \frac{\partial u(r,\phi)}{\partial r}-\alpha r_0^4\, h(\phi)+(-3\alpha+\beta)r^3 u(r,\phi)\notag \\
&&
%\hspace{18mm}
-(-3\alpha+\beta)r_0^3 u(r_0,\phi)- 3(-3\alpha+\beta)\int_{r_0}^r r^2u(r,\phi)\,dr .
\end{eqnarray}
Integrating again with respect to $r$ from $r_0$ to $r$ gives
\begin{equation} \label{SOPDE25}
\int_{r_0}^r \left ( \alpha r^4\, \frac{\partial u(r,\phi)}{\partial r} + (3\alpha-\beta)r_0^3 g(\phi) -\alpha r_0^4\, h(\phi)
 -(3\alpha-\beta)r^3 u(r,\phi) +3(3\alpha-\beta)\int_{r_0}^r r^2u(r,\phi)\right )\,dr.
\end{equation}
The first and fourth terms in (\ref{SOPDE25}) together give
\begin{eqnarray} \label{SOPDE27}
 &&\int_{r_0}^r \left ( \alpha r^4\, \frac{\partial u(r,\phi)}{\partial r} -(3\alpha-\beta)r^3 u(r,\phi) )\right )\,dr \notag \\
&=& \alpha r^4\, u(r,\phi) \bigg |_{r_0}^r-4 \alpha \int_{r_0}^r  r^3\, u(r,\phi)\,dr
-(3\alpha-\beta) \int_{r_0}^r r^3 u(r,\phi)\,dr \notag \\
&=& \alpha r^4\, u(r,\phi)-\alpha r_0^4\, u(r_0,\phi)-(7 \alpha-\beta) \int_{r_0}^r  r^3\, u(r,\phi)\,dr.
\end{eqnarray}
Hence, (\ref{SOPDE25}) takes the form
\begin{eqnarray} \label{SOPDE30}
&& \alpha r^4\, u(r,\phi)-\alpha r_0^4\, g(\phi)-(7 \alpha-\beta) \int_{r_0}^r  r^3\, u(r,\phi)\,dr
+ (3\alpha-\beta)r_0^3  \int_{r_0}^r g(\phi) \,dr \notag \\
&& - \alpha  \int_{r_0}^r  r_0^4\, h(\phi) \,dr
+3(3\alpha-\beta)\int_{r_0}^r \int_{r_0}^r r^2u(r,\phi)\,(dr)^2.
\end{eqnarray}
Again, integrating the above expression twice with respect to $\phi$ from $\phi_0$ to $\phi$,
\begin{alignat}{1} \label{SOPDE32}
%&&
&\alpha \int_{\phi_0}^\phi \int_{\phi_0}^\phi  r^4\, u(r,\phi)\, (d \phi)^2
-\alpha \int_{\phi_0}^\phi \int_{\phi_0}^\phi  r_0^4\, g(\phi)\, (d \phi)^2
%&&
-(7 \alpha-\beta) \int_{\phi_0}^\phi \int_{\phi_0}^\phi  \int_{r_0}^r  r^3\, u(r,\phi)\,dr \, (d \phi)^2 \notag \\
%
%&&
&+ (3\alpha-\beta)\int_{\phi_0}^\phi \int_{\phi_0}^\phi  \int_{r_0}^r r_0^3\, g(\phi) \,dr \, (d \phi)^2 %\notag \\
%&&
- \alpha \int_{\phi_0}^\phi \int_{\phi_0}^\phi  \int_{r_0}^r  r_0^4\, h(\phi) \,dr\, (d \phi)^2 \notag \\
%
%&&
&+3(3\alpha-\beta)\int_{\phi_0}^\phi \int_{\phi_0}^\phi  \int_{r_0}^r \int_{r_0}^r r^2u(r,\phi)\,(dr)^2\, (d \phi)^2.
%\end{eqnarray}
\end{alignat}
Write each term in (\ref{SOPDE32}) in terms of IOM noting that $u(r,\phi)=\Phi(\phi)^T\,U\,R(r)$ and $M_R^{r^4}$ is the matrix representation of $r^4R(r)$ with respect to $R(r)$. The first term in (\ref{SOPDE32}) is
\begin{equation} \label{SOPDE33}
\alpha \int_{\phi_0}^\phi \int_{\phi_0}^\phi  r^4\, u(r,\phi)\, (d \phi)^2
=\alpha \int_{\phi_0}^\phi \int_{\phi_0}^\phi  \Phi^T(\phi) U M_R^{r^4} R(r)  \,(d\phi)^2
=\alpha \Phi^T(\phi)E^T_{D\phi \phi_0}U M_R^{r^4}  R(r) .
\end{equation}
The second term in (\ref{SOPDE32}) is
\begin{equation} \label{SOPDE34}
-\alpha \int_{\phi_0}^\phi \int_{\phi_0}^\phi  r_0^4\, g(\phi)\, (d \phi)^2
=-\alpha \int_{\phi_0}^\phi \int_{\phi_0}^\phi  \Phi^T(\phi) \mathbf{g} M_R^{r_0^4} R(r)  \,(d\phi)^2
=-\alpha \Phi^T(\phi)E^T_{D\phi \phi_0} \mathbf{g} M_{r_0^4}^T R(r) .
\end{equation}
The third term in (\ref{SOPDE32}) is
\begin{eqnarray} \label{SOPDE36}
-(7 \alpha-\beta) \int_{\phi_0}^\phi \int_{\phi_0}^\phi  \int_{r_0}^r  r^3\, u(r,\phi)\,dr \, (d \phi)^2
&=& -(7 \alpha-\beta) \int_{\phi_0}^\phi \int_{\phi_0}^\phi  \int_{r_0}^r  \Phi^T(\phi) U(r^3R(r)) \, dr \,(d\phi)^2\notag \\
&=&-(7 \alpha-\beta) \Phi^T(\phi)E^T_{D\phi \phi_0} U E_{rr_0}^{r^3}R(r) .
\end{eqnarray}
The fourth term in (\ref{SOPDE32}) is
\begin{eqnarray} \label{SOPDE38}
(3\alpha-\beta)\int_{\phi_0}^\phi \int_{\phi_0}^\phi  \int_{r_0}^r r_0^3\, g(\phi) \,dr \, (d \phi)^2
&=&(3\alpha-\beta)\int_{\phi_0}^\phi \int_{\phi_0}^\phi \int_{r_0}^r \Phi^T(\phi) \mathbf{g} M_{r_0^3} R(r) \,dr\,(d\phi)^2\notag \\
&=& (3\alpha-\beta) \Phi^T(\phi) E^T_{D\phi \phi_0} \mathbf{g}M_{r_0^3} E_{rr_0}R(r) .
\end{eqnarray}
The fifth term in (\ref{SOPDE32}) is
\begin{eqnarray} \label{SOPDE40}
- \alpha \int_{\phi_0}^\phi \int_{\phi_0}^\phi  \int_{r_0}^r  r_0^4\, h(\phi) \,dr\, (d \phi)^2
&=& -\alpha \int_{\phi_0}^\phi \int_{\phi_0}^\phi \int_{r_0}^r \Phi^T(\phi) \mathbf{h}M_{r_0^4}R(r) \, dr \,(d\phi)^2 \notag \\
&=&-\alpha \Phi^T(\phi)E^T_{D\phi \phi_0} \mathbf{h}M_{r_0^4} E_{rr_0} R(r) .
\end{eqnarray}
The sixth term in (\ref{SOPDE32}) is
\begin{align} \label{SOPDE42}
3(3\alpha-\beta)\int_{\phi_0}^\phi \int_{\phi_0}^\phi  \int_{r_0}^r \int_{r_0}^r r^2u(r,\phi)\,(dr)^2\, (d \phi)^2
&= 3(3\alpha-\beta)\int_{\phi_0}^\phi \int_{\phi_0}^\phi \int_{r_0}^r \int_{r_0}^r \Phi^T(\phi) U\, r^2R(r)\, (dr)^2\,(d\phi)^2 \notag \\
&=3(3\alpha-\beta)\Phi^T(\phi) E^T_{D\phi \phi_0} UE_{Drr_0}^{r^2}  R(r) .
\end{align}
Combining the terms (\ref{SOPDE33})-(\ref{SOPDE42}), the term (\ref{SOPDE32}) takes the form,  in terms of IOMs,
\begin{align} \label{SOPDE44}
& \Phi^T(\phi)[\alpha E^T_{D\phi \phi_0}U M_R^{r^4} -\alpha E^T_{D\phi \phi_0} \mathbf{g} M_{r_0^4}^T %\notag \\
-(7 \alpha-\beta) E^T_{D\phi \phi_0} U E_{rr_0}^{r^3} +(3\alpha-\beta)  E^T_{D\phi \phi_0} \mathbf{g}M_{r_0^3} E_{rr_0} \notag \\
& \hspace{50mm} - \alpha E^T_{D\phi \phi_0} \mathbf{h}M_{r_0^4} E_{rr_0} + (3\alpha-\beta)(\phi) E^T_{D\phi \phi_0} UE_{Drr_0}^{r^2} ] R(r) .
\end{align}
The last two terms in (\ref{SOPDE3}) in terms of integration operational matrices after integrating twice with respect to $\phi$ from $\phi_0$ to $\phi$ and again twice with respect to $r$ from $r_0$ to $r$, are
\begin{eqnarray} \label{SOPDE46}
\gamma\int_{\phi_0}^\phi \int_{\phi_0}^\phi  \int_{r_0}^r  \int_{r_0}^rr^2  u(r,\phi)  \,(dr)^2\, (d \phi)^2
&=& \gamma\int_{\phi_0}^\phi \int_{\phi_0}^\phi  \int_{r_0}^r  \int_{r_0}^r \Phi^T(\phi) U (r^2\, R(r))  \,(dr)^2\, (d \phi)^2 \notag \\
&=& \gamma \Phi^T(\phi) E^T_{D\phi \phi_0} U E_{Drr_0}^{r^2}  R(r),
\end{eqnarray}
\begin{eqnarray}
\int_{\phi_0}^\phi \int_{\phi_0}^\phi  \int_{r_0}^r \int_{r_0}^r r^2 f(r,\phi) \,(dr)^2\, (d \phi)^2
&=& \int_{\phi_0}^\phi \int_{\phi_0}^\phi  \int_{r_0}^r \int_{r_0}^r \Phi^T(\phi) F (r^2\,R(r)) \,(dr)^2\, (d \phi)^2\notag \\
&=& \Phi^T(\phi) E^T_{D\phi \phi_0} F E_{Drr_0}^{r^2}  R(r) .  \label{SOPDE48}
\end{eqnarray}
Combining the terms (\ref{SOPDE20}), (\ref{SOPDE44}),  (\ref{SOPDE46}), and (\ref{SOPDE48}),  the equation (\ref{SOPDE2}) finally takes the form
\begin{align} \label{SOPDE54}
&\Phi^T(\phi) [
 E^T_{D\phi \phi_0} UM_R^{r^2}
- E^T_{D\phi \phi_0}  \mathbf{g}M_{r_0^2}^T
-3 E^T_{D\phi \phi_0} U\,E_{rr_0}^r  - E^T_{D\phi \phi_0}  \mathbf{h} (M_{r_0^2})^T E_{rr_0} \notag \\
&
 + E^T_{D\phi \phi_0} \mathbf{g}(M_{r_0} )^TE_{rr_0}
+ E^T_{D\phi \phi_0} U E_{Drr_0}
+I_{m}U  E_{Drr_0} -\mathbf{e}_1 \mathbf{p}^T  E_{Drr_0}  -E^T_{\phi \phi_0}\mathbf{e}_1 \mathbf{q}^T  E_{Drr_0} +\alpha E^T_{D\phi \phi_0}U M_R^{r^4}\notag \\
& -\alpha E^T_{D\phi \phi_0} \mathbf{g} M_{r_0^4}^T
-(7 \alpha-\beta) E^T_{D\phi \phi_0} U E_{rr_0}^{r^3} +(3\alpha-\beta)  E^T_{D\phi \phi_0} \mathbf{g}M_{r_0^3} E_{rr_0}
-\alpha E^T_{D\phi \phi_0} \mathbf{h}M_{r_0^4} E_{rr_0} \notag \\
&  %\hspace{30mm}
 + (3\alpha-\beta) E^T_{D\phi \phi_0} UE_{Drr_0}^{r^2}
+\gamma  E^T_{D\phi \phi_0} U E_{Drr_0}^{r^2}]  R(r)
=\Phi^T(\phi) [E^T_{D\phi \phi_0} F E_{Drr_0}^{r^2}]R(r).
\end{align}
Based on (\ref{SOPDE54}), define matrices $A$ and $Y$ as,
\begin{eqnarray}
A &=& (M_R^{r^2})^T \otimes E^T_{D\phi \phi_0} -3 (E_{rr_0}^r)^T \otimes E^T_{D\phi \phi_0}  + (E_{Drr_0})^T \otimes E^T_{D\phi \phi_0} +(E_{Drr_0})^T \otimes I_{m}  +  \alpha (M_R^{r^4})^T \otimes E^T_{D\phi \phi_0} \nonumber \\
&&  -(7 \alpha-\beta) (E_{rr_0}^{r^3})^T \otimes E^T_{D\phi \phi_0} + (3\alpha-\beta) (E_{Drr_0}^{r^2} )^T \otimes E^T_{D\phi \phi_0} +\gamma  (E_{Drr_0}^{r^2})^T E^T_{D\phi \phi_0}, \label{SOPDE60A}\\
%\end{eqnarray*}
%%
%\begin{eqnarray*}
 Y &=& E^T_{D\phi \phi_0} F E_{Drr_0}^{r^2}+E^T_{D\phi \phi_0}  \mathbf{g}M_{r_0^2}^T +E^T_{D\phi \phi_0}  \mathbf{h} (M_{r_0^2})^T E_{rr_0} -E^T_{D\phi \phi_0} \mathbf{g}(M_{r_0} )^TE_{rr_0}  +\mathbf{e}_1 \mathbf{p}^T  E_{Drr_0} \nonumber \\
&&
+E^T_{\phi \phi_0}\mathbf{e}_1 \mathbf{q}^T  E_{Drr_0}
+\alpha E^T_{D\phi \phi_0} \mathbf{g} M_{r_0^4}^T
-(3\alpha-\beta)  E^T_{D\phi \phi_0} \mathbf{g}M_{r_0^3} E_{rr_0}
+\alpha E^T_{D\phi \phi_0} \mathbf{h}M_{r_0^4} E_{rr_0} \label{SOPDE61A}.
\end{eqnarray}
For $Y$ given in (\ref{SOPDE61A}), define the vector $\mathbf{b}=\text{vec}(Y).$ Let $\mathbf{x}=\text{vec}(U)$ be the vector that is unknown. Using the matrix $A$ in (\ref{SOPDE60A}), (\ref{SOPDE54}) can be written as
\begin{equation} \label{EQ:LinearSystem}
A\mathbf{x}=\mathbf{b}.
\end{equation}
Therefore, the solution of the PDE in (\ref{SOPDE1}) can be found by solving a system of linear equations given by (\ref{EQ:LinearSystem}) in which $A$ is a sparse matrix of order $MN.$ The solution $\mathbf{x}$, an $MN \times 1$ matrix, is then reshaped as an $M\times N$ matrix $U$ which gives
$\widetilde{u}(r,\phi)=\Phi^T(\phi)UR(r)$, an approximate solution of (\ref{SOPDE1}).

\bigskip

\noindent
For the Laplace equation, $\alpha=\beta=\gamma=0,$ and also $F=0.$ In this case (\ref{SOPDE60A}) and (\ref{SOPDE61A}) simplify to, respectively,
\begin{align}
%\hspace{10mm}
A &= (M_R^{r^2})^T \otimes E^T_{D\phi \phi_0} -3 (E_{rr_0}^r)^T \otimes E^T_{D\phi \phi_0}  + (E_{Drr_0})^T \otimes E^T_{D\phi \phi_0} +(E_{Drr_0})^T \otimes I_{m}, \label{SOPDE60} \\
 Y &= E^T_{D\phi \phi_0}  \mathbf{g}M_{r_0^2}^T +E^T_{D\phi \phi_0}  \mathbf{h} (M_{r_0^2})^T E_{rr_0}
-E^T_{D\phi \phi_0} \mathbf{g}(M_{r_0} )^TE_{rr_0} +\mathbf{e}_1 \mathbf{p}^T  E_{Drr_0} %\notag \\
%&&\hspace{40mm}
+E^T_{\phi \phi_0}\mathbf{e}_1 \mathbf{q}^T  E_{Drr_0}. \label{SOPDE61}
\end{align}
%
%Define, $\mathbf{x}=\text{Vec}(U)$ and $\mathbf{b}=\text{Vec}(Y)$ so that (\ref{SOPDE54}) can be written as
%\begin{equation} \label{EQ:LinearSystem}
%A\mathbf{x}=\mathbf{b},
%\end{equation}
%where $\mathbf{x}$ is unknown. Therefore, the solution of the PDE in (\ref{SOPDE1}) turns out to be a solution of the linear simultaneous equation in (\ref{EQ:LinearSystem}) in which $A$ is a sparse matrix of order $MN$ with the coefficients obtained from (\ref{SOPDE60}).
%
%
%
To demonstrate the % efficacy
accuracy of this method, in Example~\ref{EX:SOPDE} below, the corresponding linear equation obtained in (\ref{EQ:LinearSystem}) is solved in two ways . In one the standard matrix pseudo-inverse of $A$ is used, in which case $\mathbf{x}=A^\dagger\mathbf{b},$ and in the other an $l_1$-optimization algorithm is used. Both have been implemented using Matlab.
%
%
%%%%%%%%%%%%%%%%%%%%%%%%%%%%%%%%%%%%%%%%%%%%%%%%%%%
%\section{Numerical Solution of a Second Order PDE}
%%%%%%%%%%%%%%%%%%%%%%%%%%%%%%%%%%%%%%%%%%%%%%%%%%%
\begin{example}[Numerical Solution of a Second Order PDE]
\label{EX:SOPDE}
\rm
Consider the second order PDE,
\begin{eqnarray}
r^2{\partial^2u \over \partial r^2}+r{\partial u \over \partial r}+{\partial^2u \over \partial \phi^2}=0,\label{SOPDE56}
\end{eqnarray}
 in a circular region of radius $a$ subject to the boundary conditions (BCs)
\begin{eqnarray*}
g(\phi)&:=&u(r_0,\phi)=\left \{
\begin{array}{ll}
V_0, & 0 <\phi <\pi;\\
0,   & \pi <\phi <2\pi;  \end{array}   \right .\label{SOPDE58} \\
h(\phi) &=& {\partial u(r,\phi) \over \partial r} \bigg |_{r=a}% \notag \\
=\frac{V_0}{\pi a \sin \phi}; \\
p(r) &=& u(r,\phi_0) = V_0 \left \{\frac 12+\frac 1\pi \tan^{-1} \frac{2ar \sin \phi_0}{(a^2-r^2)} \right \};\notag \\
q(r) &=& {\partial u(r,\phi) \over \partial \phi}\bigg |_{\phi=\phi_0}=
\frac{V_0}{\pi}\frac{2ar\cos \phi_0 (a^2-r^2)}{(a^2-r^2)^2+4a^2r^2\sin^2\phi_0}.
\end{eqnarray*}
This is a Laplace equation in polar coordinates and %is encountered
appears in determining the potential distribution in a horizontal cylindrical region with axial symmetry when the upper half is maintained at a potential $V_0$ and the lower half at zero potential.
Without any loss of generality, assume that $V_0=1$, $r_0=a=1$, and $\phi_0=0$. So, the BCs in terms of the
%Zernike circle polynomials
Zernike polynomials are,
\begin{eqnarray*}
g(\phi) &=& u(r_0,\phi) ={1 \over 2}+{2 \over \pi}\sum_{i=0}^\infty {\sin (2i+1)\phi \over (2i+1)}%\notag \\
%&=&
=\mathbf{g}^T\Phi(\phi), \notag \\
h(\phi) &=& {\partial u(r,\phi) \over \partial r} \bigg |_{r=r_0}=\frac{2}{\pi}\sum_{i=0}^\infty \sin (2i+1)\phi %\notag \\
%&=&
=\mathbf{h}^T\Phi(\phi),\notag \\
p(r) &=& u(r,\phi_0)=\tfrac{1}{2} =\mathbf{p}^TR(r), \notag \\
q(r) &=& {\partial u(r,\phi) \over \partial \phi}\bigg |_{\phi=\phi_0}=\frac{2r}{\pi(1-r^2)}=\mathbf{q}^TR(r), \notag
\end{eqnarray*}
where $\mathbf{p}= [\tfrac{1}{2}\;0\; 0\; 0\;0\;0 \cdots ]^T,$ $\mathbf{q}= \tfrac 2\pi [ 0\;1 \; 0\;0\;0\;1 \cdots]^T$,
$\mathbf{g}=\tfrac 2\pi [\pi/4\; 0\;1 \; 0\;0\;0\;1/3 \cdots ]^T$, and $\mathbf{h}=\tfrac 2\pi [0\; 0\;1 \; 0\;0\;0\;1 \cdots ]^T$.
\begin{table}[!t]
\caption{Density of the sparse matrix $A$}\label{table_sparsity}
\begin{center}
\begin{tabular}{|c|c|c|c|}
\hline
\text{Order of Zernike pol.} & \text{Order of } A & \text{Non-zero elements in } A & \text{Sparsity of} A\\
$M \times N$ & $MN \times MN$ & $x$ & $x/M^2N^2$\\
\hline
$7 \times 6$ &  $42 \times 42$ & 502 & 0.2846 \\
\hline
$9 \times 9$ & $81 \times 81$ & 1275 &  0.1943 \\
\hline
$11 \times 12$ &  $132 \times 132$  & 2680  & 0.1538 \\
\hline
$13 \times 16$ & $208 \times 208$ & 5102 & 0.1179 \\
\hline
$15 \times 20$ & $300 \times 300$ & 8745 & 0.0972 \\
\hline
$17 \times 25$ & $425 \times 425$ & 14306 & 0.0792 \\
\hline
$19 \times 30$ & $570 \times 570$ & 22009 & 0.0677 \\
\hline
$21 \times 36$ & $756 \times 756$ &33048 & 0.0578 \\
\hline
\end{tabular}
\end{center}
\end{table}
With the above BCs, the solution of the Laplace equation (\ref{SOPDE56}) is obtained numerically using (\ref{SOPDE60}) and (\ref{SOPDE61}) for values (7,6), (9,9), (11,12), (13,16), (15,20), (17,25), (19,30) and (21,36) of the pair $(M,N)$. The Zernike polynomial based solution is compared with the exact solution and the errors are listed in Table~\ref{tab:SOPSE_error} for $l_1$ as well as $l_2$ methods, the $l_1$-errors being much lower. It appears from Figure~\ref{fig:SOfigure5} that the error decreases exponentially with the increase of ZRP degrees. The solution surfaces and their contour lines for only three values (7,6), (15,20), and (21,36) of  $(M,N)$ are shown in Figures~\ref{fig:SOfigure3}-\ref{fig:SOfigure4} for the Zernike polynomial based solutions, along with the solution by the product method, and the exact Poisson integral solution from where one can compare the different solutions. The eight distinct points on the error curve in Figure~\ref{fig:SOfigure5} correspond to the values of $(M,N)$ as mentioned above. It appears that among the approximate solutions, the $l_1$ method is better, and the accuracy improves with higher order terms. The proposed method compares favorably with the exact solution particularly if $(M,N)=(21,36)$. To justify the use of an $l_1$-algorithm for sparse solutions, the sparsity, i.e. the density of the sparse matrix $A$ defined as the number of non-zero elements in $A$ is computed, and these densities corresponding to the order of the %Zernike circle polynomials
Zernike polynomials used are shown in Table~\ref{table_sparsity}. We reiterate that if the higher order terms are not approximated by projecting on the space spanned by the lower order terms, the solution in this second order case is nowhere near the true solution.
\end{example}
\begin{figure}
\centering
\includegraphics[width=.4\textwidth]{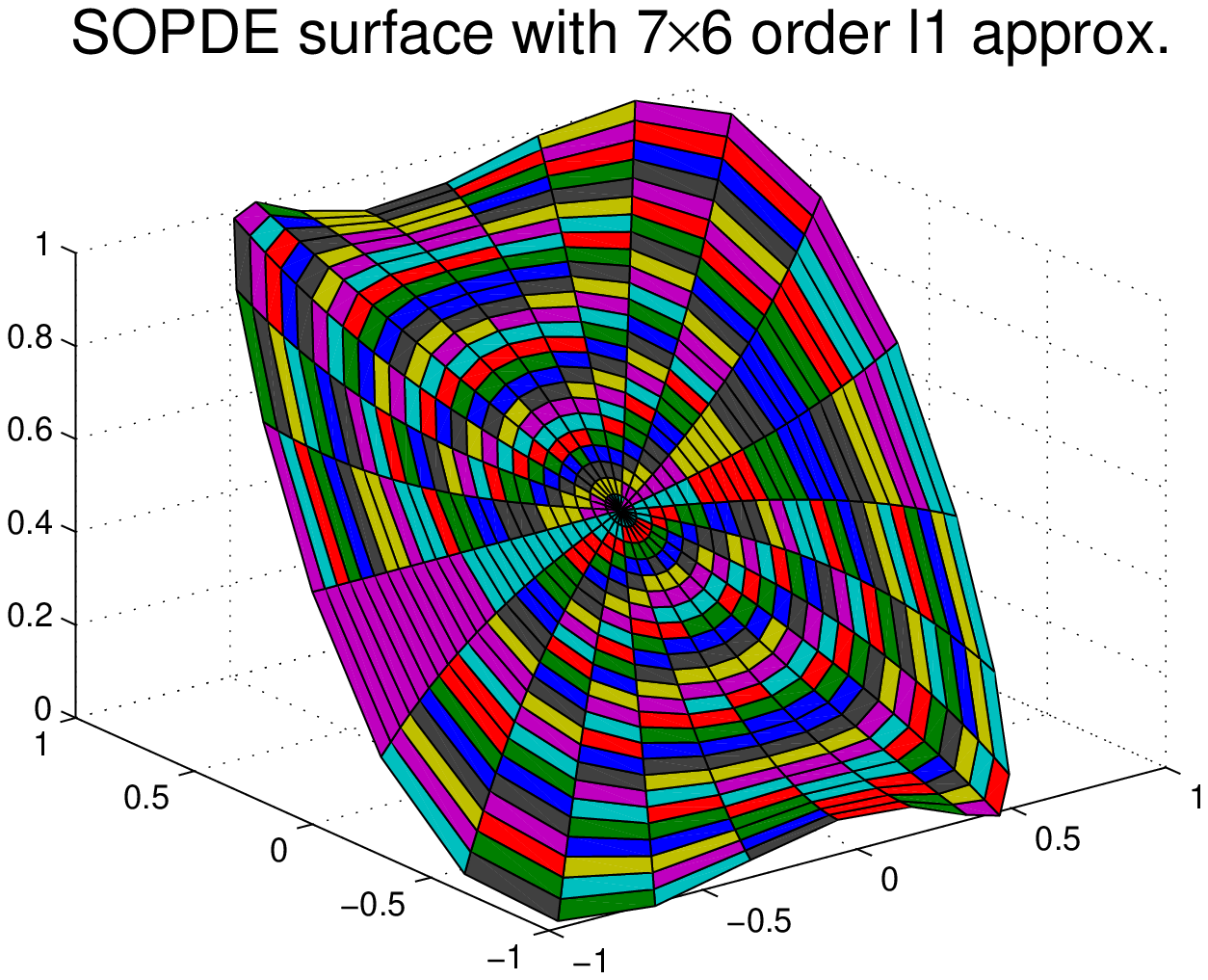}\hfill
\includegraphics[width=.4\textwidth]{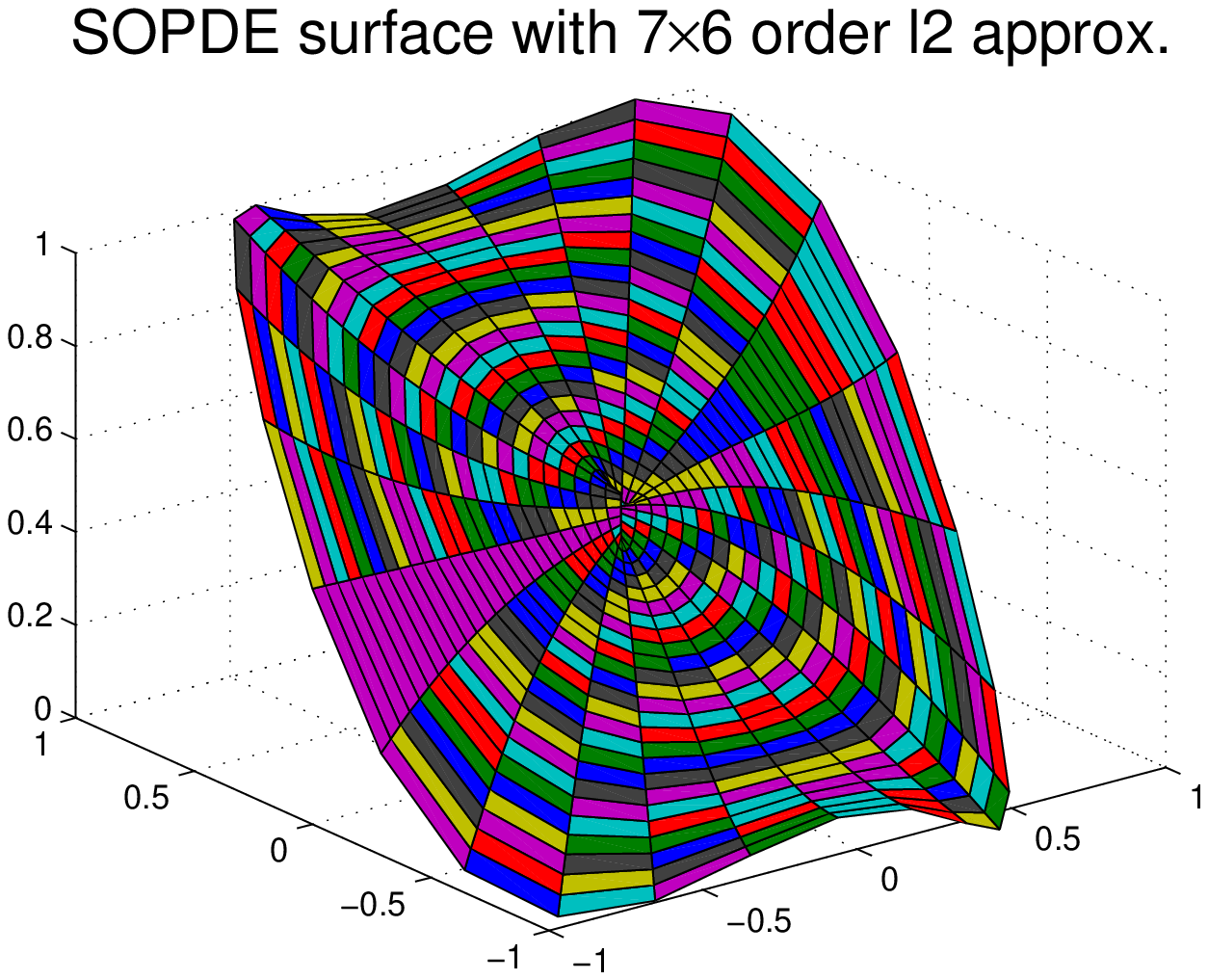}\\
\includegraphics[width=.4\textwidth]{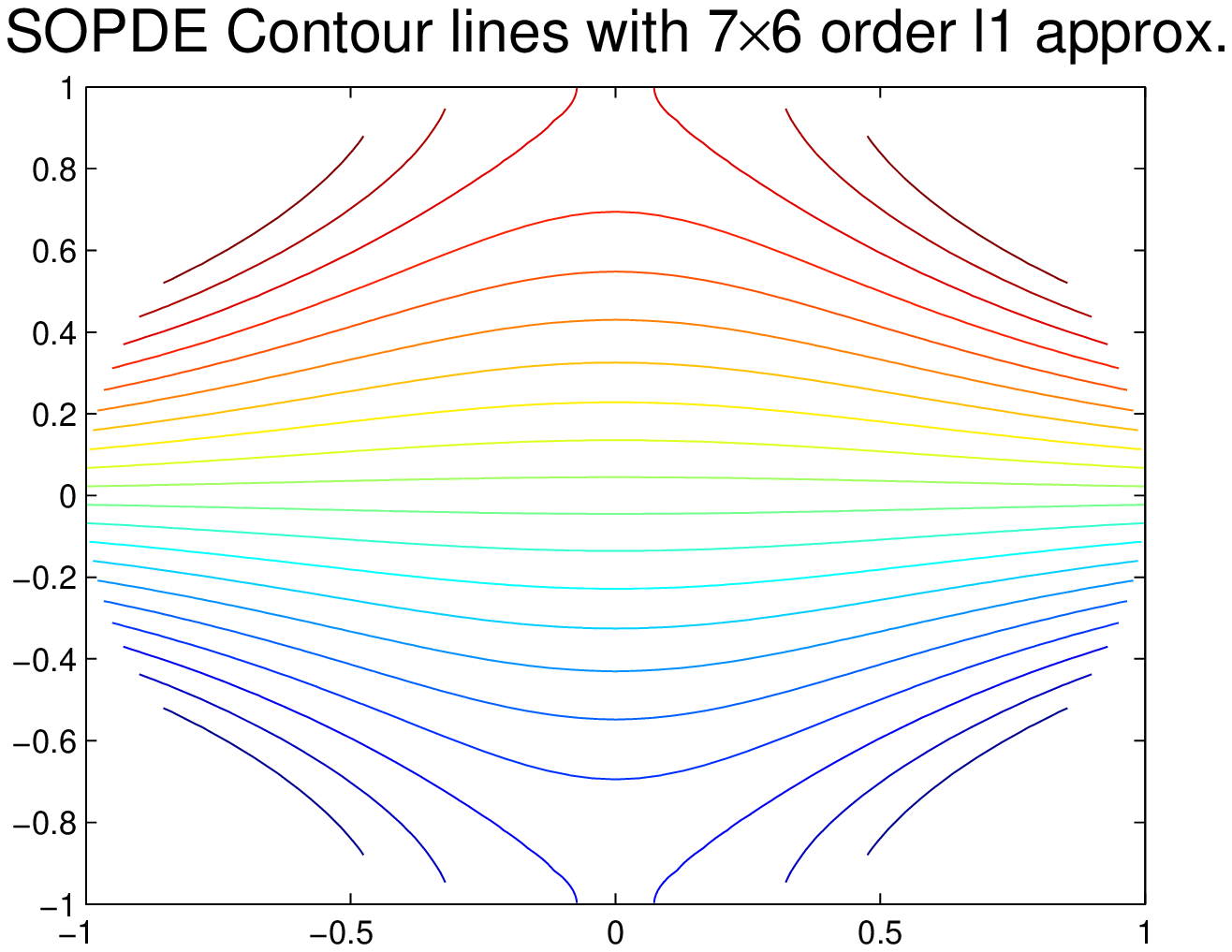}\hfill
\includegraphics[width=.4\textwidth]{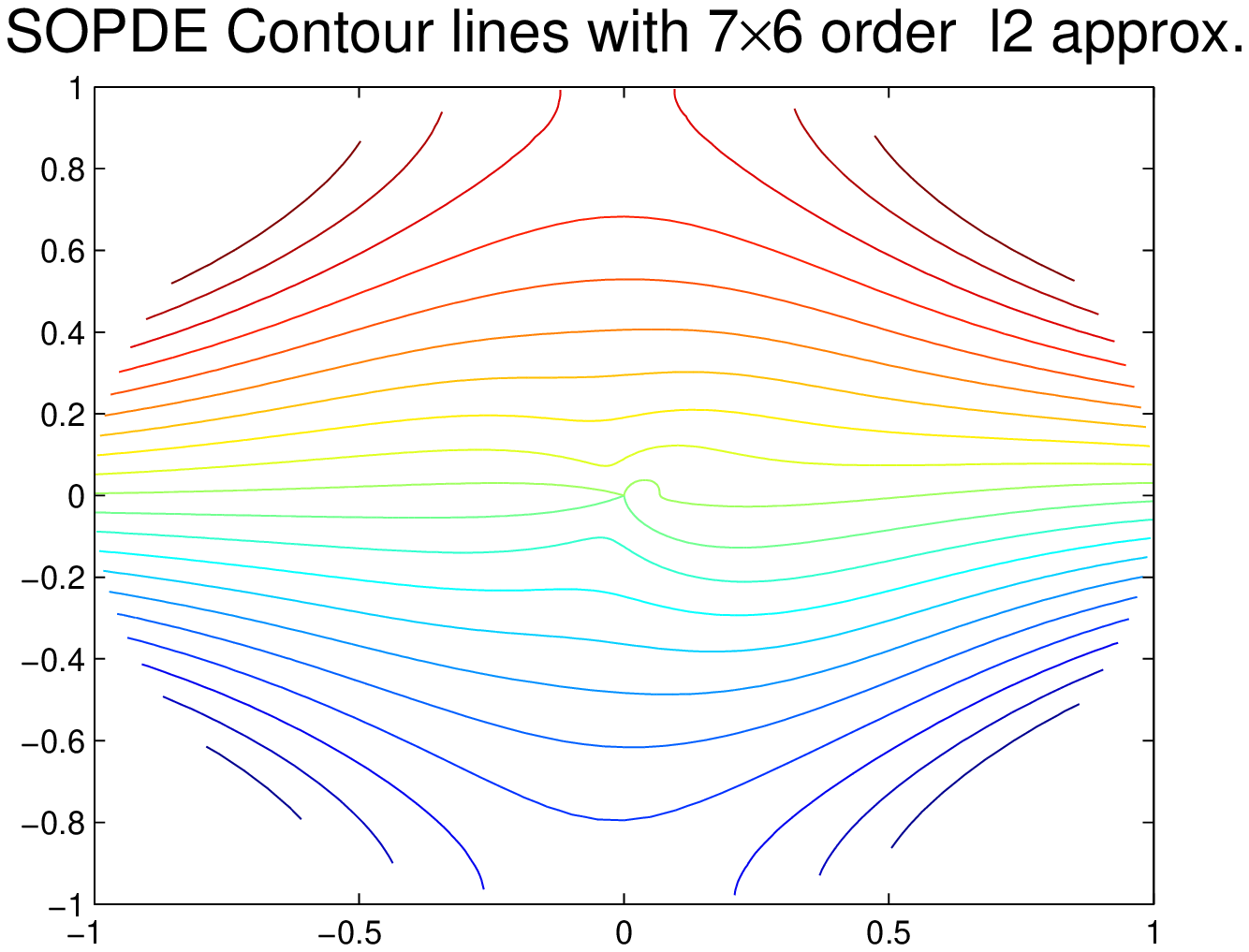}\\
\centering
\includegraphics[width=.4\textwidth]{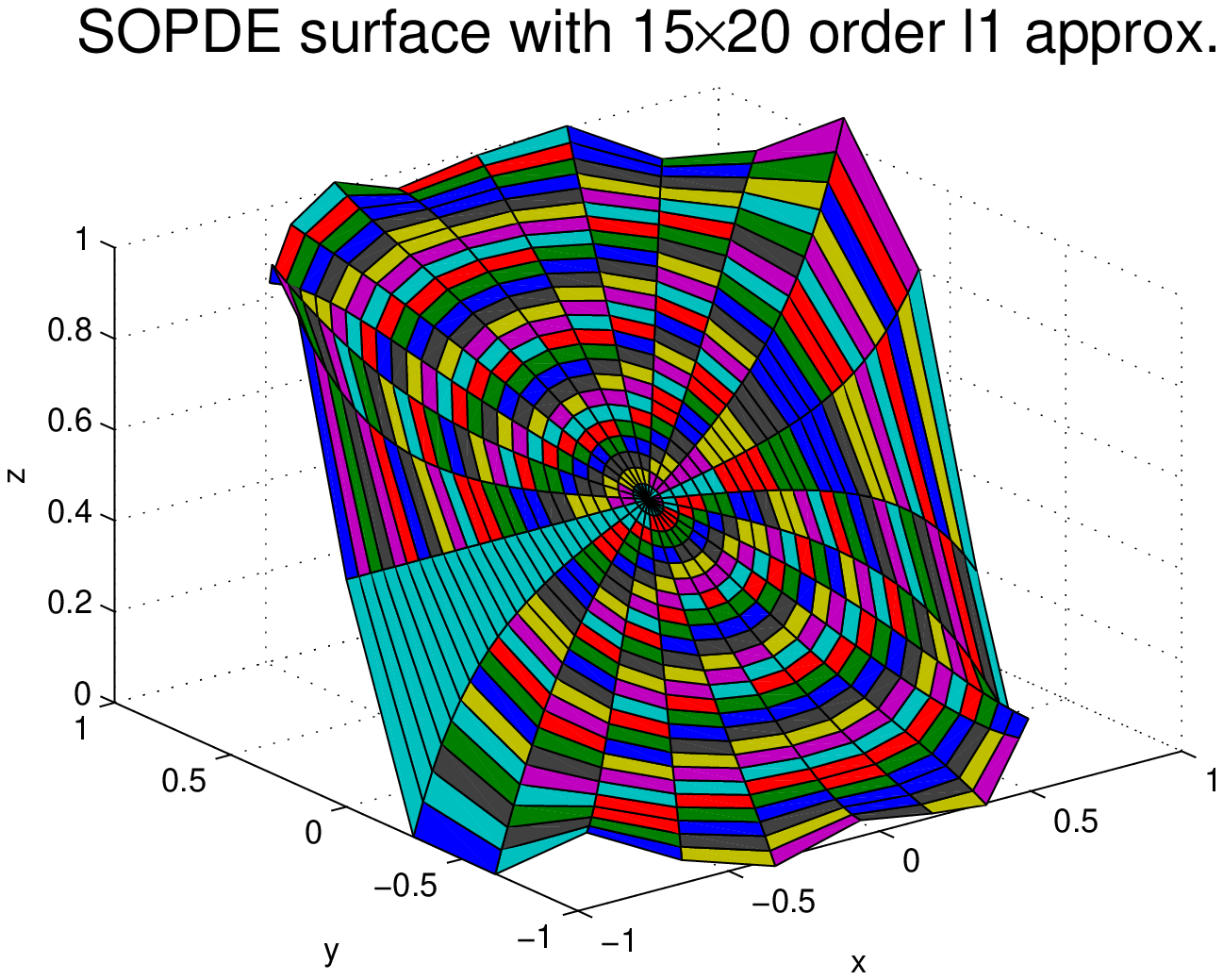}\hfill
\includegraphics[width=.4\textwidth]{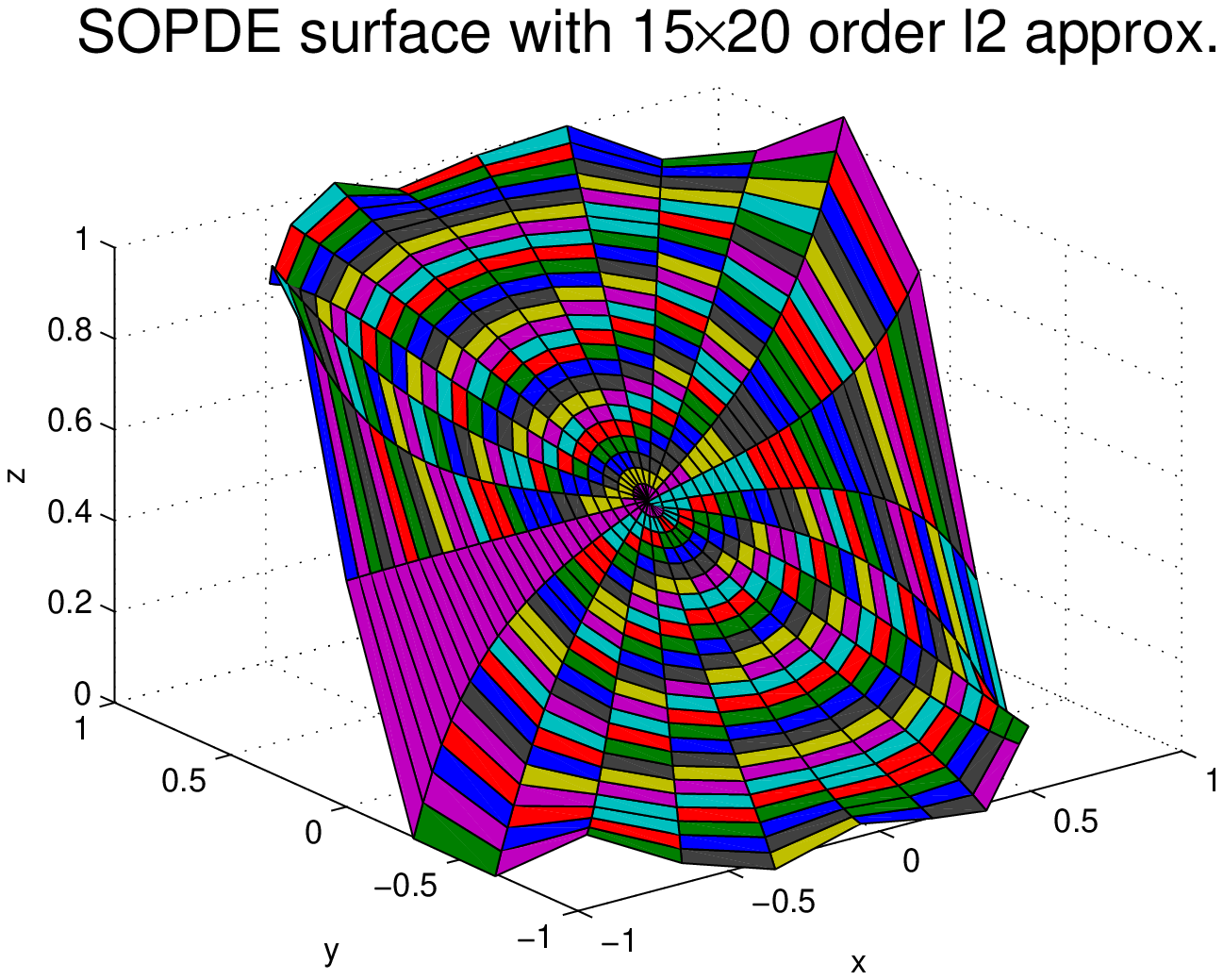}\\
\includegraphics[width=.4\textwidth]{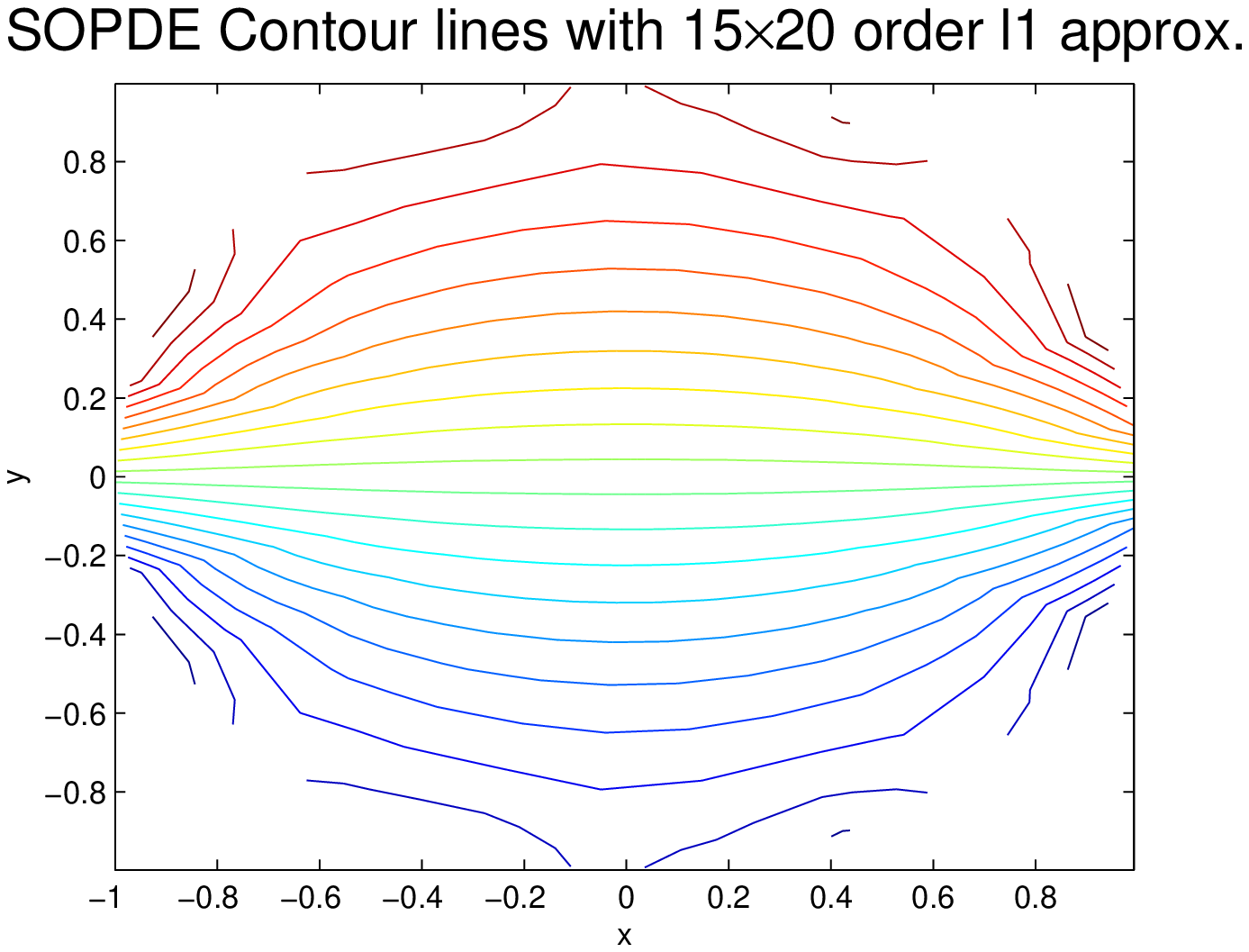}\hfill
\includegraphics[width=.4\textwidth]{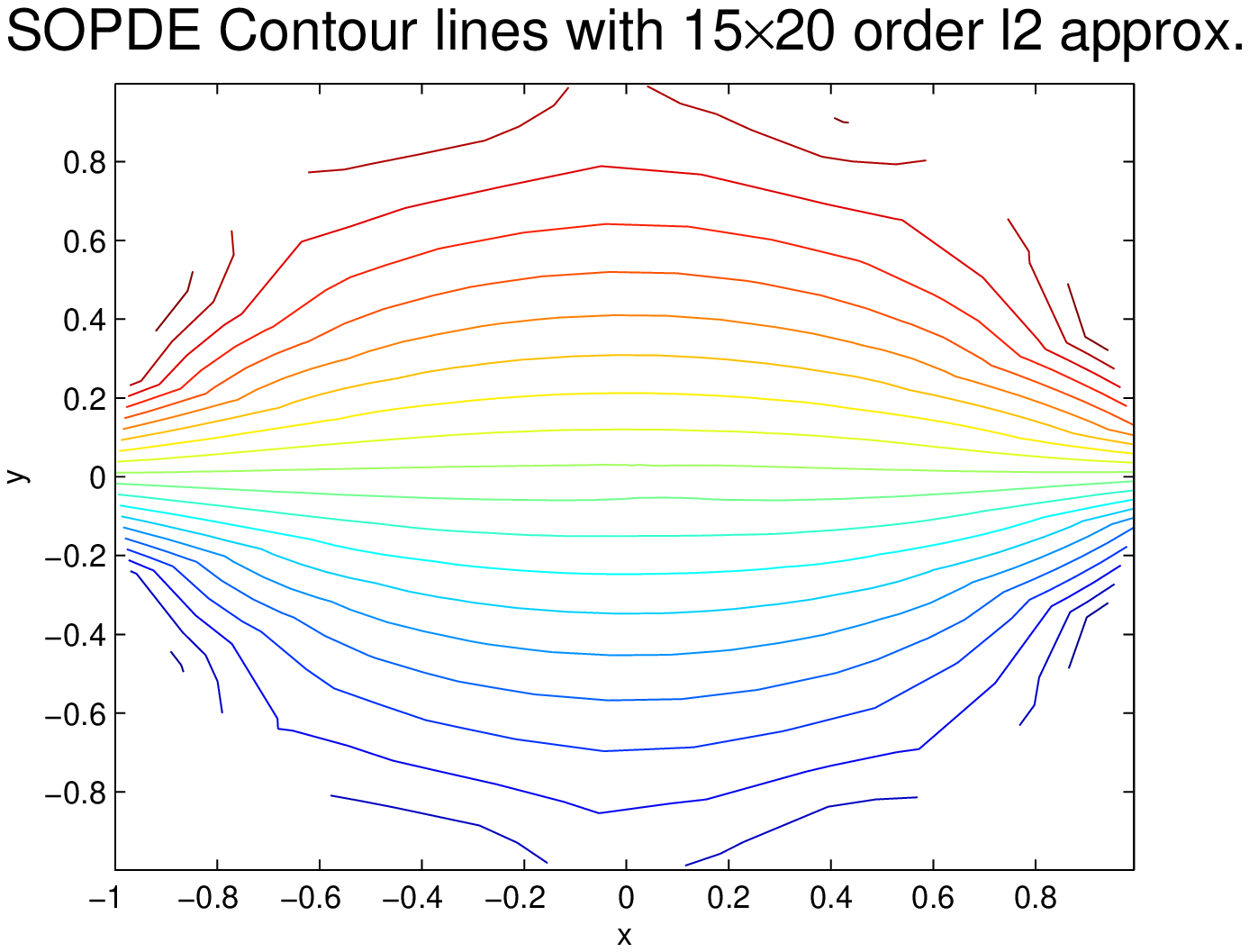}
\caption{SOPDE Solution: Solution surface and contour lines.}
\label{fig:SOfigure3}
\end{figure}
\begin{figure}
\centering
\includegraphics[width=.4\textwidth]{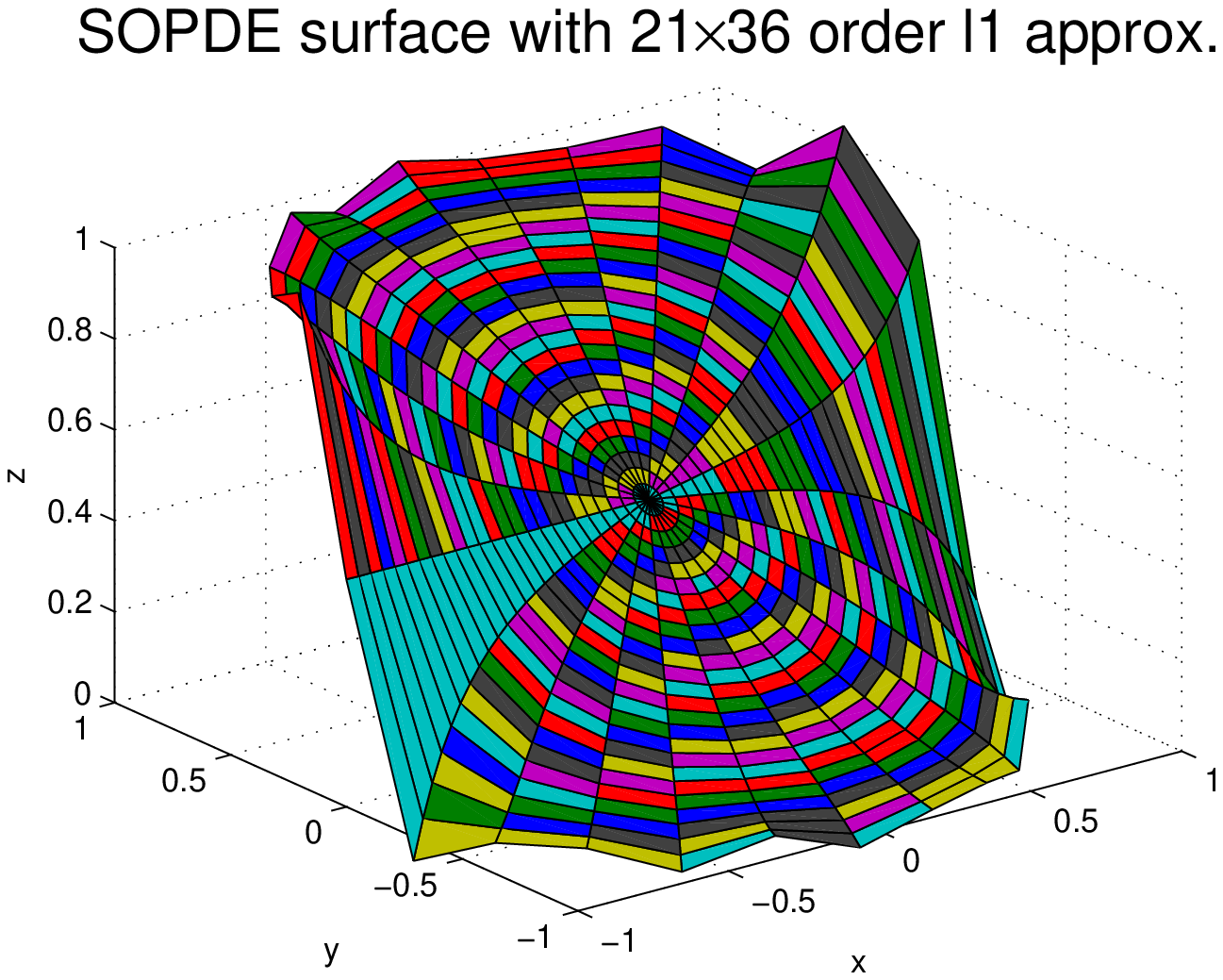}\hfill
\includegraphics[width=.4\textwidth]{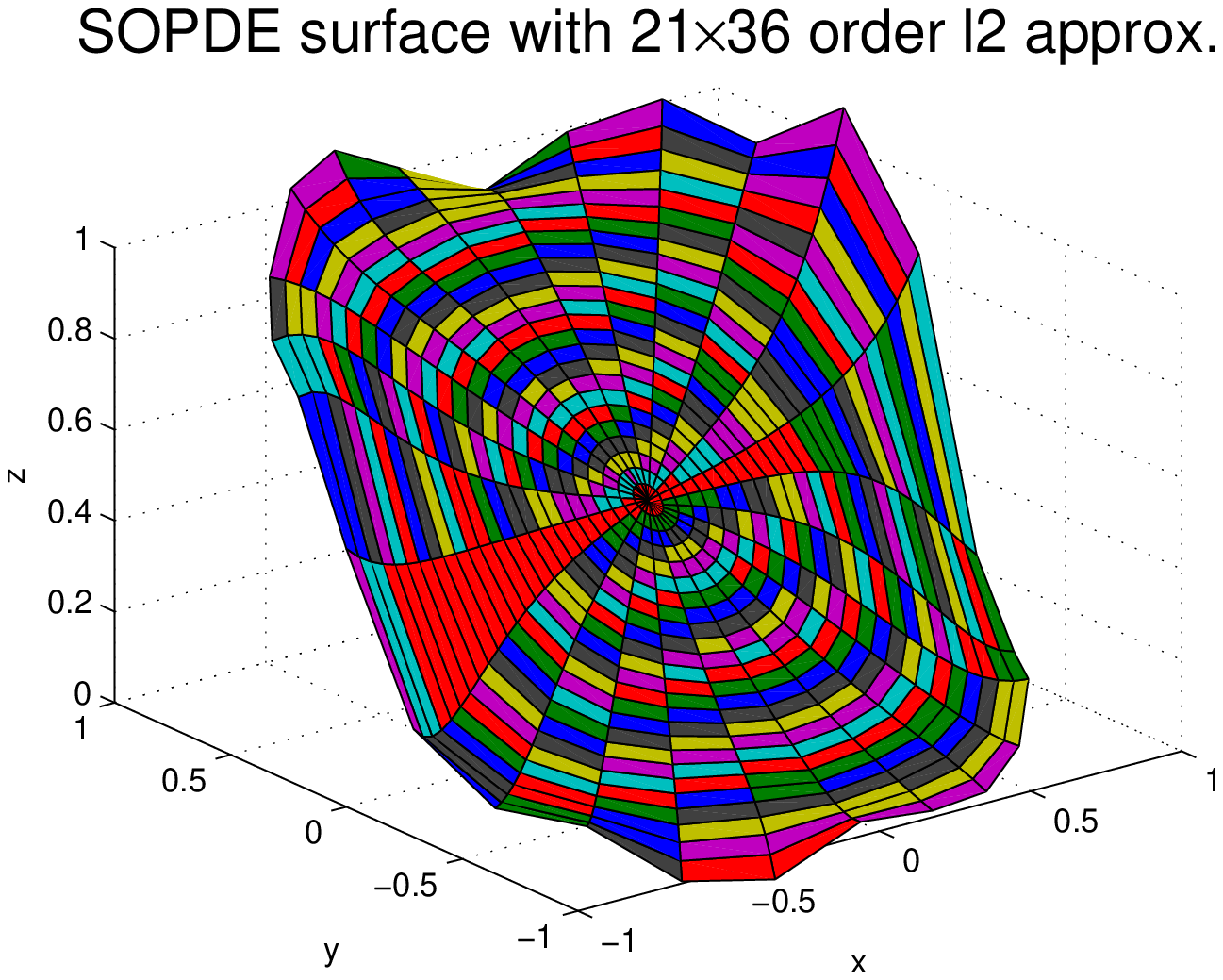}\\
\includegraphics[width=.4\textwidth]{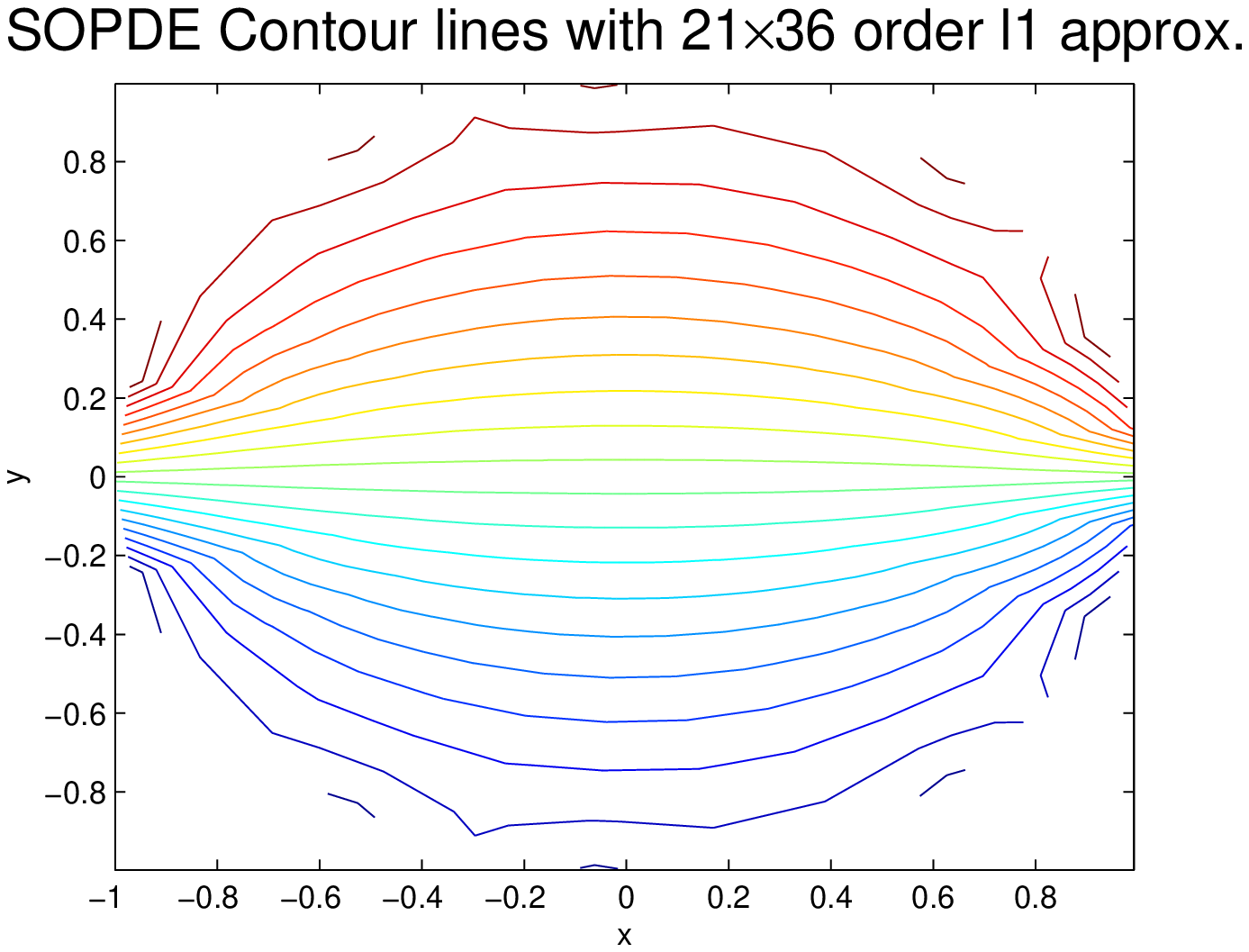}\hfill
\includegraphics[width=.4\textwidth]{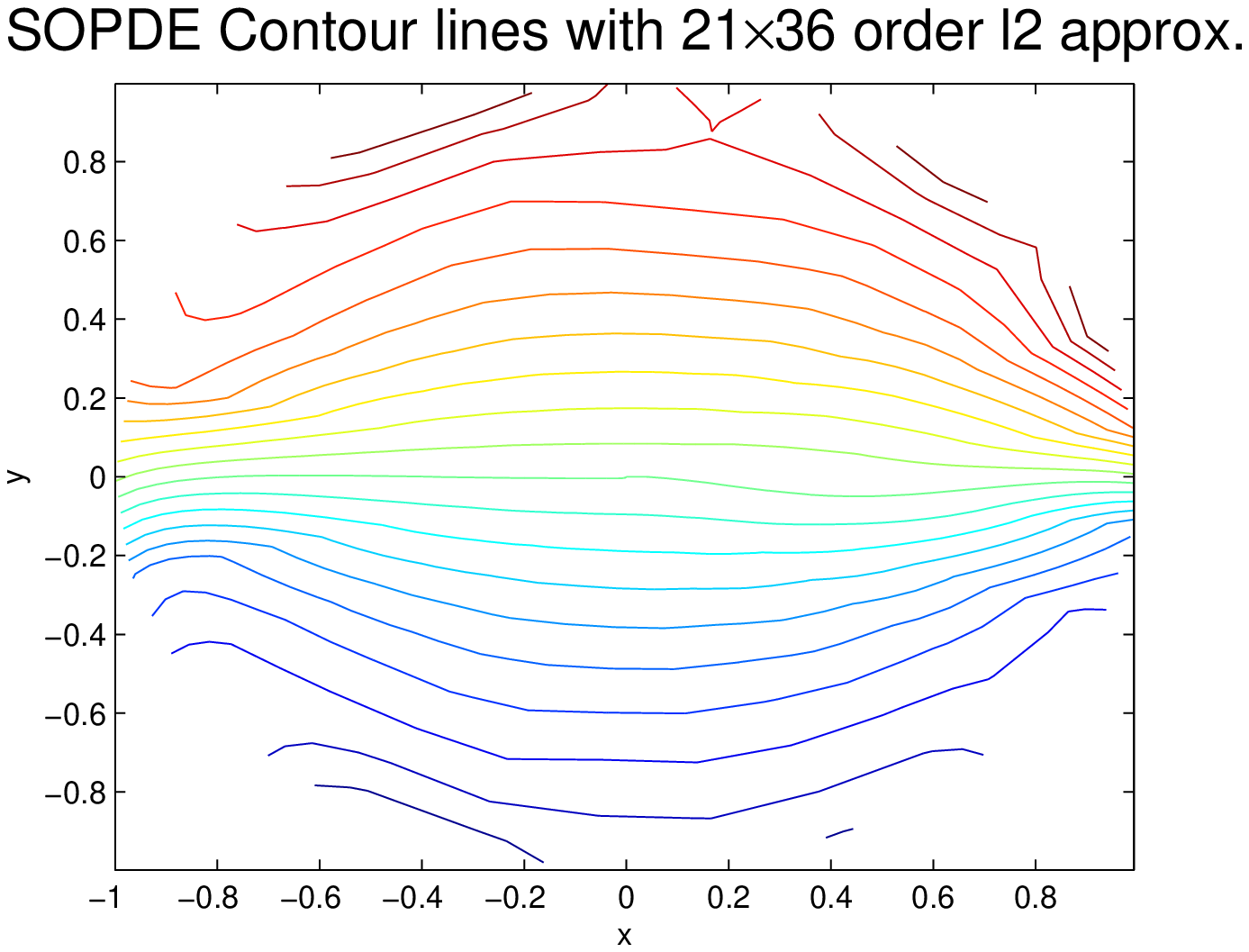}
\centering
\includegraphics[width=.4\textwidth]{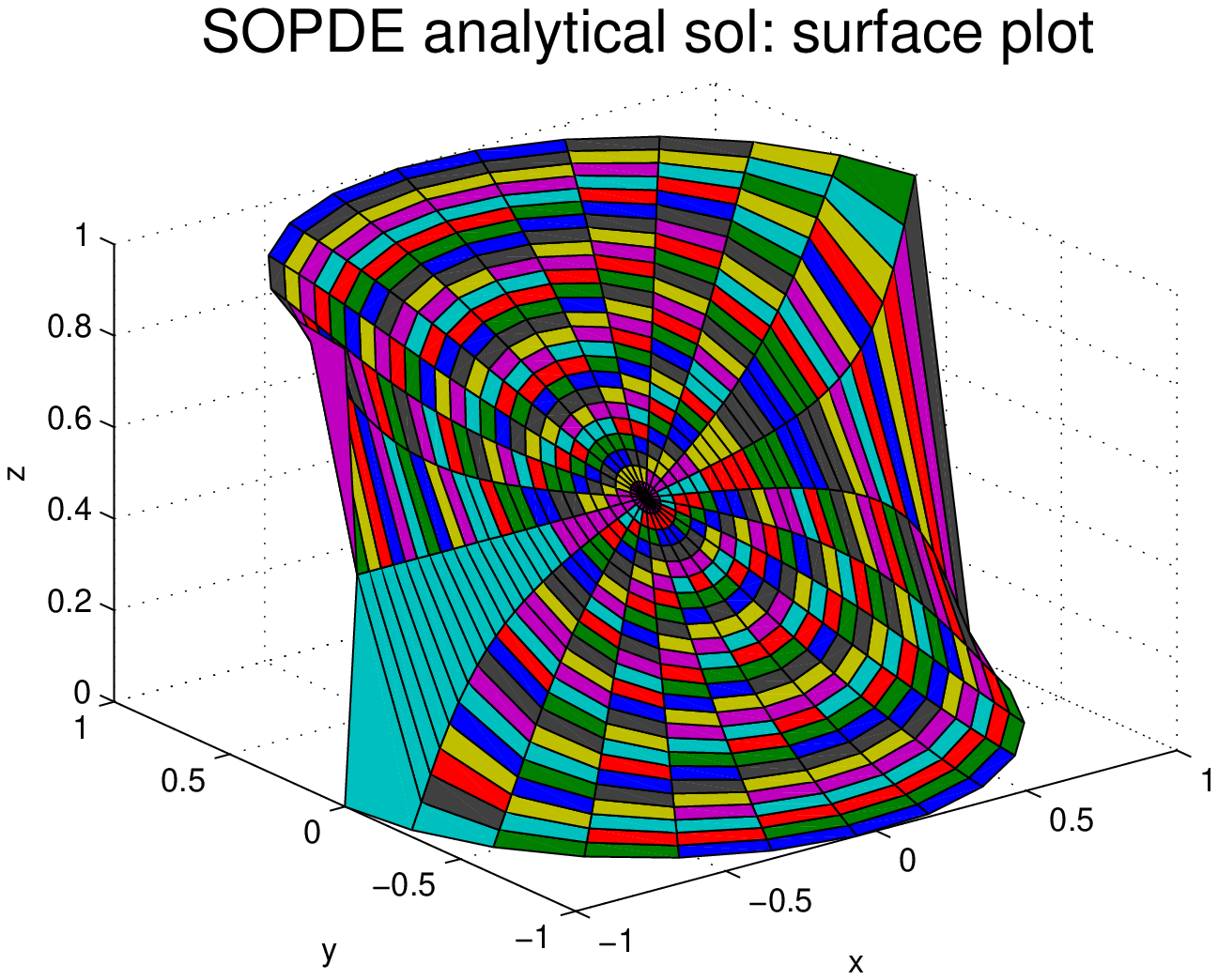}\hfill
\includegraphics[width=.4\textwidth]{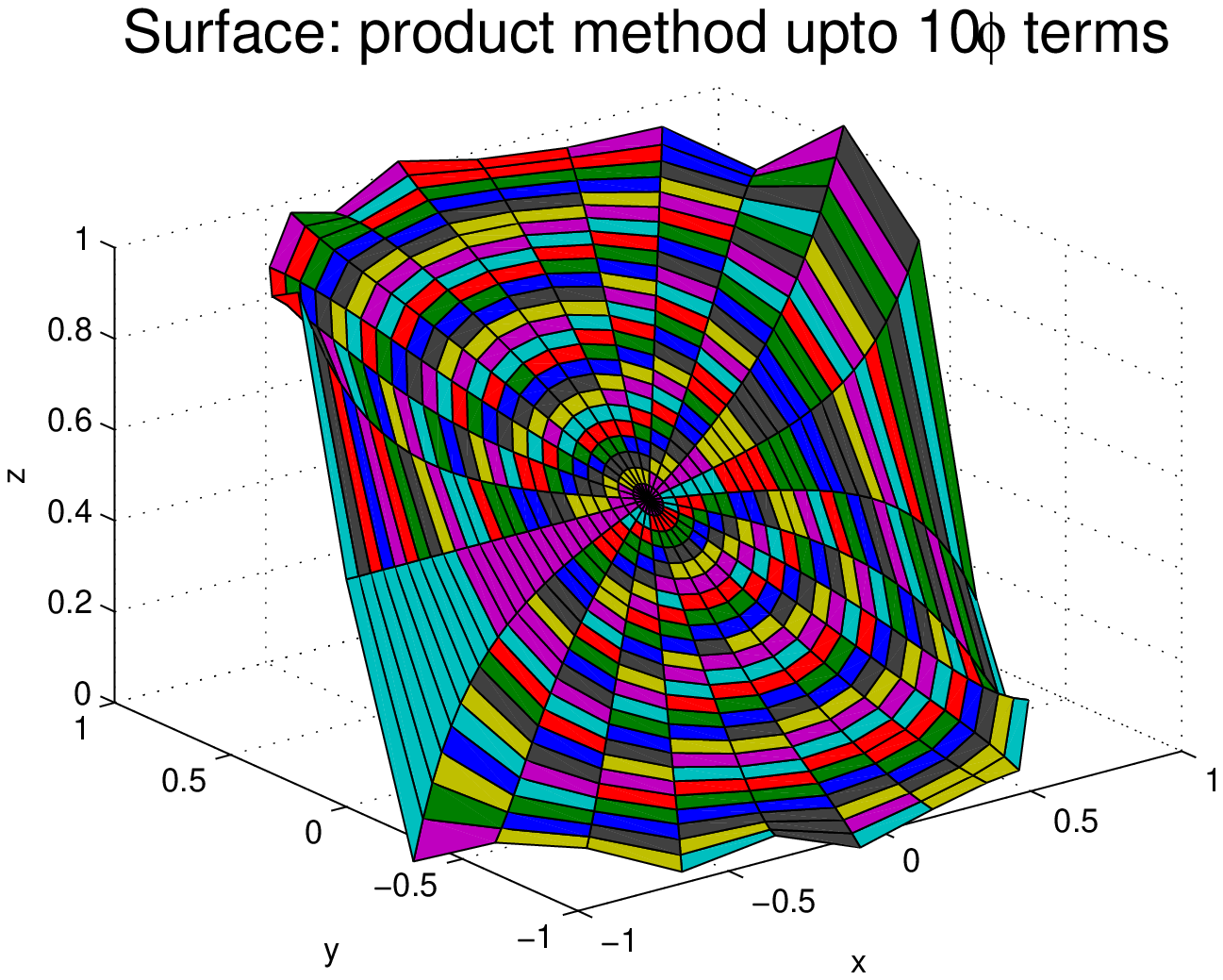}\\
\includegraphics[width=.4\textwidth]{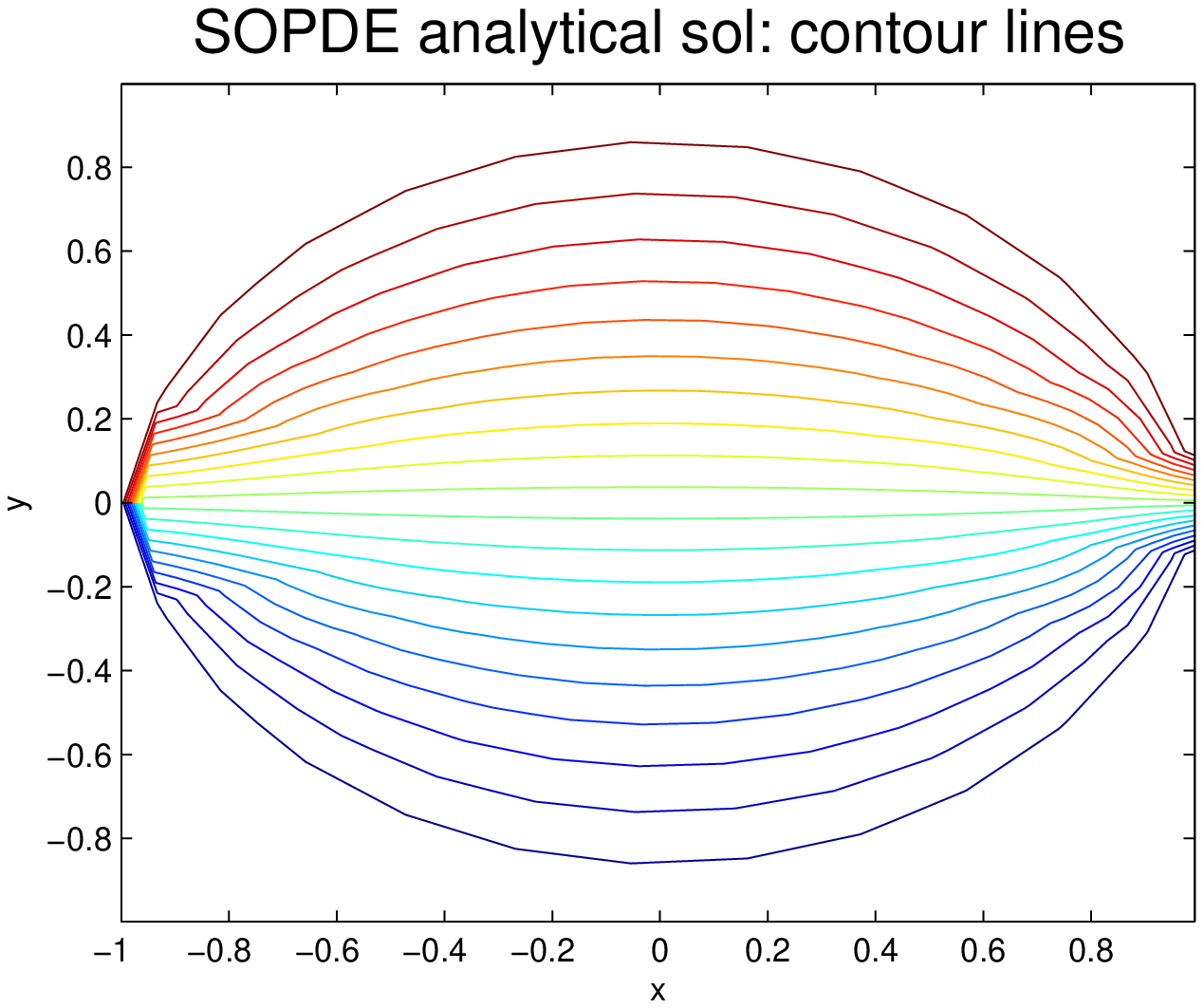}\hfill
\includegraphics[width=.4\textwidth]{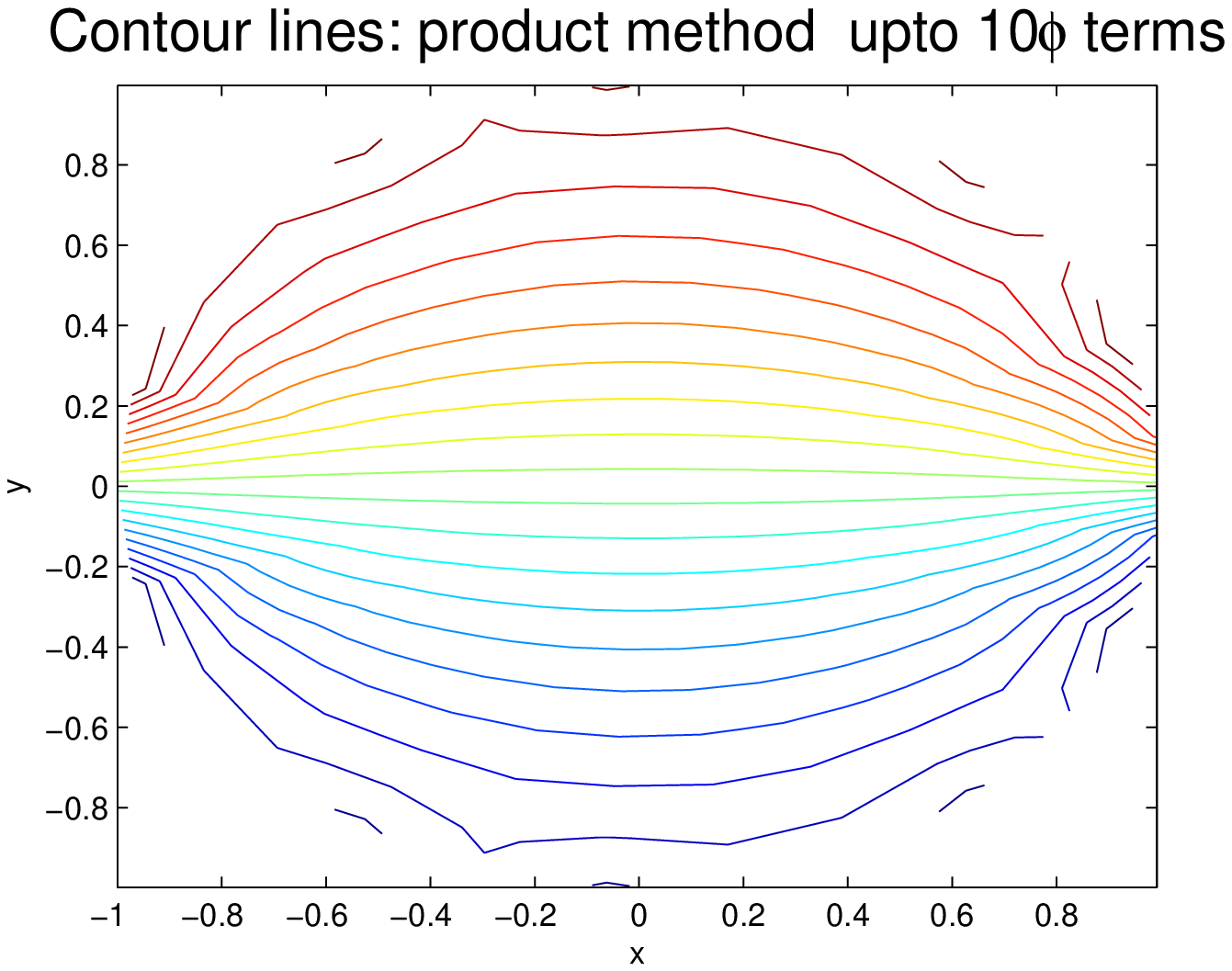}
\caption{SOPDE Solution: Solution surface and contour lines.}
\label{fig:SOfigure4}
\end{figure}
\begin{table}
%\caption{Errors in SOPDE in $10^{-4}$ scale}\label{SOPDE_error}
\caption{Errors in solving SOPDE}\label{SOPDE_error}
\vspace*{3mm}
\begin{center}
{\scriptsize
\begin{tabular}{|r|r|r|r|r|r|r|r|r|} \hline \label{tab:SOPSE_error}
Order of Zernike pol.& 7 $\times$ 6 & 9 $\times$ 9 & 11 $\times$ 12 & 13 $\times$ 16 & 15 $\times$ 20 & 17 $\times$ 25 & 19 $\times$ 30 & 21 $\times$ 36 \\
\hline
$\ell_2$-error (order $10^{-4}$) & 17.1369 &15.9970 &11.6488 & 7.3285 &4.5292 &7.4972& 5.0259&8.3604  \\
\hline
$\ell_1$-error (order $10^{-4}$) & 14.9311& 14.9311& 7.3285&7.3285 &4.4431& 4.4431&3.0503  & 3.0499
\\
\hline
\end{tabular}}
\end{center}
\end{table}
\paragraph{\textbf{Computational complexity}}
Recall that $n$ denotes the maximum degree of radial polynomials used. To compute the Zernike coefficients of the forcing function $f,$ the cost of computation is $O(n^3),$ see \cite{Boyd:2011}. In (\ref{eq:TensorRep}) and (\ref{SOPDE54}), the number of operations mainly comes from calculating the tensor product which is $O(n^4).$ However, it is important to note that the matrices obtained from the tensor products in (\ref{eq:TensorRep}) or after (\ref{SOPDE54}) have to be computed just once and thereafter the same matrices can be used to solve different problems with other boundary conditions. The linear system $A\mathbf{x} = \mathbf{b},$ where $A$ is a sparse matrix, can be efficiently solved using an $l_1$-minimization algorithm.
%
%%%%%%%%%%%%%%%%%%%%%%%%%%%%%%%%%%%%%%%%%%%%%%%%%%%%%%%%%%
\begin{figure}
\centering
\includegraphics[width=2.5in]{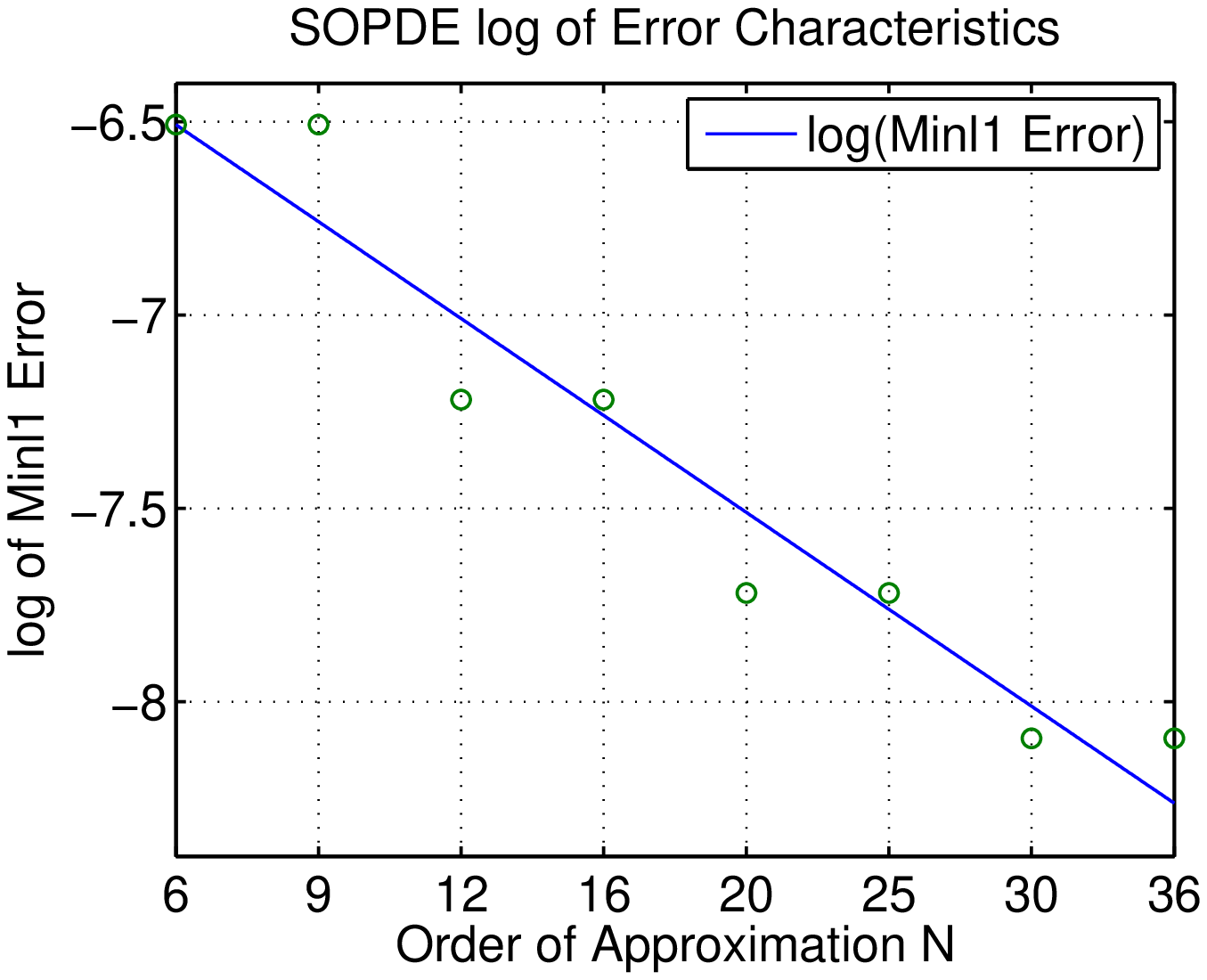}
\caption{SOPDE Solution: log of error curve.}
\label{fig:SOfigure5}
\end{figure}
\section{Rate of decay of Zernike coefficients}
\label{sec:Convergence}
%%%%%%%%%%%%%%%%%%%%%%%%%%%%%%%%%%%%%%%%%%%%%%%%%%%%%%%%%%
Recall from Section  \ref{sec:Intro} that the coefficients in a Zernike expansion as given in (\ref{RepresentationZernikePoly}) are % given by
	$$A_{nm} = \frac{\epsilon_m(n+1)}{\pi} \int_0^1 \int_0^{2\pi} f(r, \phi) \cos m \phi R_n^m(r) r d \phi dr,$$
	$$B_{nm} = \frac{\epsilon_m(n+1)}{\pi} \int_0^1 \int_0^{2\pi} f(r, \phi) \sin m \phi R_n^m(r) r d \phi dr.$$
%$\epsilon_m$ is the Neumann factor, see \cite{Prata:1989}:
%
%	$$\epsilon_m = \left\{
%	\begin{array}{cc}
%	1 & \textrm{if} \ m = 0, \\
%	2 & \textrm{otherwise.}
%	\end{array}
%	\right.
%	$$
The rate of decay of the $A_{mn}s$ and $B_{mn}$s will be calculated here for certain types of functions. These are given in
Theorem~\ref{THM:RateDecayCk} and Theorem~\ref{Thm:Decay} below.

Define the inner product on the unit disk as
	$$\langle f, g \rangle = \int_0^1 \int_0^{2\pi} f(r, \phi) \overline{g}(r, \phi) \ r \ dr \ d \phi.$$
% This gives rise to the following norm
The associated norm is
	$$\|f\| = \sqrt{\int_0^1 \int_0^{2\pi} |f(r, \phi)|^2 \ r \ dr \ d \phi } \ .$$
Define
	$$
	\left\{
	\begin{array}{c}
	^oU_n^m \\
	^eU_n^m
	\end{array}
	\right\}
	:=
	\left\{
	R_n^m(r)
	\begin{array}{c}
	\sin m \phi \\
	\cos m \phi
	\end{array}
	\right\}.
	$$
	As given in \cite{Nijboer:1942} and \cite{Born:1999}
	        $$\int_0^1 R_n^m(r) R_{n'}^m = \frac{\delta_{n n'}}{2n+2}$$
		%$$\int^1_0 R^m_n(r) R^m_{n'}(r)rdr = \frac{\delta_{mm'}}{2n+2}$$
	which means
		$$\|R_n^m\|^2 = \frac{1}{2n+2}.$$
	As discussed in  % cite{Nijboer:1942} and
\cite{Nijboer:1942, Born:1999},
	the set $
	\left\{
	\begin{array}{c}
	^oU_n^m \\
	^eU_n^m
	\end{array}
	\right\}$ forms
	% an orthogonal basis for the inside of the unit circle.
	an orthogonal basis for the space of square integrable functions in the unit disk. Since
	$\int_0^{2\pi} d\phi =2\pi,$ $\int_0^{2\pi} \sin^2 m\phi \ d\phi = \pi,$ and $\int_0^{2\pi} \cos^2 m\phi \ d\phi = \pi,$ one
	can normalize to get the following orthonormal basis for the space of square integrable functions in the unit disk
	$$
	\mathcal{B} : = \left\{ \frac{\sqrt{2n+2}}{\sqrt{\pi}} R_{n}^m (r) \begin{array}{c}
	\sin m \phi \\
	\cos m \phi
	\end{array}\right\}_{\substack{n = 0, 1, 2, \ldots \\ 1\leq m \leq n \\ n - m \ \textrm{even}}} \bigcup \left\{ \frac{\sqrt{2n+2}}{\sqrt{2 \pi}}R_n^0(r)\right\}_{n = 0, 2, 4, \ldots}
	$$
which can be rewritten more compactly as
	$$
	\mathcal{B} : = \frac{\sqrt{\epsilon_m (n+1)}}{\sqrt{\pi}}R_n^m(r) \left\{ \begin{array}{c}
	\sin m \phi \\
	\cos m \phi
	\end{array} \right\} =: \left\{
	\begin{array}{c}
	^o\underbar{U}_n^m \\
	^e\underbar{U}_n^m
	\end{array}
	\right\}_{\substack{n = 0, 1, 2, \ldots \\ 0 \leq m \leq n \\ n - m \ \textrm{even}}}.
	$$
% where
% 	$$
% 	\epsilon_m = \left\{
% 	\begin{array}{cc}
% 	1 & \textrm{if} \ m = 0 \\
% 	2 & \textrm{if} \ m \neq 0.
% 	\end{array}
% 	\right.
% 	$$
Any $f(r, \phi)$ defined on the unit disk can be expanded as
	\begin{equation} \label{EQ:fRepZernike}
	f(r, \phi) = \sum_{n=0}^{\infty} \sum_{\substack{0 \leq m \leq n \\ n - m \ \textrm{even}}}\left[
	\langle f(r,\phi), ^e\underbar{U}_n^m \rangle ^e\underbar{U}_n^m + \langle f(r,\phi), ^o\underbar{U}_n^m \rangle ^o\underbar{U}_n^m
	\right].
	\end{equation}
On comparing (\ref{EQ:fRepZernike}) with (\ref{RepresentationZernikePoly}) and (\ref{EQ:ZernikeCoeff}), note that
	$$
	A_{nm} = \langle f(r, \phi), ^e\underbar{U}_n^m \rangle, \ B_{nm} = \langle f(r, \phi), ^o\underbar{U}_n^m \rangle
	$$
are the coefficients of $f$ with respect to the Zernike polynomial basis. They will be referred to as the Zernike coefficients of $f.$
Define the $N$th partial sum as
	$$
	S_N f (r, \phi) := \sum_{n=0}^{N} \sum_{\substack{0 \leq m \leq n \\ n - m \ \textrm{even}}}\left[
	\langle f(r,\phi), ^e\underbar{U}_n^m \rangle ^e\underbar{U}_n^m + \langle f(r,\phi), ^o\underbar{U}_n^m \rangle ^o\underbar{U}_n^m
	\right].
	$$
Let $P_N$ be the space spanned by %Zernike circle polynomials
Zernike polynomials of radial degree at most $N,$ i.e.,
	%$$
%	P_N = \left\{ p(r, \phi) \ : \ p(r, \phi) = \sum_{n = 0}^N \sum_{0\leq m \leq n}\left[^ec_{mn} \ ^e\underbar{U}_n^m + ^oc_{mn} \ ^o\underbar{U}_n^m\right]\right\}
%	$$
    $$
	P_N = \left\{ p(r, \phi) \ : \ p(r, \phi) = \sum_{n = 0}^N \sum_{\substack{0 \leq m \leq n \\ n - m \ \textrm{even}}} \left[^ec_{mn} \ ^e\underbar{U}_n^m + ^oc_{mn} \ ^o\underbar{U}_n^m\right]\right\} .
	$$
Then $S_N f$ is a polynomial in $P_N.$ Since $\mathcal{B}$ forms an orthonormal basis, by property of orthonormal sets, $S_N f$ is %the polynomial closest to
the polynomial of best approximation to
$f$ among all polynomials in $P_N$ which means that for any polynomial $p \in P_N$ one has
	$$\|f - p\| \geq \|f - S_Nf\|$$
with equality holding if and only of $p = S_Nf.$ By virtue of $\mathcal{B}$ being an orthonormal basis, Parseval's Identity also holds
%	
	%\begin{equation} \label{Parseval_Id}
\begin{eqnarray*}
	\|f\|^2 = \sum_{n=0}^{\infty} \sum_{\substack{0 \leq m \leq n \\ n - m \ \textrm{even}}}\left( \left|\langle f(r,\phi), ^e\underbar{U}_n^m \rangle \right|^2 + \left|\langle f(r,\phi), ^o\underbar{U}_n^m \rangle \right|^2\right)
	\end{eqnarray*}
which in turn gives the Riemann-Lebesgue Lemma in this case:
	\begin{equation}\label{RiemannLebesgue}
%\begin{eqnarray*}
	\lim_{m,n \to \infty} \left| \langle f(r,\phi), ^e\underbar{U}_n^m \rangle \right| = \lim_{m,n \to \infty} \left| \langle f(r,\phi), ^o\underbar{U}_n^m \rangle \right|  = 0.
%	\end{eqnarray*}
    \end{equation}
In certain cases, the rate of decay of the coefficients in (\ref{RiemannLebesgue}) can be obtained as follows.

Denote by $\mathcal{C}^k(B(0,1))$ the space of functions defined on the unit disk $B(0,1)$ whose $k$th order partial derivatives all exist and are continuous on $B(0,1)$. For convenience, we shall sometimes write $C^k.$ Suppose that $u(r, \phi) \in \mathcal{C}^2.$
Then
\begin{eqnarray}
C_{nm} &=& \frac{\epsilon_m (n+1)}{\pi} \int_0^1 \int_0^{2\pi} u(r, \phi) e^{i m \phi} R^m_n(r)  r \ d\phi dr \nonumber \\
&=&  \frac{\epsilon_m (n+1)}{\pi} \int_0^1 \left[ \frac{u(r, \phi)e^{i m \phi}}{im} \Big|_{0}^{2\pi} - \frac{1}{im }\int_0^{2\pi} \frac{\partial u (r, \phi)}{\partial \phi}e^{i m \phi} d\phi \right] R^m_n(r)  r \  dr \nonumber \\
&=& - \frac{1}{im } \frac{\epsilon_m (n+1)}{\pi} \int_0^1 \int_0^{2\pi} \frac{\partial u (r, \phi)}{\partial \phi}e^{i m \phi} R^m_n(r)  r  \   d\phi dr \quad (\textrm{since $u(r, 2\pi) = u(r, 0)$}) \nonumber \\
&=& - \frac{1}{im } \frac{\epsilon_m (n+1)}{\pi} \left(- \frac{1}{im }\right) \int_0^1 \int_0^{2\pi} \frac{\partial^2 u (r, \phi)}{\partial \phi^2}  e^{i m \phi} R^m_n(r)  r  \   d\phi dr \label{EQ:RateDecayC2}
\end{eqnarray}
where in the last step we have used the fact that $\frac{\partial u }{\partial \phi}(r, 0) = \frac{\partial u}{\partial \phi}(r,2\pi)$. The integral on the right side of (\ref{EQ:RateDecayC2}) will yield the Zernike coefficients of $\frac{\partial^2 u (r, \phi)}{\partial \phi^2}.$ Denoting this by $C^{''}_{mn},$ one can write
$$
C_{mn} = \frac{\epsilon_m (n+1)}{\pi} \left(- \frac{1}{im }\right)^2 C^{''}_{mn}
$$
which decays at the rate of $m^{-2}.$ Note that $m$ and $n$ tend to infinity at the same rate. In general, one has the following Theorem~\ref{THM:RateDecayCk}.
\begin{theorem} \label{THM:RateDecayCk}
Let $u(r, \phi) \in \mathcal{C}^k.$ Then the Zernike coefficients $A_{nm},$ $B_{nm}$  of $u$ decay at the rate of $m^{-k}.$
\end{theorem}
\begin{definition}
A function $u : U \to \mathbb{R}$ is said to be H\"{o}lder continuous of order $\lambda$ if for all $x, y \in U$
	$$|u(x) - u(y)| \leq C \|x - y \|^{\lambda}$$
for some constants $\lambda$ and $C,$ where $\| . \|$ is the metric on $U.$
\end{definition}
\begin{theorem} \label{Thm:Decay}
Let $u(r, \phi)$ be H\"{o}lder continuous of order $\lambda \geq 1.$ Then the Zernike coefficients $A_{nm},$ $B_{nm}$ of $u$ decay at least like $m^{-\lambda +1}.$
\end{theorem}
\begin{proof}
Consider
	\begin{equation} \label{eq:ZernikeCoeff}
	C_{nm} := \frac{\epsilon_m(n+1)}{\pi} \int_0^1 \int_0^{2\pi} u(r, \phi) e^{im \phi}  R_n^m(r)(r) r \ d \phi \ dr.
	\end{equation}
Rewriting (\ref{eq:ZernikeCoeff}) gives
	\begin{eqnarray}
	C_{nm} &=& -\frac{\epsilon_m(n+1)}{\pi} \int_0^1 \int_0^{2\pi} u(r, \phi) e^{im \phi} e^{i \pi} R_n^m(r) r \ d \phi \ dr 																				 \nonumber \\
		&=& -\frac{\epsilon_m(n+1)}{\pi} \int_0^1 \int_0^{2\pi} u(r, \phi) e^{im (\phi+ \pi/m)} R_n^m(r) r \ d \phi \ dr
																			\nonumber \\
		&=& -\frac{\epsilon_m(n+1)}{\pi} \int_0^1 \int_{\pi/m}^{2\pi + \pi/m} u(r, \alpha - \pi/m) e^{im \alpha} R_n^m(r) r \ d 																		 \alpha \ dr \nonumber \\
		&=& -\frac{\epsilon_m(n+1)}{\pi} \int_0^1 \int_0^{2\pi} u(r, \alpha - \pi/m) e^{im \alpha} R_n^m(r) r
		\ d \alpha \ dr. \label{eq:ZernikeCoeffsNew}
	\end{eqnarray}
Adding (\ref{eq:ZernikeCoeff}) and (\ref{eq:ZernikeCoeffsNew}) gives
	\begin{eqnarray*}
	C_{nm} &=& \frac{\epsilon_m(n+1)}{2\pi} \int_0^1 \int_0^{2\pi} \left[ u(r, \phi) - u(r, \phi - \pi/m)\right] e^{im \phi}R_n^m(r) r \ \ d \phi dr \\
	\textrm{or,} \ |C_{nm}| &\leq& \frac{\epsilon_m(n+1)}{2\pi} \int_0^1 \int_0^{2\pi} \left| u(r, \phi) - u(r, \phi - \pi/m)\right|
	|R_n^m(r) r| \ d \phi \ dr \\
	&\leq & C \frac{\epsilon_m(n+1)}{2\pi} \left(\frac{\pi}{m}\right)^{\lambda} \int_0^1 \int_0^{2\pi} |R_n^m(r) r| \ d \phi \ dr \\
	&=& C \epsilon_m(n+1) \left(\frac{\pi}{m}\right)^{\lambda} \int_0^1 |R_n^m(r) r| \ dr \\
	&\leq& C \epsilon_m(n+1) \left(\frac{\pi}{m}\right)^{\lambda},
	\end{eqnarray*}
where the last inequality follows from the fact that $|R_n^m(r) r|  \leq 1,$ since $|R_n^m(r)|  \leq 1$ and $0 \le r \le 1,$ see \cite{Prata:1989}.
Looking at the real and imaginary parts of $C_{nm}$ gives the desired result. Note that $m$ and $n$ tend to infinity at the same rate.
\end{proof}
%
%%%%%%%%%%%%%%%%%%%%%%%%%%%%%%%%%%%%%%%%%%%%%%%%%%%%%%%%
\section{Appendix}
\label{Appendix}
%%%%%%%%%%%%%%%%%%%%%%%%%%%%%%%%%%%%%%%%%%%%%%%%%%%%%%%%
To derive the IOM for the %Zernike radial polynomials
radial parts of Zernike  polynomials, the recurrence relation stated in (\ref{eq:ZernikePolyRecurrence}) will be used. For convenience, this is again provided below.
%*
\begin{equation} \label{eq:ZernikePolyRecurrence_1}
\int_{r_0}^r \left[R_n^m(r) + R_{n}^{m+2}(r)\right] \ dr = \frac{1}{n+1} \left[R_{n+1}^{m+1}(r) - R_{n-1}^{m+1}(r)\right]\Big|_{r_0}^r .
\end{equation}
If all radial polynomials up to degree $n$ are used then the basis vector for the %Zernike radial polynomials
radial parts of Zernike  polynomials is
\begin{equation}\label{RadialBasis}
R(r) = [R_0^0(r), R_1^1(r), R_2^0(r), R_2^2(r), R_3^1 (r), R_3^3(r), \ldots,  R_n^m(r)], \ n \in \mathbb{N} \cup \{0\}, \ 0 \leq n-m, \ n - m \ \textrm{even} .
\end{equation}
Let $i$ be the degree of a radial polynomial, and let $p_i$ denote the total number of polynomials of degree $i.$  Due to the special structure of the radial polynomials as described in (\ref{Sol:Hypergeometric}) of Section~\ref{sec:Intro}, the value of $p_i$ is determined as follows. For a given non-negative integer $i,$ the value of $p_i$ is the number of integers $j$ for which $i - j$ is even and non negative.
For $i = 0, 1, \ldots, n,$ let $\Delta_{i+1}$ be the $p_i \times p_i$ matrix with ones along the diagonal and also above the main diagonal, and zeros elsewhere. That is,
	$$
	\Delta_{i+1} =
	\left[
	\arraycolsep=2.4pt\def\arraystretch{.75}
	\begin{array}{ccccc}
	1 & 1 & 0 & \cdots & 0 \\
	0 & 1 & 1 &\cdots & 0 \\
	\vdots & \vdots & \ddots & \ddots & \vdots \\
	0 & 0 & \cdots & 1 & 1 \\
	0 & 0 & \cdots & 0 & 1
	\end{array}
	\right]_{p_i \times p_i}
	$$
whose inverse is
	$$
	\Delta_{i+1}^{-1} =
	\left[
	\arraycolsep=2.4pt\def\arraystretch{.75}
	\begin{array}{ccccc}
	1 & -1 & & \cdots & (-1)^{p_i-1} \\
	0 & 1 & -1 &\cdots & (-1)^{p_i-2}  \\
	\vdots & \vdots & \ddots & \ddots & \vdots \\
	0 & 0 & \cdots & 1 & -1 \\
	0 & 0 & \cdots & 0 & 1
	\end{array}
	\right]_{p_i \times p_i}
	$$
Let $E_{r1}$ be a block diagonal matrix of the form
	$$
	E_{r1}
	=
	\left[
    	\arraycolsep=2.4pt\def\arraystretch{.75}
    	\begin{array}{cccc}
    	\Delta_1 &  & & \\
    	& \Delta_2 &  &   \\
     	& & \ddots &  \\
     	&  &  & \Delta_{n+1} \\
   	 \end{array}
    	\right] .
	$$
Denote $\int_{r_0}^r R(r) \ dr$ by $E_{r r _0}.$ Then from (\ref{eq:ZernikePolyRecurrence_1}) one can write
	$$
	E_{r1} E_{rr_0} = E_{r2}.
	$$
The structure of the matrix $E_{r1}$ has been described above. The matrix $E_{r2}$ is a block matrix of order $n+1,$ $n$ being the degree of $R_n^m(r).$ Each block in  $E_{r2}$ is a submatrix $\Gamma_{k l}$ and can be represented as
	$$
	E_{r2} = \left[ \Gamma_{kl} \right]
	$$
	where $\Gamma_{i+1, j+1}$ is of size $p_i \times p_j.$
For example, let $n = 5.$ Then $E_{r2}$ is a $6 \times 6$ block matrix where all  the blocks are null matrices except for the following:
$$
\Gamma_{11} = [-R_1^1(r_0)] , \ \Gamma_{12} = [1], \ \Gamma_{21} = [-\frac12 R_2^2(r_0)],
\
\Gamma_{23} =
\begin{bmatrix}
0 & \frac12
\end{bmatrix},
\Gamma_{31} =
 \begin{bmatrix}
-\frac13[R_3^1(r_0) - R_1^1(r_0)] \\
-\frac13 R_3^3(r_0)
\end{bmatrix} , \
\Gamma_{32} =
\begin{bmatrix}
-\frac13 \\
0
\end{bmatrix},
$$
$$
\Gamma_{34} =
\begin{bmatrix}
\frac13 & 0 \\
0 & \frac13
\end{bmatrix},
\Gamma_{41} =
 \begin{bmatrix}
-\frac14[R_4^2(r_0) - R_2^2(r_0)] \\
-\frac14 R_4^4(r_0)
\end{bmatrix},
\Gamma_{43} =
\begin{bmatrix}
0 & -\frac14 \\
0 & 0
\end{bmatrix},
\Gamma_{45}
=
\begin{bmatrix}
0 & \frac14 & 0 \\
0 & 0 & \frac14
\end{bmatrix},
$$
$$
\Gamma_{51} =
\begin{bmatrix}
-\frac15 [R_5^1(r_0) - R_3^1(r_0)] \\
-\frac15 [R_5^3(r_0) - R_3^3(r_0)] \\
-\frac15 R_5^5(r_0)
\end{bmatrix},
\Gamma_{54} =
\begin{bmatrix}
-\frac15 & 0 \\
0 & -\frac15 \\
0 & 0
\end{bmatrix},
\Gamma_{56}
=
\begin{bmatrix}
\frac15 & 0 & 0 \\
0 & \frac15 & 0 \\
0 & 0 & \frac15
\end{bmatrix},
$$
$$
\Gamma_{61} =
\begin{bmatrix}
-\frac16 [R_6^2(r_0) - R_4^2(r_0)] \\
-\frac16 [R_6^4(r_0) - R_4^4(r_0)] \\
-\frac16 R_6^6(r_0)
\end{bmatrix},
\Gamma_{65}
=
\begin{bmatrix}
0 & - \frac16 & 0\\
0 & 0 & - \frac16\\
0 & 0 & 0
\end{bmatrix}.
$$
For $0 \leq m \leq n,$ if one wants to form the integration operational matrix of  $\int_{r_0}^r R^m_n (r) \ dr$ from (\ref{eq:ZernikePolyRecurrence_1}) using only radial polynomials up to degree $n$, then usually the terms of higher degree are neglected. When $n$ is odd, then
for the integrals
\\
$\int_{r_0}^r R_n^n(r) \ dr,\quad $ $\int_{r_0}^r R_n^{n-2}(r) \ dr,\quad$ $\ldots,$ $\quad\int_{r_0}^r R_n^1(r) \ dr,$
\\
the terms
\\
$\frac{1}{n+1}R_{n+1}^{n+1}(r),\quad$ $\frac{1}{n+1}[R_{n+1}^{n-1}(r) - R_{n+1}^{n+1}(r)],$ $\quad\ldots,\quad$ $\frac{1}{n+1}[R_{n+1}^{2}(r) - R_{n+1}^{4}(r) + \cdots + (-1)^{\frac{n-1}{2}} R_{n+1}^{n+1}]$
\\
are neglected, respectively. On the other hand, when $n$ is even, then for the integrals
\\
$\int_{r_0}^r R_n^n(r) \ dr,\quad$ $\int_{r_0}^r R_n^{n-2}(r) \ dr,$ $\quad \ldots,\quad$ $\int_{r_0}^r R_n^0(r) \ dr,$
\\
the terms,
\\
$\frac{1}{n+1}R_{n+1}^{n+1}(r),\quad$ $\frac{1}{n+1}[R_{n+1}^{n-1}(r) - R_{n+1}^{n+1}(r)],$ $\quad \ldots,\quad $ $\frac{1}{n+1}[R_{n+1}^{0}(r) - R_{n+1}^{2}(r) + \cdots + (-1)^{\frac{n}{2}} R_{n+1}^{n+1}]$
\\
are neglected, respectively. In compact form, when one wishes to use only radial polynomials up to degree $n$ to evaluate
	$$
	\int_{r_0}^{r} R_n^{n - 2i} (r) \ dr,
	$$
then one neglects
	$$
	\frac{1}{n+1}\left[R_{n+1}^{n - 2i + 1} - R_{n+1}^{n - 2i + 3} + \cdots + (-1)^{i}R_{n+1}^{n+1}(r)\right],d
	$$
where
	$$
	i = \left\{
	\begin{array}{ll}
	0, 1, \ldots, \frac{n-1}{2} & \textrm{when $n$ is odd,} \\
	0, 1, \ldots, \frac{n}{2} & \textrm{when $n$ is even.}
	\end{array}
	\right.
	$$
For better accuracy of results, these neglected terms of degree greater than $n$ can then be represented in terms of the radial polynomials appearing in $R(r)$ of (\ref{RadialBasis}) as mentioned in Remark~\ref{Rem:Approximation}.
%%%%%%%%%%%%%%%%%%%%
\section{Conclusion}
\label{Conclusion}
%%%%%%%%%%%%%%%%%%%%
It is established in this paper that numerical solutions of partial differential equations in circular regions can be successfully done  using %Zernike circle polynomials
Zernike polynomials and IOMs which can otherwise be challenging using other orthogonal polynomials. In comparison, using multidimensional block pulse functions and  OSOMRI (one shot operational matrix for repeated integrations), with extensive computations, it was found earlier in \cite{Rao:1983} that solutions of second order PDEs do not promise numerical stability in all cases.
In solving the second order PDE by Zernike polynomials of a particular order  $(m,n),$  if the terms of order higher than $n$ are neglected in deriving operational matrices, the obtained solution is far from the actual one. By including these higher order terms as projections on the space generated by the lower order terms (see (\ref{ApproxLagInterpolation})), the solution of the second order PDE  is comparable with the true solution, and accuracy in the first order case is vastly improved.

In solving PDEs using IOM and block pulse functions, simple recursive methods could be developed in \cite{Rao:1983} due to the disjoint nature of the block pulse functions. Unfortunately, this cannot be done with other orthogonal polynomials extensively used in \cite{Datta:1995}  including the ones used here.
The integration operational matrices for double integration, $E_{Drr_0}$ and $E_{D\phi \phi_0},$ are computed as $E_{rr_0}^2$ and $E_{\phi \phi_0}^2$ respectively, however, this leads to an accumulation of errors at each stage of the integration  process. This can be improved by developing the IOM in the final stage known as OSOMRI as developed in \cite{Rao:1983}. However, the effort does not seem worthwhile because the improvement occurs in the higher order of decimal places and in many practical cases the solution without OSOMRI may serve the purpose. Another thing to note is that the second order PDEs solved here have boundary conditions that are discontinuous at two points on the circle. Due to these discontinuities, the solution shows oscillations known as Gibbs-Wilbraham phenomenon as is evident from the discontinuity of the contour lines of IOM solutions in contrast with those of the exact analytical solution in Figure~\ref{fig:SOfigure4}. In Example~\ref{EX:SOPDE}, a Laplace equation is solved \textit{with discontinuous Dirichlet and Neumann BCs, and as these discontinuous functions cannot be defined at some of the Chebyshev or Gauss-Lobatto points, the much acclaimed pseudo-spectral methods are not directly applicable to such problems.} For the purpose of demonstration of our method, examples selected are simple in nature.

 Our prime objective is to highlight how %Zernike circle polynomials
 Zernike polynomials can be directly applied to solve PDEs with discontinuous boundary conditions. There are other methods to numerically solve PDEs as outlined in Section~\ref{sec:Intro} and depending on the context and situation of the physical problems one may select an appropriate method for which the proposed approach is offered here as a potential candidate.

% In our method we have to solve a sparse system $A\mathbf{x}= \mathbf{b}$, in which the majority of the components of $\mathbf{x}$ are zeros, only few are non-zeros to be called sparse points. The reason for using $l_1$ norm to find a sparse solution is due to its special shape. It has spikes that happen to be at sparse points. Using it to touch the solution surface will very likely to find a touch point on a spike tip and thus a sparse solution gives more accurate surface than $l_2$ solution via QR algorithm which is also confirmed in our experimentation.

An extremely important problem for future investigation is the parameter estimation of distributed parameter systems that is a challenging research area for control system engineers. In this regard, another promising area of future research is to use Zernike polynomials in rectangular coordinates to solve PDEs with rectangular boundaries  and conversely to estimate the parameters in such regions if the input and the response are known.
\section*{Acknowledgment}
%Insert the Acknowledgment text here.
The authors are immensely grateful to the two anonymous referees for their insightful comments and valuable suggestions that vastly improved the quality of this article.

\end{document}